\documentclass[12pt,notitlepage]{amsart}
\usepackage{latexsym,amsfonts,amssymb,amsmath,amsthm}
\usepackage{graphicx}
\usepackage{color}
\usepackage{mathrsfs}
\usepackage{amsmath, amsthm, amsfonts, amssymb, cite}
\usepackage{wrapfig}
\usepackage{stackrel}
\usepackage{color}
\usepackage{float}
\usepackage{enumitem}
\usepackage{mathtools}
\usepackage{multicol}
\usepackage[normalem]{ulem}
\usepackage{xcolor}
\usepackage{hyperref}
\usepackage{cancel}

\usepackage[inner=1.0in,outer=1.0in,bottom=1.0in, top=1.0in]{geometry}

\newcommand{\R}{\ensuremath{\mathbb{R}}}
\newcommand{\F}{\ensuremath{\mathbb{F}}}
\newcommand{\N}{\ensuremath{\mathbb{N}}}
\newcommand{\mz}{\ensuremath{\mathbb Z}}
\newcommand{\Z}{\ensuremath{\mathbb Z}}
\newcommand{\mn}{\ensuremath{\mathbb N}}
\newcommand{\mr}{\ensuremath{\mathbb R}}

\newcommand{\mq}{\ensuremath{\mathbb Q}}
\newcommand{\Q}{\ensuremath{\mathbb Q}}
\newcommand{\A}{\ensuremath{\mathbb A}}
\newcommand{\mc}{\ensuremath{\mathbb C}}
\newcommand{\C}{\ensuremath{\mathbb C}}
\newcommand{\mf}{\ensuremath{\mathbb F}}
\newcommand{\fa}{\ensuremath{\mathfrak a}}
\newcommand{\fb}{\ensuremath{\mathfrak b}}

\newcommand{\cH}{\mathcal{H}}
\newcommand{\cW}{\mathcal{W}}

\newcommand{\cF}{\mathcal{F}}

\newcommand{\TC}{\mathcal{T}}

\newcommand{\mymod}{\ensuremath{\negthickspace \negmedspace \pmod}}
\newcommand{\shortmod}{\ensuremath{\negthickspace \negthickspace \negthickspace \pmod}}

\newcommand{\intR}{\int_{-\infty}^{\infty}}

\newcommand{\sumstar}{\sideset{}{^*}\sum}
\newcommand{\sumprime}{\sideset{}{'}\sum}

\DeclareMathOperator{\sgn}{sgn}

\DeclareMathOperator{\cond}{\mathfrak{c}}
\renewcommand{\P}{\mathbb{P}}
\DeclareMathOperator{\Stab}{Stab}
\DeclareMathOperator{\GL}{GL}
\DeclareMathOperator{\SL}{SL}
\DeclareMathOperator{\SO}{SO}
\DeclareMathOperator{\PGL}{PGL}
\DeclareMathOperator{\Vol}{Vol}

\newcommand{\real}{\mathop{\rm Re}}

\newcommand{\lcm}{\mathop{\rm lcm}}

\renewcommand{\bf}[1]{\mathbf{#1}}

\newcommand{\eps}{\varepsilon}

\newcommand{\psum}{\mathop{\sum\nolimits^+}}

\newcommand{\supp}{\mathop{\rm supp}}

\newcommand{\flrt}{{\rm flrt}}

\newcommand{\addchar}{\theta}

\theoremstyle{plain}		
	\newtheorem{mytheo}{Theorem} [section]
	
	\newtheorem{myprop}[mytheo]{Proposition}
	\newtheorem{mycoro}[mytheo]{Corollary}
     \newtheorem{mylemma}[mytheo]{Lemma}
	\newtheorem{mydefi}[mytheo]{Definition}
	
	\newtheorem{myremark}[mytheo]{Remark}

\theoremstyle{remark}

	\newtheorem*{myexam}{Example}

\numberwithin{equation}{section}
\numberwithin{figure}{section}

\begin{document}
\title[Fourth moment along cosets and the Weyl bound]{The fourth moment of Dirichlet $L$-functions along a coset and the Weyl bound}
\author{Ian Petrow} 
\email{ian.petrow@math.ethz.ch}
\address{ ETH Z\"urich\\
Department of Mathematics\\ 
R\"amistrasse 101\\
8092 Z\"urich\\
Switzerland}

\author{Matthew P. Young} 
\email{myoung@math.tamu.edu}
\address{Department of Mathematics\\
	  Texas A\&M University\\
	  College Station\\
	  TX 77843-3368\\
		U.S.A.}


 \thanks{The first author was supported by Swiss National Science Foundation grant PZ00P2\_168164.  This material is based upon work supported by the National Science Foundation under agreement No.\ DMS-170222 (M.Y.).  Any opinions, findings and conclusions or recommendations expressed in this material are those of the authors and do not necessarily reflect the views of the National Science Foundation.
 }

 \begin{abstract}
 We prove a Lindel\"of-on-average upper bound for the fourth moment of Dirichlet $L$-functions of conductor $q$ along a coset of the subgroup of characters modulo $d$ when $q^*|d$, where $q^*$ is the least positive integer such that $q^2|(q^*)^3$. As a consequence, we finish the previous work of the authors and establish a Weyl-strength subconvex bound for all Dirichlet $L$-functions with no restrictions on the conductor. 
 \end{abstract}

 \maketitle
 
\section{Introduction} 
\subsection{The Weyl bound and cubic moments}

This paper continues the previous work of the authors \cite{PetrowYoung} on the Weyl bound for Dirichlet $L$-functions of cube-free conductor.  In the present paper, we remove the cube-free hypothesis and establish the following theorem without any restrictions on the conductor of $\chi$. 
\begin{mytheo}
\label{thm:WeylBound}
For any primitive Dirichlet character $\chi$ modulo $q$ and $\eps>0$, we have 
 \begin{equation}
  L(1/2+it, \chi) \ll_{\eps} (q(1+|t|))^{1/6+\varepsilon}.
 \end{equation}
\end{mytheo}
 In another language, for any Hecke character $\chi$ over $\Q$, we have $L(1/2,\chi) \ll_{\eps} C(\chi)^{1/6+\eps}$ where $C(\chi)$ is the analytic conductor of $\chi$.

As in our previous work \cite{PetrowYoung} and that of Conrey and Iwaniec \cite{CI}, Theorem \ref{thm:WeylBound} is  based on Lindel\"of-on-average upper bounds for two closely-related cubic moments, see Theorems \ref{thm:mainthmMaassEisenstein} and \ref{thm:mainthmHybridVersion} below. Let $\cH_{it}(m,\psi)$ denote the set of Hecke-Maass newforms of conductor $m$, central character $\psi$, and spectral parameter $it$. A key new idea in \cite{PetrowYoung} was the shape of the family of automorphic forms into which we embed $\chi$, motivated by the following fact: If $\chi$ is a primitive character modulo $q$, $m\mid q$ and $\pi \in \cH_{it}(m,\overline{\chi}^2)$, then $\pi \otimes \chi \in \cH_{it}(q^2,1)$, see \cite[Prop.\ 3.8(iii)]{JacquetLanglands} or \cite[Thm.\ 3.1(ii)]{AtkinLi}. 
\begin{mytheo}
\label{thm:mainthmMaassEisenstein}
There exists a $B>2$ such that for all primitive $\chi$ modulo $q$ not quadratic and $\eps>0$ we have 
\begin{equation}\label{maththmequation}
 \sum_{|t_j| \leq T}  \sum_{m|q} \sum_{\pi  \in \mathcal{H}_{it_j}(m, \overline{\chi}^2)} L(1/2, \pi  \otimes \chi)^3 
 + 
 \int_{-T}^{T} 
 |L(1/2 + it, \chi)|^6  dt
 \ll_{\varepsilon} T^{B} q^{1+\varepsilon}.
\end{equation}
\end{mytheo}
\begin{mytheo}
 \label{thm:mainthmHybridVersion}
For all primitive $\chi$ modulo $q$, $\delta, \eps >0$, and $T \gg q^{\delta}$ we have
\begin{equation}
\label{eq:MaassEisensteinFormsBoundHybridVersion}
 \sum_{T \leq t_j < T+1}\sum_{m|q}  \sum_{\pi  \in \mathcal{H}_{it_j}(m, \overline{\chi}^2)} L(1/2, \pi  \otimes \chi)^3 
 + 
 \int_{T}^{T+1} 
 |L(1/2 + it, \chi)|^6  dt
 \ll_{\delta, \varepsilon} T^{1+\varepsilon} q^{1+\varepsilon}.
\end{equation}
\end{mytheo}
These two theorems, with the additional hypothesis that $q$ is cube-free, appeared as Theorems 1.1 and 1.2 of \cite{PetrowYoung}. That $\pi \otimes \chi$ has trivial central character is crucial because we may then rely on deep results of Guo \cite{Guo}, which guarantee that $L(1/2,\pi \otimes \chi)\geq 0$. We then may conclude Theorem \ref{thm:WeylBound} by a standard positivity argument.

The reader may wonder
why the cube-free hypothesis arose in our previous work and how we are able to remove it in this paper. In order to answer these questions, we briefly recall the proof of Theorems 1.1 and 1.2 of \cite{PetrowYoung}.

The proof of Theorem \ref{thm:mainthmMaassEisenstein} begins with several standard steps to estimate \eqref{maththmequation}. We apply an approximate functional equation to expand $L(1/2,\pi \otimes \chi)$ as a finite sum, and apply the Bruggeman-Kuznetsov formula and Poisson summation. The result is a sum of complete character sums. The archimedean integral can be treated by the method of stationary phase, and the non-archimedean sum by an explicit elementary calculation. By Mellin inversion, the result of these steps is that the cubic moment \eqref{maththmequation} is transformed to a main term plus a reciprocal ``dual moment'' of the rough shape 
\begin{equation}
\label{eq:fourthmomentIntro}
 \sum_{\psi \shortmod{q}} |L(1/2, \psi)|^4 g(\chi,\psi),
\end{equation}
where $g(\chi, \psi)$ is defined by 
\begin{equation}
\label{eq:gdef}
 g(\chi,\psi) = \sum_{u,t \shortmod{q}} \chi(t) \overline{\chi}(t+1) \overline{\chi}(u) \chi(u+1) \psi(ut-1).
\end{equation}

The existence of such a formula was first noticed in the case that $\chi$ is quadratic by the first author in \cite{PetrowMotohashi}. There have been several other examples of such reciprocal dual moments that have underpinned many other results in the literature. For instance, Motohashi \cite{Motohashi} proved 
 a formula of the rough shape $$ \int w(t) |\zeta(1/2+it)|^4 \,dt  \leftrightarrow \sum_{t_j} \sum_{\pi \in \cH_{it_j}(1,1)} \check{w}(t_j) L(1/2 , \pi)^3,$$ with an explicit transform $\check{w}$ of the test function $w$. See also \cite{MichelVenkateshGL2} for an elegant geometric proof of a special case and \cite{Nelson} for the generalization of their work to a wide class of test functions. In \cite{Young4th} the second author derived a similar duality in $q$-aspect:
$$\sum_{\chi \shortmod{p}} |L(1/2,\chi)|^4  \leftrightarrow \sum_{t_j} \sum_{\pi \in \cH_{it_j}(1,1)}\lambda_\pi(p) L(1/2 , \pi)^3$$ 
We also mention the more recent papers \cite{AndersenKiral, BlomerKhan, Frolenkov, Zacharias} giving additional reciprocity results for moments of $L$-functions.

To prove the estimates in Theorems \ref{thm:mainthmMaassEisenstein} and \ref{eq:MaassEisensteinFormsBoundHybridVersion}, it suffices to show for all $\eps>0$ that 
\begin{equation}
\label{eq:introDesiredFourthMomentBound}
 \sum_{\psi \shortmod{q}} |L(1/2, \psi)|^4 g(\chi,\psi) \ll_{\eps} q^{2+\eps}.
\end{equation}
The sum $g(\chi,\psi)$ is multiplicative, so it suffices to consider $g(\chi,\psi)$ for $q$ a prime power.
If $q=p$ is prime then the bound $g(\chi,\psi)\ll p$ follows from the theory of $\ell$-adic sheaves and trace functions, and in particular the Riemann hypothesis of Deligne. If $q=p^2$ then $g(\chi,\psi) \ll p^2$ by an elementary calculation (see \cite[\S 9.2]{PetrowYoung}). In these cases, we have for all $\eps>0$
\begin{equation*}
\sum_{\psi \shortmod{q}} |L(1/2, \psi)|^4 g(\chi,\psi) \ll_{\eps} q^{1+\eps} \sum_{\psi \shortmod{q}} |L(1/2, \psi)|^4 \ll_{\eps} q^{2+\eps}
\end{equation*}
by a standard large sieve-type inequality. This suffices to finish the proof of Theorems \ref{thm:mainthmMaassEisenstein} and \ref{eq:MaassEisensteinFormsBoundHybridVersion} in the case that $q$ is cube-free. 

If $q=p^3$ with $p \equiv 1 \pmod 4$ then (surprisingly!) there exist $2(p-1)$ characters $\psi$ modulo $q$ such that $|g(\chi,\psi)| = p^{1/2}q$. 
These $2(p-1)$ ``singular'' characters $\psi$ form two cosets of the subgroup of characters modulo $p$ sitting inside the group of all characters modulo $p^3$. So, we need to bound for two choices of $\alpha$ primitive modulo $p^3$ the sum
\begin{equation*}
\sum_{\psi \shortmod{p}} |L(1/2, \psi . \alpha)|^4g(\chi,\psi) \leq p^{\frac{1}{2}}q \sum_{\psi \shortmod{p}} |L(1/2, \psi . \alpha)|^4 .
\end{equation*}
At this point, applying the Burgess bound individually to each $L(1/2,\psi. \alpha)$ gives $\ll_{\eps} q^{2+\eps}p^{3/4}$, while over-extending to all characters modulo $p^3$ and using a large-sieve bound gives $q^{2+\eps}p^{1/2}$. Neither of the bounds is sufficient. We would need a bound of the strength 
\begin{equation}\label{eq:fourthmomentsubgroupprimecubedIntro}
\sum_{\psi \shortmod{p}} |L(1/2, \psi . \alpha)|^4 \ll_{\eps} p^{2.5+\eps}
\end{equation}
for all $\eps>0$, which already gives a subconvex bound (though not even as strong as the Burgess bound), so one needs a treatment of moments of the rough shape \eqref{eq:fourthmomentsubgroupprimecubedIntro} that goes beyond a large-sieve type inequality.
 We solve this problem of bounding fourth moments of Dirichlet $L$-functions along cosets by proving the following theorem.

\subsection{The fourth moment problem along subgroups}
Let $T \geq 1$, and $q, d \geq 1$ be integers with $d|q$. Let $q^* = \prod_{p^{\beta}||q} p^{\lceil \frac{2\beta}{3} \rceil}$, so that $q^*$ is the least positive integer so that $q^2|(q^*)^3$.
\begin{mytheo}
\label{thm:fourthmoment}
For all primitive $\chi$
 modulo $q$ and $\eps>0$ we have
\begin{equation}
\label{eq:fourthmoment}
 \int_{-T}^{T} \sum_{\psi \shortmod{d}} |L(1/2+it, \psi.\chi)|^4 \ll_{\eps} T \lcm(d, q^*) (qT)^{\varepsilon}.
\end{equation}
\end{mytheo}
Note that the set of characters $\{ \psi.\chi : \psi \pmod{d} \}$ is a coset of the subgroup of characters modulo $d$ inside the group of all characters modulo $q$. 
For example, if $q=p^3$ and $d=p^2$, \eqref{eq:fourthmoment} is a Lindel\"of-on-average upper bound, and more than suffices to establish the required estimate \eqref{eq:fourthmomentsubgroupprimecubedIntro}. This proves Theorems \ref{thm:mainthmMaassEisenstein} and \ref{eq:MaassEisensteinFormsBoundHybridVersion} in the case $q=p^3$.

In fact, Theorem \ref{thm:fourthmoment} is strong enough to establish 
\eqref{eq:introDesiredFourthMomentBound}, and hence
Theorems \ref{thm:mainthmMaassEisenstein} and \ref{eq:MaassEisensteinFormsBoundHybridVersion} in general. To see this, we perform an exhaustive calculation of the complete sums $g(\chi,\psi)$ in Sections \ref{section:miscCharacterSums} and \ref{section:gchipsiBehavior}, culminating in Theorems \ref{thm:gboundevencase} and \ref{thm:gboundoddcase}. These two theorems form one of the main achievements of this paper, describing completely the structure of the cosets of singular characters $\psi$ for which $|g(\chi,\psi)|$ is exceptionally large.

 Theorem \ref{thm:fourthmoment} may  be viewed as a $q$-aspect variant on Iwaniec's \cite{IwaniecZeta} short interval fourth moment bound
\begin{equation}
\label{eq:zetafourthmoment}
 \int_T^{T + T^{2/3}} |\zeta(1/2+it)|^4 dt \ll_{\eps} T^{2/3+\varepsilon}.
\end{equation}
See Section \ref{section:subfamilies} below for more discussion on why these results are analogous.
Iwaniec proves a number of other bounds on moments of zeta beyond \eqref{eq:zetafourthmoment}, and it would be interesting to prove $q$-aspect variants of those bounds also. The second moment problem along cosets has been studied in some cases by Nunes \cite{Nunes} and recently by Mili\'{c}evi\'{c} and White  \cite{MilicevicWhite}.

There are many other works in the literature on different variants of the fourth moment problem for Dirichlet $L$-functions and the zeta function. To name just a few, we mention
\cite{IwaniecZeta} \cite{HeathBrownFourthMoment} 
\cite{JutilaMotohashiCrelle}
\cite{Young4th}
\cite{BlomerMili}
\cite{KMS}
\cite{BHKM}.  Many of these papers focus on the problem of obtaining an \emph{asymptotic} formula for the fourth moment, 
which leads to some difficult analytic problems
that may be circumvented in the proof of Theorem \ref{thm:fourthmoment},
which is an upper bound.  
The asymptotic moment problem requires solving a shifted convolution sum in an unbalanced range, where the shift is very large compared to the length of summation. Since we are only interested in an upper bound, a simple Cauchy-Schwarz argument is able to completely sidestep this unbalanced problem (see Section \ref{section:reductionofFourthMomentToShiftedSums}). 

It is also interesting to compare the subgroup structure of the family of Dirichlet characters appearing in \eqref{eq:fourthmoment} with the thin Galois orbits studied in \cite{KMN}.

\subsection{Shifted divisor sum with character}\label{subsec:shifteddivisorwchar}
The main problem faced in the proof of Theorem \ref{thm:fourthmoment} is a strong bound on a shifted divisor sum with characters.    We now discuss this problem. 
Suppose that $w(x) = w_N(x)$ is a smooth weight function supported on $x \asymp N$. 
Let $\chi$ be a primitive Dirichlet character modulo $q$. For $h \geq 1$, consider 
\begin{equation}
\label{eq:SCSplain}
 \sum_{n} \chi(n+h) \tau(n+h) \overline{\chi}(n) \tau(n) w(n).
\end{equation}
For analytic reasons, it is preferable to study a closely-related variant of the form
\begin{equation}
\label{eq:SchihDef}
S(\chi,h):= \sum_{n \geq 1} \chi(n+h) \tau(n+h) \overline{\chi}(n) \sum_{n_1 n_2 = n} w(n_1, n_2),
\end{equation}
where $w(x,y)$ is smooth of compact support.  One can always apply a partition of unity to write \eqref{eq:SCSplain} as a short linear combination of sums of this form.  We suppose $w(x, y)$ is supported on $x \asymp N_1$, $y \asymp N_2$, with $N_1 N_2 = N$.  We also assume
\begin{equation*}
 h \ll N,
\end{equation*}
to avoid the more analytically difficult unbalanced shifted divisor sum.

We will also gain additional savings summing over $h$.  Let
\begin{equation}
 S(\chi) = \sum_{h \equiv 0 \shortmod{d}} 
 \sum_{n \geq 1} \chi(n+h) \tau(n+h) \overline{\chi}(n) \sum_{n_1 n_2 = n} w(n_1, n_2, h),
\end{equation}
where $w$ is part of a smooth family of functions of $x, y, h$, supported on $x  \asymp N_1$, $y \asymp N_2$, and $h \ll H \ll N$.  
The range relevant for proving Theorem \ref{thm:fourthmoment} is $N \ll q$.
We suppose that $w$ satisfies
\begin{equation}
\label{eq:wderivativebounds}
x^j y^k r^{\ell} w^{(j,k,\ell)}(x,y,r) \ll_{j,k,\ell} 1.
\end{equation}
Note that we can view $S(\chi)$ as a sum over $h$ of $S(\chi, h)$, provided we allow $w(x,y)$ appearing in \eqref{eq:SchihDef} to also depend on $h$.

\begin{mytheo}
\label{thm:ShiftedSumBounds}
 Suppose $d|q$ and $q^2|d^3$.   Then for all $\eps>0$
\begin{equation}
\label{eq:SchiFinalBound}
 S(\chi) \ll_{\eps} N  \Big(1 + \frac{H}{q}\Big) (qN)^{\varepsilon}.
\end{equation}
\end{mytheo}
Applying an approximate functional equation and orthogonality of characters, Theorem \ref{thm:fourthmoment} follows quickly from Theorem \ref{thm:ShiftedSumBounds}. The reduction step is detailed in Section \ref{section:reductionofFourthMomentToShiftedSums}. 

A pleasant technical feature of Theorem \ref{thm:ShiftedSumBounds} is that the bound in \eqref{eq:SchiFinalBound} does not include any factors that are sensitive to the current progress towards the Ramanujan conjecture (typically, the spectral analysis of shifted convolution sums with an individual shift  will give rise to such dependence).  The work of Blomer and Mili\'{c}evi\'{c} \cite{BlomerMili} also has this feature, which in the proof arose from a clever arrangement of H\"{o}lder's inequality after the spectral decomposition of the shifted convolution sum, and we were able to adapt their idea to our present setting.

\subsection{A sketch}
\label{section:sketch}
Recall that $q$ is our main parameter, and $d$ is an integer with $d|q$ and $q^2|d^3$.  This hypothesis ensures that $d$ and $q$ share the same set of prime factors.
In this sketch, we restrict ourselves to the special case $q=p^3$ and $d=p^2$, which illustrates the nature of the argument in a relatively simple situation.

A main idea of the proof of Theorem \ref{thm:ShiftedSumBounds} is that the sum $S(\chi)$ exhibits a conductor dropping phenomenon: writing $h=h'p^2$, we have 
\begin{equation}\label{eq:conductordrop}
\overline{\chi}(n) \chi(n+h) = \chi(1+h'\overline{n}p^2) = e_p(\ell_\chi h' \overline{n})
\end{equation}
for some integer $\ell_\chi$ with $(\ell_\chi,p)=1$, since $\chi$ has conductor $p^3$. Thus 
\begin{equation}\label{eq:SCSkloostermanfractions}
S(\chi) \approx \sum_{h' \ll \frac{H}{p^2}} \sum_{n \asymp N} \tau(n+h'p^2)\tau(n) e_p(\ell_\chi h' \overline{n}).
\end{equation}
(In this sketch, we use the symbol $\approx$ merely to mean that the left hand side may be transformed into an expression resembling the right hand side, plus an acceptable error term.) Observe that in \eqref{eq:SCSkloostermanfractions}, there is no possible cancellation in the interior sum when $p\mid h'$. However, these terms make a contribution of at most $\ll_{\eps} N \frac{H}{q}p^\eps$, which is acceptable, so we may assume that $(h',p)=1$ from here on. This step corresponds to the factorization $q=q_1q_2$ in Corollary \ref{coro:TchiEvaluation}, i.e. in the present sketch we may assume that $q=q_2=p^3$ and $h_{q_2} =p^2$.  

Next we solve the shifted convolution problem in \eqref{eq:SCSkloostermanfractions}. There are many ways to do this, and we opt to use an approximate functional equation-type formula for the divisor function of the rough form 
\begin{equation}
\label{eq:divisorfunctionAFEsketchVersionFAKE}
 \tau(n+h) \approx \sum_{c \ll \sqrt{N}} \frac{S(n+h,0;c)}{c},
\end{equation}
a method which is similar to using either the delta method or the circle method.  After using \eqref{eq:divisorfunctionAFEsketchVersionFAKE} to separate $n$ and $h$, we apply a double Poisson summation (i.e. Voronoi summation) to the sum over $n$.

It is technically more convenient not to use the formula \eqref{eq:conductordrop} at the outset, and instead to first apply an approximate functional equation-type formula similar to \eqref{eq:divisorfunctionAFEsketchVersionFAKE} for $\tau(n)\chi(n)$ (see Lemma \ref{lemma:divisorfunctionAFE}). We then use the conductor dropping formula \eqref{eq:conductordrop} in the course of computing the complete character sums that result from Poisson summation (see also the remarks following Corollary \ref{cor:Schihbound2}). 

Either way, the result of these steps is a formula of the shape 
\begin{equation}
\label{eq:SchiBoundIntroSketch}
  S(\chi) \approx \sum_{\substack{c \ll \sqrt{N} \\ (c,p) = 1}} \frac{N}{c^2 p^2}
 \sum_{\substack{h' \ll \frac{ H}{p^2} \\ (h',p) = 1}} \sum_{n_1 n_2 \ll \frac{(cp)^2}{N}} 
 S(p^2 h', - \overline{p}^2 n_1 n_2 ;c)
 \mathrm{Kl}_3(\ell_{\chi} h', \overline{c} n_1, \overline{c} n_2;p).
\end{equation}
The formula \eqref{eq:SchiBoundIntroSketch} is a simplified form of \eqref{eq:SchihWithKl3}.
Note that the dual sum after Poisson summation is of length $\frac{c^2 p^2}{N} \ll p^2$, while the original length was of size $N \ll q = p^3$, so this represents a significant savings.
At this point, if one uses the Weil bound for Kloosterman sums and Deligne's bound for hyper-Kloosterman sums, we obtain only $S(\chi) \ll p N^{3/4} \frac{H}{p^2} = H \frac{N^{3/4}}{p}$, which is far from what is needed for Theorem \ref{thm:ShiftedSumBounds} or even the weaker goal of \eqref{eq:fourthmomentsubgroupprimecubedIntro}.

To go further, we apply spectral methods from the theory of automorphic forms to the sum over $c$ in the guise of the Bruggeman-Kuznetsov formula (see section \ref{section:automorphicforms}). We first must resolve the $\overline{c}$ inside the argument of the ${\rm Kl}_3$, and do so by expanding into multiplicative characters, i.e. using the formula \begin{equation*}
 \mathrm{Kl}_3(x,y,z;p) = \frac{1}{\varphi(p)} \sum_{\eta \shortmod{p}} \tau(\overline{\eta})^3 \eta(xyz),
\end{equation*}
when $(xyz, p) = 1$ (see Lemma \ref{lemma:Kl3FourierTransform} for the general version). This leads to
\begin{equation}\label{eq:readyforBKsketch}
  S(\chi) \approx \frac{N}{p^3} 
  \sum_{\eta \shortmod{p}} \eta(\ell_{\chi}) \tau(\overline{\eta})^3
 \sum_{\substack{h' \ll \frac{H}{p^2}}} \sum_{n_1 n_2 \ll \frac{(cp)^2}{N}}  \eta(h'n_1 n_2) 
 \sum_{\substack{c \ll \sqrt{N} \\ (c,p) = 1}} \frac{\overline{\eta}^2(c)}{c^2 }
  S(p h', - \overline{p} n_1 n_2 ;c),
\end{equation}
where we also used $S(p^2 h', -\overline{p}^2 n_1 n_2 ;c) = S(ph', -\overline{p} n_1 n_2 ;c)$. Now we may apply the Bruggeman-Kuznetsov formula for $\Gamma_0(p)$ with central character $\eta^2$ at the cusps $\infty,0$  to the sum over $c$ in \eqref{eq:readyforBKsketch}. After some careful analysis of test functions, we obtain a spectral reciprocity formula for $S(\chi)$ of the rough shape
\begin{equation}\label{eq:sketchSpecRecipForm}
 S(\chi) \approx 
 \frac{N}{p^2} 
  \sum_{\eta \shortmod{p}} \eta(\ell_{\chi}) \frac{\tau(\overline{\eta})^3}{p^{3/2}}
  \sum_{t_j \ll 1} 
  \sum_{\pi \in \mathcal{H}_{it_j}(p, \eta^2)} \epsilon(\pi)_{\rm fin} \overline{\lambda_\pi}(p)
 L(1/2, \overline{\pi} \otimes \eta)^3   + ({\rm Hol.}) + ({\rm Eis.}),
\end{equation}
where $({\rm Hol.})$ and $({\rm Eis.})$ represent similar contributions from holomorphic cusp forms and Eisenstein series, respectively, and $\epsilon(\pi)_{\rm fin}$ is the finite part of the root number of $\pi$. 
See \eqref{eq:SMaassLfunctions} for the closest cousin to \eqref{eq:sketchSpecRecipForm}.

Applying H\"older's inequality,
we are reduced to the problem of bounding
\begin{equation}\label{eq:4thSketch}
 \sum_{\eta \shortmod{p}} 
  \sum_{t_j \ll 1} 
  \sum_{\pi \in \mathcal{H}_{it_j}(p, \eta^2)} |
 L(1/2, \overline{\pi} \otimes \eta)|^4.
\end{equation}
Using that $\overline{\pi} \otimes \eta \in \mathcal{H}_{it_j}(p^2, 1)$, we can bound this with a standard spectral large sieve inequality for level $p^2$.  The restriction to $q=p^3$ and $d=p^2$ in this sketch has led to \eqref{eq:4thSketch} being an overly-simplistic fourth moment problem.  In Theorem \ref{thm:fourthmomentSpectralBound} below, we bound the more general and difficult moment that arises.  See the remarks following Theorem \ref{thm:fourthmomentSpectralBound} for further discussion on this independently-interesting problem.

It is also instructive to compare the above sketch with Motohashi's spectral decomposition of the (smoothed) fourth moment of the zeta function over a short interval.  Theorem 5.1 of \cite{Motohashi} gives, roughly,
\begin{equation}
\label{eq:MotohashiFormulaSketch}
\int_{T}^{T+H} |\zeta(1/2+it)|^4 dt 
\leftrightarrow (\text{main term}) + \frac{H}{\sqrt{T}} \sum_{t_j \ll \frac{T}{H}} L(1/2,u_j)^3 t_j^{-1/2} \sin(\theta_j),
\end{equation}
where $\theta_j \sim t_j \log t_j$.  Motohasi derives \eqref{eq:MotohashiFormulaSketch} from an exact formula for the weighted fourth moment of zeta, and the sequence of steps used in the proof is similar to that presented in the above sketch.  In particular, the dual family of Maass cusp forms arises from a spectral decomposition of the shifted divisor problem.  In Motohashi's case, the shifted divisor problem includes a $t$-aspect oscillatory factor, as in $\sum_{n} \tau(n) \tau(n+h) (\frac{n+h}{n})^{iT}$; this should be compared with \eqref{eq:SCSplain}.  The fact that $h$ is small means that $(\frac{n+h}{n})^{iT} \approx \exp(iT \frac{h}{n})$, which is an archimedean analog of the conductor-dropping phenomenon of \eqref{eq:conductordrop}.  
It is pleasant to compare \eqref{eq:MotohashiFormulaSketch}
with \eqref{eq:sketchSpecRecipForm}; taking $H = T^{2/3}$ gives the closest comparison.  The archimedean oscillatory factor $\sin(\theta_j)$ is analogous to the argument of $\tau(\overline{\eta})^3$, which is in line with Stirling's approximation, and the analogy between Gauss sums and the gamma function.

\subsection{Remarks on close-knit families}
\label{section:subfamilies} 
 A key idea going into the proof of Theorem \ref{thm:WeylBound} is the shape of the family of automorphic forms in Theorems \ref{thm:mainthmMaassEisenstein} and \ref{thm:mainthmHybridVersion}.  This is yet another example of the by now well-known and powerful technique of deforming in a family of automorphic forms or $L$-functions (see \cite{SarnakShinTemplier} for more discussion).
To this end, we now offer some brief remarks on families in an ad-hoc context,  which may be  useful for interpreting the moment problems considered in this article and our previous work \cite{PetrowYoung}.

To fix ideas, let us work in the context of some ambient family of automorphic forms $\cF$. 
 Let $\pi_0 \in \cF$ and suppose that one wishes to prove subconvexity for $L(\pi_0, 1/2)$. A typical strategy is to choose a sub-family $\mathcal{F}_0 \subseteq \cF$ containing $\pi_0$, and consider, for example, a second moment of $L$-functions of the form $\sum_{\pi \in \mathcal{F}_0} |L(\pi, 1/2)|^2$. 
It is then advantageous to choose the family $\mathcal{F}_0$ to have high spectral  completeness, while at the same time to be as small as possible.   

A natural way to quantify the closeness of two automorphic forms or representations is through the quantity 
\begin{equation}
D(\pi_1,\pi_2) :=\frac{C(\pi_1 \otimes \overline{\pi_2})}{C(\pi_1 \otimes \overline{\pi_1})^{1/2} C(\pi_2 \otimes \overline{\pi_2})^{1/2}},
\end{equation}
where $C$ is the analytic conductor. Given a family $\cF_0$, one can reasonably speak of the diameter of $\cF_0$ with respect to $D(\pi_1,\pi_2)$. Alternatively, one can define a sub-family $\cF_0\subseteq \cF$ by $\cF_0= \cF_0(r)= \{ \pi \in \cF : D(\pi,\pi_0) \leq r\}$. 
Such small families fit into the framework of harmonic families of \cite{SarnakShinTemplier}, since the analytic conductor is a local invariant and Rankin-Selberg convolutions may be computed locally. Informally, we call families with small $D(\pi_1,\pi_2)$ `close-knit'.

Working locally we can be a bit more precise. Let $k$ be a non-archimedean local field with finite residue field. It follows easily from much more general work of Bushnell-Henniart \cite{BushnellHenniartRamification} (see \cite[Thm.\ 1]{LapidOnAnIneqOfBH}) that the function 
\begin{equation}
d(\pi_1,\pi_2) := c(\pi_1 \otimes \overline{\pi_2}) - \tfrac{1}{2} c(\pi_1 \otimes \overline{\pi_1})- \tfrac{1}{2} c(\pi_2 \otimes \overline{\pi_2}),
\end{equation}
where $c$ is the conductor exponent, defines a pseudometric on the space of irreducible supercuspidal representations of $\GL_n(k)$.

We now consider some simple examples. Let $\cF$ be the set of Dirichlet characters and let $\pi_0 = \chi \in \cF$ be of conductor $q$.  For $d \mid q$, the set $\cF_\chi(d) = \{ \chi.\psi : \psi \pmod{d} \}$ is an example of a close-knit family of diameter $d$ around $\chi$.  The family $\cF_\chi(d)$ is precisely the family considered in Theorem \ref{thm:fourthmoment} (see also \cite{Nunes} and \cite{MilicevicWhite}).
  
 Considering the archimedean aspect, one finds many examples of families grouped according to $D(\pi_1,\pi_2)$ in the literature. The short-interval $t$-aspect integral
 found in \eqref{eq:zetafourthmoment} is such an instance.  To describe a slightly more advanced example, let $\cF$ be the set of Hecke-Maass eigenforms for $\SL_2(\Z)$. Given $u \in \cF$, write $t_u$ for its spectral parameter. For parameters $1 \ll \Delta \ll T$, let $\cF_T(\Delta) = \{u \in \cF: T < t_u \leq T+\Delta\}.$  The conductor of $u \otimes u'$  is $\asymp (1+|t_u - t_{u'}|)^2(1+|t_u+t_{u'}|)^2 \ll \Delta^2 T^2$, so that the close-knit family $\cF_T(\Delta)$ has diameter $\ll \Delta^2$. The cubic moment of $L$-functions over this family $\sum_{u \in \cF_T(\Delta)} L( 1/2,u)^3$ was studied by Ivi\'c \cite{Ivic}, from which he derived Weyl-strength subconvexity in the spectral aspect.

The family of automorphic forms appearing in Theorem \ref{thm:mainthmMaassEisenstein} provides another example. Let $\cF = \mathcal{H}_{it}(q^2,1)$. For $\chi$ a primitive character modulo $q$ that is \emph{not} quadratic, consider the family of twists 
\begin{equation}
\cF_\chi:= \{ \pi \otimes \chi :  \pi \in \mathcal{H}_{it}(m,\overline{\chi}^2), \, m \mid q\} \subseteq  \cF.
\end{equation}
The family $\cF_\chi$ admits a simple interpretation in terms of local representation theory. 
The local components of $\pi \otimes \chi \in \cF_\chi$ are principal series at all finite places. Precisely, $(\pi \otimes \chi)_p \simeq \pi(\chi_p , \overline{\chi_p})$ for all $p < \infty$, where $\chi_p$ is a quasi-character of $\Q_p^\times$ whose restriction to $\Z_p^\times$ matches the restriction of $\chi$ to $\Z_p^\times$. Thus, the family $\cF_\chi$ could also have been described by specifying the local component at finitely many places of ramification to be a single principal series representation (up to unramified twists). Locally at $p$, we have $d(\pi_{1,p}, \pi_{2,p}) =0$ for any $\pi_1, \pi_2 \in \cF_\chi$, so the family $\cF_\chi$ is as close-knit as possible at finite places.

Another interesting example occurs for thin Galois orbits of Dirichlet $L$-functions; see \cite[pp. 6961-6963]{KMN} for more details.

It is illuminating to view many families of $L$-functions under this lens, and the authors hope that this way of thinking may lead to beneficial choices of families of $L$-functions for problems in analytic number theory.  

\subsection{Bounds on character sums}
 Theorem \ref{thm:WeylBound} leads to an improvement on the Burgess bounds for character sums in some ranges.
\begin{mytheo}
\label{prop:charactersum}
For all primitive Dirichlet characters $\chi$ modulo $q$, $x\geq 1$, and $\eps>0$ we have 
 \begin{equation}
  \sum_{n \leq x} \chi(n) \ll
  \begin{cases}
    x^{1/2} q^{11/64+\varepsilon} \\
    x^{8/15} q^{7/45 + \varepsilon}.
  \end{cases}
 \end{equation}
\end{mytheo}
Remarks. The former bound is better than the latter for $x \gg q^{47/96}$.  Recall the Burgess bound states $\sum_{y < n \leq y+x} \chi(n) \ll x^{1-\frac{1}{r}} q^{\frac{r+1}{4r^2} + \varepsilon}$, for $r=2,3$, and for any $r \geq 1$ if $q$ is cube-free (see \cite[Thm.\ 12.6]{IK}).  Theorem \ref{prop:charactersum} improves on the Burgess bounds with $y=0$ and $r = 2$ or $3$ in all non-trivial ranges.
\begin{proof}[Sketch of proof]
Let $0 < h < x$ be a parameter to be chosen later.
Let $w$ be a smooth weight function so that $w(t) = 1$ for $0 \leq t \leq x$, $w(t) = 0$ for $t \geq x + h$, and satisfying $w^{(j)}(t) \ll_j h^{-j}$, for all $t > 0$.  Then
\begin{equation*}
S(\chi,w) :=  \sum_{n=1}^{\infty} \chi(n) w(n) = \frac{1}{2 \pi i} \int_{(\sigma)} \widetilde{w}(s) L(s, \chi) ds.
\end{equation*}
Integration by parts shows that the integral may be essentially truncated at $\text{Im}(s) \ll x/h$.  Taking $\sigma = 1/2$ and using Theorem \ref{thm:WeylBound} gives a bound on the smoothed sum, showing $S(\chi, w) \ll x^{1/2} q^{1/6+\varepsilon} (x/h)^{1/6}$.  Next, we have $\sum_{n \leq x} \chi(n) = S(\chi, w) - \sum_{x<n \leq x+h} \chi(n) w(n)$.  For the latter sum, we may use summation by parts and the Burgess bound  with $r=2$ or $r=3$.  Choosing $h$ optimally then gives the two bounds.
\end{proof}
The interested reader may derive additional bounds for cube-free conductors using the Burgess bound for larger values of $r$ in the final step of the above proof.
The authors thank Roger Heath-Brown for suggesting the use of the Burgess bound on the short interval.

\subsection{Organization of the paper}
This paper is divided into two parts that are almost entirely independent of each other, and the notation is not necessarily consistent between the two parts.  The authors believe this is a feature and not a bug.

The first part of this paper is devoted to the cubic moment problem and its reduction to the bound on the fourth moment along subgroups (i.e., Theorem \ref{thm:fourthmoment}), and is contained in Sections \ref{section:miscCharacterSums}\textendash\ref{section:ZpropertiesandtheCubicMomentBound}.  Specifically, Section \ref{section:miscCharacterSums} contains a variety of character sum lemmas, Section \ref{section:gchipsiBehavior} has a full analysis of $g(\chi,\psi)$, and Section \ref{section:ZpropertiesandtheCubicMomentBound} finishes the proof of Theorem \ref{thm:WeylBound} given the veracity of Theorem \ref{thm:fourthmoment}.

The second part gives the proof of Theorem \ref{thm:fourthmoment}, and is contained in Sections \ref{section:reductionofFourthMomentToShiftedSums}\textendash\ref{section:spectralanalysisofshiftedsum}.  
Section \ref{section:reductionofFourthMomentToShiftedSums} briefly deduces the proof from the shifted sum bound (Theorem \ref{thm:ShiftedSumBounds}).
In Section \ref{section:automorphicforms}, we present the background from the theory of automorphic forms,with an emphasis on the use of canonically-normalized Fourier expansions in the style of \cite{MichelVenkateshGL2}. Section \ref{section:analyticnumbertheory} contains some tools from analytic number theory.  The proof of the shifted sum bound begins in earnest in Section \ref{section:harmonicanalysis} and is completed in Section \ref{section:spectralanalysisofshiftedsum}.

\subsection{Notation}
\begin{itemize}
\item We denote by $\N$ the set of natural numbers $\{1,2, \ldots\}$ without zero. 
\item For a finite abelian group $G$, we denote by $\widehat{G}$ its unitary dual. Exception: in Sections \ref{section:automorphicforms} and \ref{section:spectralanalysisofshiftedsum} we write $\widehat{\Z}=\prod_p \Z_p \simeq \varprojlim \Z/n\Z$. 
\item For $\pi$  a newform/automorphic representation on $\GL_2$ and $\chi$ a Dirichlet character, there are (at least) two standard conventions for the meaning of $L(s, \pi \otimes \chi)$.  One convention is that it equals the straightforward Dirichlet series $\sum_n \lambda_{\pi}(n) \chi(n) n^{-s}$, and the other is that it equals the automorphic $L$-function associated to the twist of $\pi$ by the Hecke character corresponding to $\chi$.  In this paper, all $L$-functions of the form $L(s,\pi \otimes \chi)$ use the automorphic definition. However, since the two conventions can only differ at Euler factors corresponding to primes dividing the conductor of $\pi$ or the conductor of $\chi$, all statements of theorems or lemmas involving $L(s,\pi \otimes \chi)$ remain equally valid using either convention.
\item For $\chi$ a Dirichlet character, we use $L(s,\chi)$ to denote the classically-defined Dirichlet series $\sum_n \chi(n) n^{-s}$.  If $\chi$ is primitive, this agrees with the automorphic convention.  If $\chi$ is not primitive, but is induced by $\chi^*$, then it is easy to convert between $L(s,\chi)$, and $L(s,\chi^*)$.
\end{itemize}

\subsection{Acknowledgements}
The authors thank Roger Heath-Brown, Rizwan Khan, Emmanuel Kowalski, Djordje Mili\'{c}evi\'{c}, and Lillian Pierce for comments and encouragement. We also thank the referees for many corrections and helpful suggestions.

\section{Character sums to prime power modulus}
\label{section:miscCharacterSums}
In this section we collect some lemmas that are useful for evaluating the character sums to prime-power modulus that arise in our work.  

\subsection{The Postnikov formula}
\begin{mylemma}
\label{lemma:Postnikov}
 Let $p$ be an odd prime, and $\beta \geq 2$.  There exists a unique 
  group homomorphism $\ell: \widehat{(\mz/p^\beta \mz)^{\times}} \rightarrow \mz/p^{\beta-1} \mz$, $ \chi \mapsto \ell_{\chi}$, such that the Postnikov formula holds: for each Dirichlet character $\chi$ modulo $p^\beta$  and $t\in \Z$ we have
 \begin{equation}
 \label{postnikov} \chi(1+pt) = e_{p^\beta}(\ell_\chi \log_p(1+pt)).
 \end{equation}
 The map $\ell$ is surjective, and
 for $1 \leq \alpha \leq \beta$  we have that $\ell_{\chi_1} \equiv \ell_{\chi_2} \pmod {p^{\beta-\alpha}}$ if and only if 
  $\chi_1 \overline{\chi_2}$ is a character modulo $p^\alpha$. 
 \end{mylemma} 

 \begin{proof}
  For $1\leq \alpha \leq \beta$, consider the reduction modulo $p^\alpha$ map $$(\Z/p^\beta \Z)^\times \to (\Z/p^\alpha\Z)^\times,$$ and denote its kernel by $U_\alpha$. 
  Let $e(x)$ be the continuous character of $\mq_p$ agreeing with $e^{2 \pi i x}$ for $x \in \mq$, and let
  $e_{p^\beta}(x) = e(p^{-\beta} x)$. 
  Let $\log_p: 1 + p \mz_p \rightarrow p \mz_p$ be the $p$-adic logarithm defined by the convergent power series expansion
  \begin{equation*}
   \log_p(1+x) = x - x^2/2 + x^3/3 \mp \dots.
  \end{equation*}
It is easy to check that $\log_p(1 + p^\beta \mz_p) \subseteq p^\beta \mz_p$, and in fact 
\begin{equation}
\label{eq:padiclogarithmLinearApproximation}
\log_p(1+p^\beta x) \equiv p^\beta x \pmod{p^{2\beta}},
\end{equation}
since $p$ is odd.
  
  Consider the map $f:U_1 \to S^1$ defined by
  \begin{equation*}
  \label{f} f:t \mapsto e_{p^{\beta}}(\log_p(t)).
  \end{equation*} 
 The function $f$ is well-defined by \eqref{eq:padiclogarithmLinearApproximation}, and is a group homomorphism since $\log_p(xy) = \log_p(x) + \log_p(y)$ for $x,y \in 1+p \mz_p$ (see e.g. \cite[Prop.\ 5.5]{Neukirch}).  We claim that $f$ has order $p^{\beta-1}$ in $\widehat{U_1}$. Indeed, if $t = 1+px \in U_1$, then we have $f(t)^{p^{\beta-2}}= e_{p^{2}}(\log_p(t)) = e_p(x)$, so $f^{p^{\beta-2}}$ is not trivial in $\widehat{U_1}$, yet $U_1$ has order $p^{\beta-1}$. Therefore $\widehat{U_1}$ is cyclic and $f$ is a generator.  Define $\ell_\chi$ to be the unique integer modulo $p^{\beta-1}$ such that 
 $\chi\vert_{U_1} = f^{\ell_\chi}$, 
which is equivalent to the Postnikov formula \eqref{postnikov}.  We easily see that $\ell$ is a group homomorphism.
Next we show this map is surjective.   The kernel of $\ell$ is the subgroup of characters trivial on $U_1$, which is isomorphic to $\widehat{(\mz/p\mz)^{\times}}$.  Hence by comparing cardinalities, we see $\ell$ is surjective.

We claim that $f \vert_{U_\alpha}$ has order $p^{\beta-\alpha}$ in the group $\widehat{U_\alpha}$. Indeed, writing $t=1+p^\alpha x$, we have $f(t)^{p^{\beta-\alpha-1}} = e_{p^\beta}(p^{\beta-\alpha-1} \log_p(t)) = e_{p^{\alpha+1}}(\log_p(t)) = e_p(x)$, showing the claim.
 Then
 $ \chi \vert_{U_\alpha} = f \vert_{U_\alpha}^{\ell_\chi},$ and we deduce that $\ell_{\chi} \equiv 0 \pmod{p^{\beta-\alpha}}$ if and only if $\chi \vert_{U_\alpha} = 1$, which in turn is equivalent to the condition that $\chi$ is a character modulo $p^\alpha$.  The final statement of the lemma now follows, since $\ell$ is a group homomorphism.
 \end{proof}

\subsection{Character sums}
A rational function $f \in \Z(t)$ is an equivalence class of pairs of polynomials $f_1/f_2$ with integer coefficients and $f_2$ not identically zero. An integer $t_0$ is said to be in the domain of $f$ if $f_2(t_0)\neq 0$ with $f=f_1/f_2$ written in lowest terms (i.e., with $f_1$ and $f_2$ coprime). Meanwhile, a rational function $f \in (\Z/p^\beta \Z)(t)$ is an equivalence class of pairs of polynomials $f_1/f_2$ with coefficients in $\Z/p^\beta\Z$ and with $p$ not dividing all of the coefficients of $f_2$. Similarly, $t_0 \in \Z/p^\beta\Z$ is said to be in the domain of $f$ if $p \nmid f_2(t_0)$ with $f=f_1/f_2$ in lowest terms. (Recall that in a commutative ring $A$, two elements $a,b \in A$ are called coprime if $(a)+(b)=A$.)
 If $p$ does not divide $x \in \Z/p^\beta \Z$ then we call $x$ a ``$p$-adic unit''.  
 The above notions also extend naturally to several variables. Lastly, in the character sums of the form $\sumstar_t \chi(f(t)) \psi(g(t))$ that we study in Sections \ref{section:miscCharacterSums} and \ref{section:gchipsiBehavior} of this paper, the $*$ is always taken to mean that we sum over those $t$ lying in the intersection of the domains of $f$ and $g$.

 Let $g \in \Z(t)$ be a rational function whose reduction $\overline{g}$ modulo $p^\beta$ exists.  Let $t_0$ be an integer whose reduction modulo $p^\beta$ lies in the domain of $\overline{g}$. Then, it is easy to see that $\frac{g^{(n)}(t_0)}{n!} \in \mz_p$ for all $n \geq 0$.
 In particular, this shows that 
\begin{equation}
\label{eq:rationalfunctionTaylorExpansion}
g(x_0 + p^\beta x_1) \equiv g(x_0) + p^\beta  g'(x_0) x_1  \pmod{p^{2\beta}}
\end{equation}
for any integer $x_0$ reducing to the domain of $\overline{g}$.

More generally,  suppose that
$p$ does not divide the whole denominator of $g \in \Z(t_1,\ldots,t_n)$ and $x_0 \in \Z^n$ reduces modulo $p^\beta$ to lie in the domain of $\overline{g}$. Then, the Taylor expansion of $g$ at $x_0$ has coefficients in $\Z_p$ and
 we have 
\begin{equation}
\label{eq:rationalfunctionTaylorExpansionMultiVariableVersion}
 g(x_0 + p^\beta x_1) \equiv g(x_0) + p^\beta g'(x_0) x_1 + p^{2\beta} \tfrac12 g''(x_0)[x_1] \pmod{p^{3\beta}},
\end{equation}
where $g'$ denotes the gradient of $g$, $g''$ is the Hessian matrix, and $A[x]= x^\intercal A x$ is the quadratic form associated to a square matrix $A$ and evaluated at $x$.

  For $f_i \in \Z(t_1,\ldots, t_n)$, $i = 1, \ldots d$, let  $f = (f_1, \ldots, f_d) \in \Z(t_1,\ldots, t_n)^{d}$
 be the associated $d$-tuple of rational functions.  
 For such an $f$ we have the associated $d \times n$ Jacobian matrix, which we denote by $f' \in M_{d \times n} (\Z(t_1,\ldots, t_n))$. Similarly, we have the logarithmic Jacobian $(\log f)'$, where the $ij$ entry is given by 
 $\partial_j f_i/f_i$.

Define an additive character $\addchar$ modulo $q = (q_1, \dots, q_d) \in \mn^d$ as a group homomorphism $\mz^d/q\mz^d \rightarrow \mc^{\times}$, lifted to $\mz^d$ by periodicity.
 By the Chinese remainder theorem, $\addchar$ can be expressed uniquely as
  $\addchar(n) = \addchar_1(n_1)\ldots \addchar_d(n_d)$, with
$\addchar_i(n) =e_{q_i}(a_{\addchar_i}n)$ for some $a_{\addchar_i} \in \mz$. 
If $q$ is diagonal, we may abuse notation and write simply $\addchar(n) = e_q(a_{\addchar} n)$ where $a_{\addchar} n$ is the standard scalar product. 
 
Likewise, a Dirichlet character modulo $q = (q_1, \dots, q_d)$ is a map $(\mz^d/q\mz^d)^{\times} \rightarrow \mc^{\times}$ extended to $\mz^d$ in the natural way.  Again, $\chi$ may be expressed uniquely as 
  $  \chi((n_1,\ldots, n_d)) = \chi_1(n_1) \cdots \chi_d(n_d)$, where
 $\chi_i$ is modulo $q_i$, $i=1, \ldots,d$.  
 If $p$ is odd, $q=(p^{\beta}, \dots, p^{\beta})$ with  $\beta \geq 2$, we  define $\ell_{\chi} = (\ell_{\chi_1}, \ldots, \ell_{\chi_d})$ with $\ell_{\chi_i}$ as in Lemma \ref{lemma:Postnikov}.  Note that  the Postnikov formula generalizes to give for $n = (n_1, \dots, n_d)$ with each $n_i \equiv 1 \pmod{p}$ the formula
  $\chi(n) = e_{p^{\beta}}(\ell_{\chi} \log_p(n))$, with the standard scalar product and where $\log_p(n) = (\log_p(n_1), \dots, \log_p(n_d))$.
  
  \begin{mylemma}
  \label{lemma:characterSumPrimePowerEvenExponent}
  Let $p$ be an odd prime, 
$\chi$ be a Dirichlet character modulo $(p^{2 \alpha}, \dots, p^{2 \alpha})$, 
$\addchar$ be an additive character modulo $(p^{2 \alpha}, \dots, p^{2 \alpha})$ and $f, g \in \Z(t_1,\ldots, t_n)^d$  as above.   Consider the congruence
\begin{equation}
  \label{eq:linearConditionVanishingSum}
  \ell_{\chi} (\log f)'(t_0) + a_{\addchar} g'(t_0) \equiv 0 \pmod{p^{\alpha}}. 
  \end{equation}
We have 
\begin{equation}
\label{eq:characterSumPrimePowerEvenExponent}
S:=\sumstar_{t \in \left( \Z/p^{2\alpha} \Z \right)^n } \chi(f(t)) \addchar(g(t)) = p^{n\alpha}
\sumstar_{\substack{t_0 \in \left( \Z/p^\alpha \Z \right)^n \\ 
\text{\eqref{eq:linearConditionVanishingSum} holds}}
} 
\chi(f(t_0)) \addchar(g(t_0)).
\end{equation}
The right hand side does not depend on the choice of 
lift of $t_0$ to $\Z_p^d$.
  \end{mylemma}
Remark.  This is a natural multi-variable generalization of \cite[Lem.\ 12.2]{IK}. 
  \begin{proof}
  Write $t=t_0+p^\alpha t_1$, and $\chi(f(t)) = \chi(f(t_0)) \chi(f(t)/f(t_0)).$ 
  Then, by the Postnikov formula \eqref{postnikov}, \eqref{eq:padiclogarithmLinearApproximation}, and \eqref{eq:rationalfunctionTaylorExpansionMultiVariableVersion}, we have
  $$\chi(f(t)/f(t_0)) = e_{p^{2\alpha}} (\ell_\chi \log_p( f(t)/f(t_0) ) )
  = e_{p^{\alpha}}(\ell_{\chi} (\log_{p} f)'(t_0) t_1)
  .$$ 
  Similarly, $\addchar(g(t)) = \addchar(g(t_0)) \addchar(g(t) - g(t_0))$, and 
  \begin{equation*}
  \addchar(g(t) - g(t_0)) =  e_{p^\alpha}(a_{\addchar} g'(t_0) t_1 ) 
 \end{equation*} 
 Then $$S= \sumstar_{t_0 \shortmod {p^\alpha}} \chi(f(t_0)) \addchar(g(t_0)) \sum_{t_1 \shortmod {p^\alpha}} e_{p^{\alpha}}( \ell_\chi(\log f)'(t_0) t_1 + a_{\addchar} g'(t_0) t_1).$$
The inner sum vanishes unless \eqref{eq:linearConditionVanishingSum} holds, 
  giving the formula stated in the lemma.  The proof shows that the right hand side of \eqref{eq:characterSumPrimePowerEvenExponent} is independent of choice of lifts.
  \end{proof}
 
  Next we generalize the odd exponent case of \cite[Lem.\ 12.3]{IK}.  To this end, we introduce multi-variable Gauss sums.  Let $L: \mz^n \rightarrow \mz$ be a linear form with integer coefficients, and $Q: \mz^{n} \rightarrow \mz$ be a quadratic form (see e.g. 
  \cite[Ch.IV Def.\ 1]{SerreCinA}). 
 Define
 \begin{equation}
  G_p(Q,L) = \sum_{t \in \F_p^n} e_p(Q[t] + L t).
 \end{equation}
  
 \begin{mylemma}
 \label{lemma:characterSumPrimePowerOddExponent}
  Let $p$ be an odd prime, 
$\chi$ be a Dirichlet character modulo $(p^{2 \alpha+1}, \dots, p^{2 \alpha+1})$, $\addchar$ be an additive character modulo $(p^{2 \alpha+1}, \dots, p^{2 \alpha+1})$,  and $f, g \in \Z(t_1,\ldots, t_n)^d$ as above. Then 
$$
S=
\sumstar_{t \in \left( \Z/p^{2\alpha+1} \Z \right)^n } \chi(f(t)) \addchar(g(t)) 
= 
p^{n\alpha}
\sumstar_{\substack{t_0 \in \left( \Z/p^\alpha \Z \right)^n 
\\ 
\eqref{eq:linearConditionVanishingSum} \text{ holds}
}
} 
\chi(f(t_0)) \addchar(g(t_0)) G_p(Q,L),
$$
where 
\begin{equation}
\label{eq:Mlinearformdef}
 L = p^{-\alpha} (\ell_{\chi} (\log f)'(t_0) + a_{\addchar} g'(t_0))
\end{equation}
and $Q$ is the quadratic form with associated matrix (in the standard basis for $\mz^n$) given by
\begin{equation}
\label{eq:QquadraticformDef}
 Q = \tfrac12 \ell_{\chi} (\log f)''(t_0) + \tfrac12 g''(t_0).
\end{equation}
The right hand side does not depend on the choice of 
lift of $t_0$ to $\Z_p^d$.
  \end{mylemma}
\begin{proof}
 Write $t=t_0+p^\alpha t_1$, and $\chi(f(t)) = \chi(f(t_0)) \chi(f(t)/f(t_0)).$ 
  Then, by the Postnikov formula \eqref{postnikov}, \eqref{eq:padiclogarithmLinearApproximation}, and \eqref{eq:rationalfunctionTaylorExpansionMultiVariableVersion}, we have
  \begin{equation*} \chi(f(t)/f(t_0)) = e_{p^{2\alpha+1}} (\ell_\chi \log_p( f(t)/f(t_0) )  = e_{p^{\alpha+1}}( \ell_\chi (\log f)'(t_0) t_1) e_p(\tfrac12 \ell_{\chi} (\log f)''(t_0)[t_1] ).\end{equation*}
  Similarly, $\addchar(g(t)) = \addchar(g(t_0)) \addchar(g(t) - g(t_0))$, and 
  \begin{equation*}
  \addchar(g(t) - g(t_0)) =  e_{p^{\alpha+1}}(a_{\addchar} g'(t_0) t_1 ) e_p(\tfrac12 g''(t_0)[t_1]).
 \end{equation*} 
Changing variables $t_1 \rightarrow t_1 + p e_j$, where $e_j$ is the $j$-th standard basis vector, leaves the quadratic terms unchanged.  Hence the inner sum vanishes unless \eqref{eq:linearConditionVanishingSum} holds, in which case we obtain the claimed result.
\end{proof}

In view of Lemma \ref{lemma:characterSumPrimePowerOddExponent}, it will be useful to estimate quadratic Gauss sums.  
\begin{mylemma}
\label{lemma:quadraticGaussSum}
 Let $p$ be an odd prime, let $Q$ be a quadratic form over $\mf_p$, and $L$ a linear form, as above.  Let $V$ be the isotropic subspace of $Q$.  Let $r_Q$ denote the rank of $Q$.  Then $G_p(Q,L)$ vanishes unless $L \vert_V=0$, in which case
 \begin{equation*}
  |G_p(Q,L)|  = p^{\frac{r_Q}{2}}  p^{(n-r_Q)}.
 \end{equation*}
\end{mylemma}
\begin{proof}
 It is well-known that one can change basis for $\mf_p^n$ so that the quadratic form $Q$ is orthogonal with respect to this basis (e.g. see \cite[Ch.IV.1.4 Thm.\ 1]{SerreCinA}). 
In particular, we have $\mf_p^n = V \oplus U$ where $V$ is the isotropic subspace of $Q$, and $U$ is a complementary subspace.  Therefore, if $v \in V$ and $u \in U$, then $Q[v+u] = Q[u]$.  Using this basis to calculate the Gauss sum, we have
\begin{equation*}
 G_p(Q,L) = \Big(\sum_{v \in V} e_p(Lv) \Big) \Big( \sum_{u \in U} e_p(Q[u] + Lu) \Big).
\end{equation*}
Note that the sum over $v$ vanishes unless $L \vert_V =0$, while the sum over $u$ has absolute value $p^{r_Q/2}$, where $r_Q$ is the rank of the quadratic form, since $U$ has a basis on which $Q$ is diagonalized, and by the standard one-variable evaluation of quadratic Gauss sums.  This completes the proof.
\end{proof}

Motivated by an application (namely, Lemma \ref{lemma:someCharacterSumBoundUsefulforHhat}), we wish to mildly generalize Lemmas \ref{lemma:characterSumPrimePowerEvenExponent} and \ref{lemma:characterSumPrimePowerOddExponent} as follows. 
Let $p$ be an odd prime and suppose $1\leq \beta \leq \gamma$.  Let $f,g \in \Z(t_1, \ldots, t_n)$ with $p$ not dividing every coefficient of the denominators of $f,g$. Let $V$ be the subset of $ x \in (\Z/p^\beta \Z)^n$ for which $p^{\gamma-\beta}x$ modulo $p^\gamma$ lies in the domain of $f$ modulo $p^\gamma$. 
   Then $F(x) = f(p^{\gamma-\beta}x)$ defines a function $F: V \to \Z/p^\gamma\Z$ and we call $V$ the domain of $F$. Let us write $G$ for the same construction applied to $g$.

 These definitions extend component-wise, as follows. Let $\gamma=(\gamma_1, \dots, \gamma_d)$ with each $\gamma_i \geq \beta \geq 1$, and $p^\gamma = (p^{\gamma_1},\ldots, p^{\gamma_d}) $. 
Given $f=(f_1,\ldots,f_d), g = (g_1,\ldots, g_d) \in \Z(x_1,\ldots,x_n)^d$, define $F= (F_1, \ldots,F_d)$ and $G=(G_1, \dots, G_d)$ by $F_i(x)=f_i(p^{\gamma_i-\beta}x)$ and $G_i(x)=g_i(p^{\gamma_i-\beta}x)$ for all $1 \leq i \leq d$. Then $F,G$ define functions with domains given by the intersection of the domains of the $F_i, G_i$, as above. 
Let $\chi$ be a Dirichlet character modulo $p^\gamma$ and $\addchar$ an additive character modulo $p^\gamma$.  If $F$ and $G$ are two such $d$-tuples of rational functions, then the functions $\chi(F(t))$ and $\addchar(G(t))$ are well-defined on the domains of $F$ and $G$.

\begin{mylemma}
\label{lemma:characterSumPrimePowerEvenExponent2}
Let $\gamma= (\gamma_1, \ldots, \gamma_d)$ with each $\gamma_i \geq 2\alpha \geq 2$. Write $\beta = 2\alpha$. Let $p$ be an odd prime, $\chi$ a Dirichlet character modulo 
$(p^{\gamma_1}, \dots, p^{\gamma_d})$ , $\addchar$ an additive character modulo $(p^{\gamma_1}, \dots, p^{\gamma_d})$ , and $F,G$ as above.  
Define the congruence condition
\begin{equation}
\label{eq:linearConditionVanishingSumModified}
 \ell_{\chi} (\log f )'(p^{\gamma-\beta} t_0) + a_{\addchar} g'(p^{\gamma-\beta} t_0) \equiv 0 \shortmod {p^\alpha}.
\end{equation}
We have 
\begin{equation}
\label{eq:SdefFG1}
 S := \sumstar_{t \in (\mz/p^\beta \Z)^n} \chi(F(t)) \addchar(G(t)) = p^{n\alpha} 
 \sumstar_{\substack{t_0 \in \left( \Z/p^\alpha \Z \right)^n 
 \\ 
 \eqref{eq:linearConditionVanishingSumModified} \text{ holds}
 }
 } \chi(F(t_0)) \addchar(G(t_0)).
\end{equation}

  \end{mylemma} 
\begin{proof}
The proof is very similar to that of Lemma \ref{lemma:characterSumPrimePowerEvenExponent}.
We have
$$
\chi(F(t_0 + p^{\alpha} t_1)) = \chi(F(t_0)) \chi(\frac{F(t_0 + p^{\alpha} t_1)}{F(t_0)})
=
\chi(F(t_0)) e_{p^{\gamma}}(\ell_{\chi} \frac{F'}{F}(t_0) p^{\alpha} t_1),
$$
and similarly
$$
\addchar(G(t_0 + p^{\alpha} t_1)) = \addchar(G(t_0)) \addchar(a_{\addchar}  G'(t_0) p^{\alpha} t_1).
$$
Therefore,
$$
S = p^{n \alpha} \sumstar_{\substack{t_0 \shortmod{p^{\alpha}} \\ \eqref{eq:congruencecondition} \text{ holds}}} \chi(F(t_0))\addchar(G(t_0)),
$$
where \eqref{eq:congruencecondition} is the congruence condition
\begin{equation}\label{eq:congruencecondition}
\ell_{\chi} \frac{F'}{F}(t_0) p^{\alpha} t_1 + a_{\addchar} G'(t_0) p^{\alpha} t_1
\equiv 0 \pmod{p^{\gamma}}.
\end{equation}
Note that $G'(t_0) = p^{\gamma-\beta} f'(p^{\gamma-\beta} t_0)$, and likewise
$\frac{F'}{F}(t_0) = p^{\gamma-\beta} (\log f)'(p^{\gamma-\beta} t_0)$.  Hence the congruence condition \eqref{eq:congruencecondition} is seen to be the same as \eqref{eq:linearConditionVanishingSumModified}.
\end{proof}

   Similarly, the generalization of Lemma \ref{lemma:characterSumPrimePowerOddExponent} is given by: 
\begin{mylemma}
\label{lemma:characterSumPrimePowerOddExponent2}
Let $\gamma= (\gamma_1, \ldots, \gamma_d)$ with each $\gamma_i \geq \alpha \geq 1$. Write $\beta = 2\alpha+1$. Let $p$ be an odd prime, $\chi$ a Dirichlet character modulo 
$(p^{\gamma_1}, \dots, p^{\gamma_d})$, $\addchar$ an additive character modulo $(p^{\gamma_1}, \dots, p^{\gamma_d})$, and $F,G$ as above.  
We have 
\begin{equation}
\label{eq:SdefFG2}
 S := \sumstar_{t \in (\mz/p^\beta \Z)^n} \chi(F(t)) \addchar(G(t)) \\ = p^{n\alpha}
 \sumstar_{\substack{t_0 \in \left( \Z/p^\alpha \Z \right)^n 
 \\ 
 \eqref{eq:linearConditionVanishingSumModified} \text{ holds}
 }
 } 
 \chi(F(t_0)) \addchar(G(t_0))
  G_p(Q,L),
\end{equation}
  where
\begin{equation*}
L = p^{-\alpha} (\ell_{\chi} (\log f)'(p^{\gamma- \beta} t_0) + a_{\addchar} g'(p^{\gamma-\beta} t_0)),
\end{equation*}  
and $Q$ is the quadratic form with associated matrix (in the standard basis for $\mz^n$) given by
\begin{equation*}
Q=  \tfrac12 \ell_{\chi} p^{\gamma-\beta} (\log f)''(p^{\gamma-\beta} t_0) + \tfrac12 p^{\gamma-\beta} g''(p^{\gamma-\beta}t_0).
\end{equation*}
  \end{mylemma}
Since the proof is similar to those of Lemmas \ref{lemma:characterSumPrimePowerOddExponent} and \ref{lemma:characterSumPrimePowerEvenExponent2}, we omit the details.
For the sake of clarity, we remark that $\ell_{\chi} p^{\gamma-\beta} (\log f'')(p^{\gamma-\beta} t_0)$ 
is shorthand for 
$$\sum_{i=1}^{d} \ell_{\chi_i} p^{\gamma_i -\beta} (\log f_i)''(p^{\gamma_i-\beta} t_0),$$
and similarly for $g''$.
It will be useful later, in the proof of Lemma \ref{lemma:someCharacterSumBoundUsefulforHhat}, to observe that if $\gamma_i > \beta$ then the $i$-th component makes no contribution to the quadratic form $Q$.

The following lemma, with its easy proof omitted, will be helpful for solving the linear congruence in \eqref{eq:linearConditionVanishingSum} in future applications.
\begin{mylemma}
\label{lemma:linearalgebra}
Let $R$ be a commutative ring, with group of units $R^{\times}$.  Let $M = (a_{ij}) \in M_{2 \times 2}(R)$ with $a_{ij} \in R^{\times}$ for all $i,j$.  Then there is a solution to $(x_1,x_2).M = (0,0)$ with $x_1, x_2 \in R^{\times}$ if and only if $\det(M) = 0$, in which case the solutions are given by $x_1 a_{11} + x_2 a_{21} = 0$ (whence $x_1 = a_{21} r$, $x_2 = -a_{11} r$, for some $r \in R^{\times}$).
\end{mylemma}

\subsection{Application}
 In \cite[Conj.\ 6.6]{PetrowYoung}, we left the estimation of a certain character sum as a conjecture, which we prove here aided by Lemmas \ref{lemma:characterSumPrimePowerEvenExponent2} and \ref{lemma:characterSumPrimePowerOddExponent2}.  

 \begin{mylemma}
 \label{lemma:someCharacterSumBoundUsefulforHhat}
Let $p$ be an odd prime, let $\chi$ be a multiplicative character of conductor $p^\gamma$, $\gamma \geq 2$, and suppose $\psi$ is a multiplicative character with conductor $p^\beta$, $1 \leq \beta < \gamma$.  Then
 \begin{equation}
 \label{eq:HhatCharacterSumBoundIntermediateCase}
   \sum_{y \shortmod{p^\beta}}
 \sum_{u\shortmod{p^\beta}}
  \overline{\chi}( u p^{2(\gamma-\beta)} y + 1)
\chi(1+p^{\gamma-\beta} y ) \chi(1 - p^{\gamma-\beta} u) \psi(u)
 \psi(y)  \ll p^{\beta}.
 \end{equation}
 \end{mylemma}
\begin{proof}
 This is an instance of $S$ defined by \eqref{eq:SdefFG1}, \eqref{eq:SdefFG2}, where
$d=2$, $\gamma_1 = \gamma$, $\gamma_2 = \beta$,
 \begin{equation*}
  (f_1(y,u), f_2(y,u)) = \Big( \frac{(1+y)(1-u)}{1+yu}, yu \Big),
  \qquad
  (\chi_1, \chi_2) = (\chi, \psi),
 \end{equation*}
 $F_1(y,u) = f_1(p^{\gamma-\beta}y, p^{\gamma-\beta} u)$, $F_2(y,u) = f_2(y,u)$, and
of course the additive character is not present.  
A short calculation shows 
\begin{equation*}
 (\log f)'(p^{\gamma-\beta} t) = \begin{pmatrix}
              \frac{1-p^{\gamma-\beta}u}{(1+p^{\gamma-\beta}y)(1+p^{2(\gamma-\beta)}y u)} & \frac{-1-p^{\gamma-\beta}y }{(1-p^{\gamma-\beta}u)(1+p^{2(\gamma-\beta)}y u)} \\
              y^{-1} & u^{-1}
             \end{pmatrix}.
\end{equation*}

The summations in Lemmas \ref{lemma:characterSumPrimePowerEvenExponent2} and \ref{lemma:characterSumPrimePowerOddExponent2} run  over $t_0$ such that $$\ell_{\chi} (\log f)'(p^{\gamma-\beta} t_0) \equiv 0 \pmod{p^\alpha},$$ where $\alpha=\beta/2$ for $\beta$ even and $\alpha=(\beta-1)/2$ for $\beta \geq 3$ odd, so we write $t_0 = (u_0,y_0)$ and work out what this means in terms of conditions on $u_0$ and $y_0$. 
  Some simple algebra (cf. Lemma \ref{lemma:linearalgebra}) shows that this reduces to $u_0 \equiv - y_0 \pmod{p^\alpha}$, which uniquely determines $u_0$ in terms of $y_0$, and then
\begin{equation}
\label{eq:y0conditionmakinglogderivativevanish}
\frac{\ell_{\psi}}{y_0} \equiv \frac{-\ell_{\chi}}{1-p^{2(\gamma-\beta)} y_0^2} \pmod{p^\alpha},
\end{equation}
which uniquely determines $y_0 \pmod{p^\alpha}$, by Hensel's lemma.  Hence, when $\beta$ is even, $|S| \leq p^{\beta}$, by Lemma \ref{lemma:characterSumPrimePowerEvenExponent2}, giving the bound \eqref{eq:HhatCharacterSumBoundIntermediateCase}.

 Now consider the case that $\beta = 2\alpha+1 \geq 3$ is odd; it was already shown above that $u_0$ and $y_0$ are uniquely determined modulo $p^\alpha$, so the only remaining question is the size of the Gauss sum $G_p(Q,L)$.  It is easy to see that $Q$ is non-singular, since only the $f_2$-aspect enters into the calculation , and the Hessian of $\log{f_2}$ is diagonal with entries $-y_0^{-2}, -u_0^{-2}$.  Therefore, $|G_p(Q,L)| = p$, and \eqref{eq:HhatCharacterSumBoundIntermediateCase} follows immediately.
 
 Finally, we consider the case $\beta=1$.  In this case, we have $\chi(1+p^{\gamma-1}x) = e_{p}(\ell_{\chi} x)$, for any $x \in \mz$, so it is easy to directly evaluate \eqref{eq:HhatCharacterSumBoundIntermediateCase} as a product of two Gauss sums, giving the desired bound.
\end{proof}

 \subsection{The case $p=2$}

 The previous work in this section largely assumed $p \neq 2$.  The case $p=2$ has some minor differences, and for clarity we treat this case separately.
\begin{mylemma}
\label{lemma:PostnikovEvenp}
 Let $p=2$, and $\beta \geq 3$.  There exists a unique 
  group homomorphism $\ell: \widehat{(\mz/p^\beta \mz)^{\times}} \rightarrow \mz/p^{\beta-2} \mz$, $ \chi \mapsto \ell_{\chi}$, such that the Postnikov formula holds: for each Dirichlet character $\chi$ modulo $p^\beta$  and $t\in \Z$ we have
 \begin{equation}
 \label{postnikovEvenp} \chi(1+4t) = e_{p^\beta}(\ell_\chi \log_2(1+4t)).
 \end{equation}
 The map $\ell$ is surjective, and
 for $2 \leq \alpha \leq \beta$  we have that $\ell_{\chi_1} \equiv \ell_{\chi_2} \pmod {2^{\beta-\alpha}}$ if and only if 
  $\chi_1 \overline{\chi_2}$ is a character modulo $2^{\alpha}$.
 \end{mylemma} 
The proof is very similar to the case $p > 2$, so we give only a brief outline of the proof.  Using the notation $U_{\alpha}$ from the proof of Lemma \ref{lemma:Postnikov}, define $f:U_2 \to S^1$ by $f(t) = e_{2^{\beta}}(\log_p(t))$.  One easily checks that $f$ is well-defined and has order $2^{\beta-2}$, so $\widehat{U_2}$ is cyclic generated by $f$.  Therefore, \eqref{postnikovEvenp} holds for some $\ell_{\chi}$. The final statement of the lemma is easy to check.

 \begin{mylemma}
 Let $p=2$.  Let 
$\chi$ be a Dirichlet character modulo $p^{\beta}$, and $\psi$ be an additive character modulo $p^{\beta}$,
where $\beta \geq 3$.
Let $f, g \in \Z(t_1,\ldots, t_n)^d$  as in Lemma \ref{lemma:characterSumPrimePowerEvenExponent}. 
Let $\alpha = \lfloor \frac{\beta-1}{2} \rfloor$.
Then 
\begin{equation}
\label{eq:characterSumPrimePowerEvenExponentp=2}
S:=\sumstar_{t \in \left( \Z/p^{\beta} \Z \right)^n } \chi(f(t)) \psi(g(t))  = p^{n\alpha}\sumstar_{\substack{t_0 \in \left( \Z/p^{\beta-\alpha} \Z \right)^n \\ \ell_{\chi} (\log f )'(t_0) + a_{\psi} g'(t_0) \equiv 0 \shortmod {p^\alpha}}} \chi(f(t_0)) \psi(g(t_0)),
\end{equation}
 where the star indicates that the sum runs over numbers for which $f_i(t) \in \mz_p^{\times}, g_i(t) \in \mz_p$.  The right hand side does not depend on the choice of lifts of $f(t_0)$ and $g(t_0)$ to $\Z_p^d$.
  \end{mylemma}
 Remark.  For our later purposes, this result is a suitable replacement for Lemmas \ref{lemma:characterSumPrimePowerEvenExponent} and \ref{lemma:characterSumPrimePowerOddExponent}.  In practice, the linear congruence almost entirely determines $t_0 \pmod{p^{\alpha}}$ (which then almost entirely determines $t_0 \pmod{p^{\beta-\alpha}}$, since $\beta - \alpha = \alpha + O(1)$, and $p^{O(1)} = O(1)$ for $p=2$). 
\begin{proof}
 Let $t = t_0 + 2^{\beta-\alpha} t_1$.  Since $\beta \geq 3$, we have $\alpha \leq \beta- 2$, so $\frac{f(t)}{f(t_0)} \equiv 1 \pmod{4}$.  Then
 \begin{equation*}
  \chi(f(t)) = \chi(f(t_0)) \chi(f(t)/f(t_0)) = 
  \chi(f(t_0))
  e_{2^{\beta}}(\ell_{\chi} \log_2(f(t)/f(t_0))).
 \end{equation*}
 Next we note
 \begin{equation*}
  \log_2(f(t)/f(t_0)) \equiv (\log f)'(t_0) 2^{\beta - \alpha} t_1
  \pmod{2^{\beta}},
 \end{equation*}
 under the assumption $2(\beta-\alpha) - 1 \geq \beta$, equivalently, $\alpha \leq \frac{\beta-1}{2}$.  Note $\frac{\beta-1}{2} \leq \beta -2$ since $\beta \geq 3$.  The rest of the proof then proceeds exactly as in Lemma \ref{lemma:characterSumPrimePowerEvenExponent}.
 \end{proof}

\begin{mylemma}
 The bound in Lemma \ref{lemma:someCharacterSumBoundUsefulforHhat} holds for $p=2$.
\end{mylemma}
The proof is similar to the odd $p$ case, so we omit the details.

\section{The behavior of $g(\chi,\psi)$}
\label{section:gchipsiBehavior}

\subsection{Introductory lemmas}
Let $A \in \mz$, and let
\begin{equation*}
 Q(x) =  Q_A(x) = x^2 + Ax - 1 \in \mz[x].
\end{equation*}
For an odd prime $p$ and integer $\beta \geq 1$, define
\begin{equation*}
 r(A,p^\beta) = \# \{ x \mymod{p^\beta} : Q_A(x) \equiv 0 \mymod{p^\beta} \}.
\end{equation*}
Let $\Delta = A^2 + 4$ be the discriminant of $Q_A$.  By completing the square, note
\begin{equation}
\label{eq:QDeltaRelation}
 Q(x) = (x + \tfrac{A}{2})^2 - \tfrac{\Delta}{4}.
\end{equation}
We then have $r(A,p^\beta) = \rho(\Delta,p^\beta)$, where
\begin{equation*}
\rho(\Delta,p^\beta)  := \# \{ x \mymod{p^\beta} : x^2 \equiv \Delta \mymod{p^\beta} \}.
\end{equation*}

\begin{mylemma}
\label{lemma:nuDeltabound}
 Let $p$ be an odd prime, and $\beta \geq 1$.  If $p \nmid \Delta$, then
 \begin{equation*}
  \rho(\Delta, p^\beta) = \rho(\Delta, p)  = 1 + (\tfrac{\Delta}{p}).
 \end{equation*}
 If $p^\beta | \Delta$, then
\begin{equation*}
  \rho(\Delta, p^\beta) = 
  p^{\lfloor \beta/2 \rfloor} = 
  \begin{cases}
    p^{\frac{\beta}{2}}, \qquad &\text{$\beta$ even},\\
    p^{\frac{\beta-1}{2}}, \qquad &\text{$\beta$ odd}.
  \end{cases}
\end{equation*}
If $p| \Delta$, but $p^\beta \nmid \Delta$, then
\begin{equation}
\label{eq:rhoDeltaIntermediateCase}
 \rho(\Delta, p^\beta) \leq 2 (\Delta, p^\beta)^{1/2} \delta((p^\beta, \Delta) = \square).
\end{equation}
\end{mylemma}
\begin{proof}
The case $p \nmid \Delta$ follows from Hensel's lemma.  The conclusion when $p^\beta| \Delta$ is easy to verify directly.

Now suppose $(\Delta, p^\beta) = p^{\alpha}$, with $1 \leq \alpha < \beta$. 
Write $\Delta = p^{\alpha} \Delta'$ with $(\Delta', p ) = 1$.  
It is easy to see that if $\alpha$ is odd then $\rho(\Delta, p^\beta) = 0$.  If $\alpha$ is even (which means $(p^\beta, \Delta) = \square)$ then we write $x = p^{\alpha/2} x_1$, say, where $x_1$ runs modulo $p^{\beta-\frac{\alpha}{2}}$.  Then $x_1$ solves the congruence
\begin{equation*}
 x_1^2 \equiv \Delta' \pmod{p^{\beta-\alpha}}.
\end{equation*}
By Hensel's lemma, there are $1+(\frac{\Delta'}{p})$ solutions $x_1 \pmod{p^{\beta-\alpha}}$ to this congruence, and so 
 in total there are most $2 p^{\alpha/2}$ values of $x_1$ modulo $p^{\beta-\frac{\alpha}{2}}$, giving \eqref{eq:rhoDeltaIntermediateCase}.
\end{proof}

\subsection{The bounds on $g(\chi,\psi)$}
Recall that $g(\chi,\psi)$ is defined by \eqref{eq:gdef}, and that both $\chi$ and $\psi$ are primitive characters modulo $q=p^\beta$.  Anticipating some future simplifications, we apply the simple change of variables $t \rightarrow t-1$ and $u \rightarrow u-1$ giving
\begin{equation}
\label{eq:gdefAlternateVersion}
g(\chi,\psi) = \sumstar_{t,u \shortmod{p^\beta}} \chi\Big(\frac{u(t-1)}{t(u-1)}\Big) \psi(ut-t-u),
\end{equation}
where we recall that the asterisk on the sum means that the sum is restricted to $u,t$ such that the denominator of $\frac{u(t-1)}{t(u-1)}$ is coprime to $p$.

\begin{myremark}\label{rem:gboundoddestprime}
Note that if $p=2$ and $q=p^{\beta}$, $\beta \geq 1$, then $g(\chi, \psi)$ trivially vanishes, since $t(t+1)$ is even for all $t \in \mz$. 
\end{myremark}

\begin{mytheo}
\label{thm:gboundevencase}
Suppose $q=p^\beta$ with $p$ odd and $\beta=2\alpha$.  Then
\begin{equation*}
|g(\chi,\psi)| \leq q \rho(\Delta, p^\alpha),
\end{equation*}
where $\Delta = A^2 + 4$ and $A \equiv \ell_{\chi} \overline{\ell_{\psi}} \pmod{p^{\beta-1}}$. 
\end{mytheo}

\begin{mytheo}
\label{thm:gboundoddcase}
Suppose $q=p^\beta$ with $p$ odd and $\beta=2\alpha+1$, $\alpha \geq 1$.  Then
\begin{equation}
\label{eq:gboundoddcase}
|g(\chi,\psi)| \leq 
\begin{cases}
2q, \qquad &p \nmid \Delta, \\
q p^{1/2} \delta(p^2|\Delta)  \rho(p^{-2} \Delta, p^{\alpha-1}), \qquad &p | \Delta.
\end{cases}
\end{equation}
where $\Delta = A^2 + 4$  with $A \equiv \ell_{\chi} \overline{\ell_{\psi}} \pmod{p^{\beta-1}}$.
\end{mytheo}

\begin{proof}[Proof of Theorem \ref{thm:gboundevencase}]
The sum \eqref{eq:gdefAlternateVersion} falls into the template of Lemma \ref{lemma:characterSumPrimePowerEvenExponent}, with
\begin{equation*}
(f_1(t,u), f_2(t,u)) = \Big(\frac{u(t-1)}{t(u-1)}, tu-t-u \Big), \qquad (\chi_1,\chi_2) = (\chi, \psi).
\end{equation*}
No additive character is present, of course.  A short calculation gives
\begin{equation}
\label{eq:logderivativeAppearingInProofOfgbound}
(\log f)' = 
\begin{pmatrix}
 \frac{1}{t(t-1)} & \frac{-1}{u(u-1)} \\ \frac{u-1}{ut-t-u} &  \frac{t-1}{ut-t-u}
\end{pmatrix}.
\end{equation}
Note that the vanishing (mod $p^\alpha$) of the determinant of $(\log f)'$ is equivalent to 
\begin{equation}
\label{eq:utrelation}
u \equiv - t \pmod{p^\alpha},
\end{equation}
and that 
\begin{equation}
\label{eq:logderivativeAppearingInProofOfgboundSimplified}
(\log f)'\vert_{u=-t} = 
\begin{pmatrix}
 \frac{1}{t(t-1)} & \frac{-1}{t(t+1)} \\ \frac{t+1}{t^2} &  \frac{1-t}{t^2}
\end{pmatrix}.
\end{equation}

By Lemma \ref{lemma:linearalgebra}, 
the condition \eqref{eq:linearConditionVanishingSum} is seen to be equivalent to \eqref{eq:utrelation} combined with
\begin{equation}
\label{eq:linearConditiongchipsiPreliminary}
\frac{\ell_{\chi}}{t(t-1)} + \ell_{\psi} \frac{t+1}{t^2}  \equiv 0 \pmod{p^\alpha}.
\end{equation}
Simplifying \eqref{eq:linearConditiongchipsiPreliminary}, we obtain the equivalent congruence
\begin{equation}
\label{eq:linearConditiongchipsi}
t^2+A t   -1 \equiv 0 \pmod{p^\alpha}, \qquad A \equiv \ell_{\chi} \overline{\ell_{\psi}} \pmod{p^\alpha}.
\end{equation}
Hence, $|g(\chi,\psi)| \leq q \rho(\Delta, p^\alpha)$, as claimed.
\end{proof}

\begin{proof}[Proof of Theorem \ref{thm:gboundoddcase}]
The beginning steps of the proof are identical to those of Theorem \ref{thm:gboundevencase};  the linear congruences in both cases are the same, so we obtain that 
$$g(\chi, \psi) = p^{2\alpha} \sumstar_{\substack{t_0, u_0 \shortmod{p^\alpha} \\\text{\eqref{eq:utrelation} and \eqref{eq:linearConditiongchipsi} hold} }} \chi(f_1(t_0, u_0)) \psi(f_2(t_0, u_0)) G_p(Q,L),$$ 
 What is new is the presence of the quadratic Gauss sum $G_p$, so we next focus on this aspect.  Note that the quadratic form $Q$ present in $G_p(Q,L)$ is given with respect to the standard basis by
\begin{equation}\label{2Qdefs3}
2Q = \ell_{\chi} (\log f_1)'' + \ell_{\psi} (\log f_2)''
\end{equation}

Working in $\F_p$ until further notice, the Hessian of $\log{f_1}$ is
\begin{equation*}
\begin{pmatrix}
t^{-2} - (t-1)^{-2} & 0 \\ 0 & -u^{-2} + (u-1)^{-2}
\end{pmatrix}
=
\begin{pmatrix}
\frac{-2t+1}{t^2(t-1)^2} & 0 \\ 0 & \frac{-2t-1}{t^2(t+1)^2}
\end{pmatrix}
,
\end{equation*}
by simplifying with \eqref{eq:utrelation}.  The Hessian of $\log f_2$ is
\begin{equation*}
\begin{pmatrix}
-\frac{(u-1)^2}{(tu-t-u)^2} & \frac{-1}{(tu-t-u)^2} \\ \frac{-1}{(tu-t-u)^2} & -\frac{-(t-1)^2}{(tu-t-u)^2}
\end{pmatrix}
=
\begin{pmatrix}
\frac{-(t+1)^2}{ t^4} & -t^{-4} \\ -t^{-4} & \frac{-(t-1)^2}{t^4}
\end{pmatrix}.
\end{equation*}
Therefore, 
\begin{equation}
\label{eq:QformulaInProof}
2Q = \frac{-\ell_{\psi}}{t^2} 
\left[
\begin{pmatrix}
\frac{(t+1)^2}{ t^2} & t^{-2} \\ t^{-2} & \frac{(t-1)^2}{t^2}
\end{pmatrix}
+
\begin{pmatrix}
\frac{A(2t-1)}{(t-1)^2} & 0 \\ 0 & \frac{A(2t+1)}{(t+1)^2}
\end{pmatrix}
\right].
\end{equation}
Using a computer algebra package, we evaluate the determinant of the expression in square brackets above as
\begin{equation*}
1 + \frac{4A}{t} - \frac{2}{t^2} + \frac{A^2(4t^2 -1)}{(t^2-1)^2} = \frac{1}{t^2}(5 t^2 +4At - 3) = \frac{2-At}{t^2},
\end{equation*}
using $t^2-1 = -A t $.
Therefore, the determinant vanishes if and only if $t = 2/A$.

By Lemma \ref{lemma:quadraticGaussSum},  to determine the size of $|G_p(Q,M)|$ we need the rank of $Q$.
It is clear from \eqref{eq:QformulaInProof} that $Q$ does not have rank $0$.  Therefore, $Q$ 
has rank $1$ if the determinant vanishes, and 
rank $2$ otherwise.  

Next we note that the two algebraic equations $t^2 +At - 1=0$ and $t =2/A$ have a common solution in $\mf_p$ if and only if $A^2 + 4 =0$ in $\mf_p$, i.e. $p|\Delta$.
Hence, if $p \nmid \Delta$, then $|G_p(Q,M)| = p$, and so $|g(\chi,\psi)| \leq q \rho(\Delta, p^\alpha) \leq 2q$, as desired.
If $p|\Delta$, then $Q$ has rank $1$, so we obtain
\begin{equation*}
|g(\chi,\psi)| \leq p^{2\alpha} \sum_{\substack{t_0 \shortmod{p^\alpha} \\ t_0^2 + At_0 - 1 \equiv 0 \shortmod{p^\alpha}}} p^{3/2} = q p^{1/2} \rho(\Delta, p^\alpha). 
\end{equation*}

This bound is not as strong as \eqref{eq:gboundoddcase}; we will next gain some extra information by studying the behavior of the linear form $L$ restricted to the isotropic subspace of $Q$.
Note that $\Delta=0$ means $A^2 = -4$, whence $t=-A/2 = 2/A$ and so $t^2 = -1$.  Therefore $(t+1)^2 = 2t$ and $(t-1)^2 = -2t$,
and we can simplify \eqref{eq:QformulaInProof} as
\begin{equation*}
2Q = \ell_{\psi}
\left[
\begin{pmatrix}
-2t & -1\\ -1 & 2t
\end{pmatrix}
+
\begin{pmatrix}
\frac{A(2t-1)}{-2t} & 0 \\ 0 & \frac{A(2t+1)}{2t}
\end{pmatrix}
\right]
= \ell_{\psi} 
\begin{pmatrix}
-1 & -1 \\ -1 & -1
\end{pmatrix}.
\end{equation*}
Hence the isotropic subspace of $Q$ is spanned by the vector $(1,-1)^\intercal$.

Next we work out an easily-checked characterization for the linear form $L$ to be trivial on this isotropic subspace. By \eqref{eq:Mlinearformdef}, \eqref{eq:logderivativeAppearingInProofOfgboundSimplified}, and the above calculation of the isotropic subspace, this means
\begin{equation*}
p^{-\alpha}\Big(\frac{\ell_{\chi}}{t-1} + \frac{\ell_{\psi}(1+t)}{t}\Big)
- p^{-\alpha} \Big(\frac{-\ell_{\chi}}{t+1} + \frac{\ell_{\psi}(1-t)}{t}\Big) \equiv 0 \pmod{p},
\end{equation*}
which reduces to $t$ satisfying
\begin{equation*}
t^2 + At - 1 \equiv 0 \pmod{p^{\alpha+1}}.
\end{equation*}
Thus the number of $t_0$ to be estimated is
\begin{equation}
\label{eq:t0count}
\# \{ t_0 \mymod{p^\alpha} : (t_0 + A/2)^2 \equiv \Delta \mymod{p^{\alpha+1}},
\end{equation}
and we study this a bit more closely (which along the way will confirm this quantity is well-defined).
This count equals $\# \{ x \pmod{p^\alpha} : x^2 \equiv \Delta \pmod{p^{\alpha+1}}\}$.  Since $p | \Delta$, then $p|x$ also, so \eqref{eq:t0count} equals $\# \{ x_1 \pmod{p^{\alpha-1}} : x_1^2 \equiv \frac{\Delta}{p^2} \pmod{p^{\alpha-1}} \}$, which is well-defined.  Therefore, we obtain a more refined bound
\begin{equation*}
|g(\chi, \psi)| \leq q p^{1/2} \delta(p^2|\Delta)  \rho(p^{-2} \Delta, p^{\alpha-1}). \qedhere
\end{equation*}
\end{proof}

\section{Bounding the cubic moment}
\label{section:ZpropertiesandtheCubicMomentBound}
In this section, we prove \cite[Conj.\ 8.2]{PetrowYoung} assuming Theorem \ref{thm:fourthmoment}. Conjecture 8.2 of \cite{PetrowYoung} implies the cubic moment bounds (Theorems \ref{thm:mainthmMaassEisenstein} and \ref{thm:mainthmHybridVersion}) and hence the Weyl bound (Theorem \ref{thm:WeylBound}). The precise statement of \cite[Conj.\ 8.2]{PetrowYoung} appears as Lemma \ref{lemma:ZpropertiesCubic}, below. The proof of Theorem \ref{thm:fourthmoment} is deferred to Sections \ref{section:harmonicanalysis} and \ref{section:spectralanalysisofshiftedsum}.

We begin by reviewing the notation and re-stating this conjecture.   We have a Dirichlet series
\begin{equation}
\label{eq:ZdefCubic}
 Z(s_1, s_2, s_3, s_4) = \frac{1}{\varphi(q)} \sum_{\psi \shortmod{q}} \frac{L(s_1, \psi) L(s_2, \psi) L(s_3, \psi) L(s_4, \overline{\psi})}{\zeta^{(q)}(s_1 + s_4)} Z_{\text{fin}},
\end{equation}
where 
$Z_{\text{fin}} = Z_{\text{fin}}(\chi, \psi)$, 
$|Z_{\text{fin}}| = |\prod_{p | q} Z_{\text{fin}, p}|$, and
$Z_{\text{fin},p}$ is a certain Dirichlet series supported on powers of $p$.  Its precise definition is not necessary here, but rather we quote Lemma 7.1 from \cite{PetrowYoung}.
\begin{mylemma}
\label{lemma:ZfinBound}
Let $q=p^\beta$, and let $\chi=\chi_p$ be primitive modulo $q$. The series $Z_{{\rm fin}, p}$ converges absolutely when $\real(s_j) =\sigma_j >0$ for all $j=1,2,3,4$.
If $\sigma_j \geq \sigma > 1/2$ for all $j$, then
\begin{equation}
\label{eq:ZfinHalfLineBound}
Z_{{\rm fin},p}(s_1, s_2, s_3, s_4) \ll_{\sigma} \delta_{\psi} q^{1/2} |g(\chi,\psi)| + q^{3/2+\varepsilon},
\end{equation}
where $\delta_{\psi} = 1$ if $\psi$ is primitive, and $0$ otherwise.
If $\sigma_j \geq \sigma > 1$ for all $j$, and $\psi_p$ is the trivial character, then
\begin{equation}
\label{eq:ZfinOneLineBound}
Z_{{\rm fin},p}(s_1, s_2, s_3, s_4) \ll_{\sigma} q^{1+\varepsilon}.
\end{equation}
\end{mylemma}
We remark that Lemma \ref{lemma:ZfinBound} appeared as \cite[Lem. 7.1]{PetrowYoung}, however there it was conditional on \cite[Conj. 6.6]{PetrowYoung} which has been proved here as Lemma \ref{lemma:someCharacterSumBoundUsefulforHhat}.

As in \cite{PetrowYoung}, it is helpful to treat the trivial character separately.  To this end, write $Z = Z_0 + Z_1$, where $Z_0$ is the contribution to $Z$ from the trivial character.
Now we state the main lemma.
\begin{mylemma}
\label{lemma:ZpropertiesCubic}
The functions $Z_0$ and $Z_1$ satisfy the following properties.  Firstly,
$Z_0$ is meromorphic for $\real(s_j) \geq \sigma > 1/2$ for all $j$ and
analytic for $\real(s_j) \geq \sigma >  1$ for all $j$.  
In this domain, it may only have polar divisors on the hyperplanes $s_j = 1$. 
In the region $\real(s_j) \geq \sigma >  1$ it satisfies the bound
\begin{equation}
\label{eq:Z0boundCubic}
Z_0(s_1, s_2, s_3, s_4) \ll_{\sigma} q^{\varepsilon}.
\end{equation}

Secondly, $Z_1$ is analytic for $\real(s_j) \geq \sigma > 1/2$ for all $j$, wherein it satisfies the bound 
\begin{equation}
\label{eq:Z1boundIntegralVersion}
\int_{-T}^{T} |Z_1(\sigma+it, \sigma+it, \sigma+it, \sigma-it)| dt \ll q^{3/2+\varepsilon} T^{1+\varepsilon},
\end{equation}
for $T \geq 1$.  
The same bounds stated for $Z_1$ also hold for $Z_0$ (in an even stronger form), provided $1/2 \leq \text{Re}(s_j) \leq 0.99$. 
\end{mylemma}

Remark.  
Theorem \ref{thm:fourthmoment} is the crucial new ingredient in the proof of Lemma \ref{lemma:ZpropertiesCubic} in the case that $q$ is not cube-free.
\begin{proof} 
  The holomorphic (resp. meromorphic) continuation of $Z_1$ (resp. $Z_0$) follows from the definition \eqref{eq:ZdefCubic} and Lemma \ref{lemma:ZfinBound}.  All the required properties of $Z_0$ follow from Lemma \ref{lemma:ZfinBound}, so we now focus on $Z_1$.

Supposing that $\real(s_j) > 1/2$ for $j=1,2,3,4$, we have 
\begin{equation}\label{Z1bound1}
Z_1(s_1,s_2,s_3,s_4)   \ll \frac{(1+|t|)^\eps}{q^{1-\varepsilon}} \sum_{\psi \neq \psi_0} | L(s_1,\psi)L(s_2,\psi)L(s_3,\psi)L(s_4,\overline{\psi})| \prod_{p \mid q} |Z_{{\rm fin},p}|.
\end{equation}
From Lemma \ref{lemma:ZfinBound} we have for $\chi_p,\psi_p$ the $p$-parts of $\chi, \psi$ respectively, each modulo $p^\beta$,
\begin{equation*}
|Z_{{\rm fin},p}| \ll p^{(\frac{3}{2}+\eps)\beta} \left(
 \delta_{\psi_p} \frac{|g(\chi_p,\psi_p)|}{p^\beta} +1 
\right).
\end{equation*}
Recall from Remark \ref{rem:gboundoddestprime} that $g(\chi_p,\psi_p) = 0$ if $p=2$, so for the forthcoming analysis of $Z_{\text{fin}, p}$ we largely assume $p$ is odd.
In \cite[Thm.\ 6.9]{PetrowYoung} it was shown that if $\beta =1$ then $|g(\chi_p,\psi_p)| \leq Cp$ for some absolute constant $C\geq 2$. On the other hand, when $\beta \geq 2$, we see from Theorems \ref{thm:gboundevencase} and \ref{thm:gboundoddcase} that $|Z_{{\rm fin},p}|$ is controlled by the quantity $\Delta_p  = A^2+4  = (\ell_{\chi} \overline{ \ell_{\psi}})^2 +4 \pmod{p^{\beta-1}}$ where $p^\beta \| q$. 
Therefore it is natural to parametrize the sum in \eqref{Z1bound1} over the possible values of the parameters $\Delta_p$. To this end, for $\beta-1 \geq \alpha \geq 0$ and $C$ the above absolute constant, let 
\begin{equation*}
m(\alpha,\beta) = \inf \{m \in \tfrac{1}{2}\Z:  \max_{\substack{ \psi_p \shortmod{p^\beta} \text{ prim.} \\ v_p(\Delta_p)=\alpha}} \frac{|g(\chi_p,\psi_p)|}{p^\beta} \leq Cp^m\},
\end{equation*}
which depends on $p$ and $\chi_p$, but we suppress this from the notation.  For $p=2$, $m(\alpha,\beta) =  -\infty$.
For $a \mid \frac{q}{\tilde{q}}$ with $\tilde{q} = \prod_{p \mid q}p$, let 
\begin{equation*}
 M(a, q) = \prod_{p^\beta  \| q} p^{m(\alpha,\beta)}, \quad \text{ where } \quad \alpha = v_p(a).
\end{equation*}
Write $\Delta = \Delta(\psi) \in [1, \frac{q}{\tilde{q}}]\subset  \Z$ 
with $\Delta \equiv \Delta_p \pmod{p^{\beta-1}}$ for each $p \mid q$. Note the condition that $v_p(\Delta_p)=\alpha$ for all $p \mid q$ is equivalent to $a \| \Delta$. 
Then we have 
\begin{equation}\label{Z1bound1.5}
Z_1(s_1,s_2,s_3,s_4)  \ll  q^{\frac{1}{2}+\eps}(1+|t|)^\eps  \sum_{a \mid \frac{q}{\tilde{q}}} M(a,q)  \sum_{\psi:   \Delta(\psi) \equiv 0 \shortmod{a}} | L(s_1,\psi)L(s_2,\psi)L(s_3,\psi)L(s_4,\overline{\psi})|,
\end{equation}
where we over-extended the condition $a \| \Delta$ to $\Delta \equiv 0 \pmod{a}$.
With an eye towards applying Theorem \ref{thm:fourthmoment}, we next break up \eqref{Z1bound1.5} over cosets. Let $G =  \{\psi \pmod q  \}$, and $H_a$ be the subgroup $H_a = \{\psi \pmod{q/a} \}$.  

Lemma \ref{lemma:Postnikov} implies that $\psi, \psi' \in G$ are in the same $H_a$-coset if and only if $\ell_{\psi} \equiv \ell_{\psi'} \pmod{p^{v_p(a)}}$ for each $p|q$.  Hence if $\psi,\psi'$ are in the same $H_a$-coset, then $\Delta(\psi) \equiv \Delta(\psi') \pmod{a}$.
Thus
\begin{multline*}
Z_1(s_1,s_2,s_3,s_4) \\ \ll q^{\frac{1}{2}+\eps}(1+|t|)^\eps  
\sum_{a \mid \frac{q}{\tilde{q}}} 
M(a,q)  
\sum_{\substack{\theta \in G/H_a \\ \Delta(\theta) \equiv 0 \shortmod a}}  \sum_{\eta \in H_a} | L(s_1,\eta.\theta)L(s_2,\eta.\theta)L(s_3,\eta. \theta)L(s_4,\overline{\eta.\theta})|.
\end{multline*}
Next, we introduce an integral as in \eqref{eq:Z1boundIntegralVersion}, and apply Theorem \ref{thm:fourthmoment} to find
\begin{equation}\label{Z1bound2}
\frac{1}{T^{1+\eps}}\int_{-T}^T |Z_1|\,dt  \ll q^{\frac{1}{2}+\eps}  \sum_{a \mid \frac{q}{\tilde{q}}} M(a,q)  \sum_{\substack{\theta \in G/H_a \\ \Delta(\theta) \equiv 0 \shortmod{a}}}\lcm(q/a , q^*),
\end{equation}
where $Z_1$ is shorthand for $Z_1(\sigma+it, \sigma+it, \sigma+it, \sigma-it)$.

The right hand side of \eqref{Z1bound2} is a multiplicative function of $q$, and so is the desired bound of $q^{3/2+\eps}$, so it suffices to work with $q= p^\beta$ an odd prime power, which we henceforth assume. 
Note that there are at most two $\theta \in G/H_{p^\alpha}$ satisfying the condition $\Delta \equiv 0 \pmod {p^\alpha}$. Indeed, $\Delta \equiv 0 \pmod {p^\alpha}$ means that $\ell_\chi^2 = -4 \ell_\psi^2 \pmod {p^\alpha}$, which has at most two solutions $\ell_\psi \pmod {p^\alpha}$, since $(\ell_\chi,p)=1$. 
Thus the right hand side of \eqref{Z1bound2} takes the form
\begin{equation}\label{Z1bound3}
 p^{(\frac{1}{2}+\eps) \beta} \sum_{\alpha=0}^{\beta-1}  p^{m(\alpha,\beta)+\max(\beta-\alpha, \lceil \frac{2 \beta}{3} \rceil )}.
\end{equation}
To finish the proof of the lemma, it suffices to show for all $0 \leq \alpha \leq \beta-1$ the inequality 
\begin{equation}
\label{eq:exponentbound}
 m(\alpha, \beta) + \max(\beta-\alpha, \lceil \tfrac{2 \beta}{3} \rceil ) \leq \beta.
\end{equation}

 By \cite[Thm.\ 6.9]{PetrowYoung}, and Theorems \ref{thm:gboundevencase} and \ref{thm:gboundoddcase}, we have
\begin{equation*}
\frac{|g(\chi_p, \psi_p)|}{p^\beta} \leq
\begin{cases}
 \rho(\Delta, p^{\beta/2})  & \text{ for } \beta \text{ even}, \\
 C & \text{ for } \beta = 1, \\
  p^{1/2} \rho(\frac{\Delta}{p^2}, p^{\frac{\beta-3}{2}}) & \text{ for } \beta  \text{ odd, $\beta \geq 3$, $p^{2}|\Delta$}, \\
  2 & \text{ for } \beta  \text{ odd, $\beta \geq 3$, $p^{2}\nmid \Delta$}.
\end{cases}
 \end{equation*}
 By Lemma \ref{lemma:nuDeltabound}, we get for $\beta$ even 
 \begin{equation}
 \label{eq:bdefjeven}
 m(\alpha,\beta) \leq 
 \begin{cases} 
 -\infty  & \text{ for } \alpha \text{ odd}, \alpha < \beta/2, \\
 \alpha/2  & \text{ for } \alpha \text{ even}, \alpha < \beta/2, \\
 \lfloor \beta/4 \rfloor  & \text{ for }  \alpha \geq \beta/2,
 \end{cases}
 \end{equation}
and for $\beta$ odd,
\begin{equation*}
  m(\alpha,\beta) \leq 
 \begin{cases} 
0 & \text{ for }  \alpha = 0, \\ 
-\infty & \text{ for }  \alpha \text{ odd,  $\alpha < \frac{\beta+1}{2}$}, \\
 \frac{\alpha-1}{2} & \text{ for } \alpha \text{ even,  $2\leq \alpha < \frac{\beta+1}{2}$}, \\
 \lfloor \frac{\beta+1}{4} \rfloor -\frac12  & \text{ for } \alpha \geq \frac{\beta+1}{2}.
 \end{cases}
 \end{equation*}

We proceed to prove \eqref{eq:exponentbound}. First suppose $\beta$ is even, so $m(\alpha, \beta)$ is bounded by \eqref{eq:bdefjeven}.  
If $\alpha \geq \beta/2$, then $\max(\beta-\alpha, \lceil 2 \beta/3 \rceil) = \lceil 2 \beta/3 \rceil$, and 
it reduces to checking $\lfloor \beta/4 \rfloor + \lceil 2\beta/3 \rceil \leq \beta$.  To show this last inequality, it suffices to check it for each $\beta \in \{0,2,4,6,8,10 \}$,
which may as well be done by brute force using a computer.
For $\alpha < \beta/2$, we have $\alpha/2 + (\beta-\alpha) \leq \beta$, as well as $\alpha/2 + \lceil 2 \beta/3 \rceil \leq \lfloor \beta/4 \rfloor + \lceil 2 \beta/3 \rceil \leq \beta$, so we are done.
Similarly easy arguments hold when $\beta$ is odd, and we omit the details. 
\end{proof}

\section{Reduction of Theorem \ref{thm:fourthmoment} to Theorem \ref{thm:ShiftedSumBounds}}
\label{section:reductionofFourthMomentToShiftedSums}
In this section, we prove Theorem \ref{thm:fourthmoment}, subject to the veracity of Theorem \ref{thm:ShiftedSumBounds}.  The rest of the paper is then devoted to the proof of Theorem \ref{thm:ShiftedSumBounds}.

First note that by positivity, to prove Theorem \ref{thm:fourthmoment}, it suffices to consider the case $q^*|d$, which means $q^2|d^3$.
By an approximate functional equation, dyadic partition of unity, and Cauchy's inequality applied on the dyadic sum, it suffices to show
\begin{equation*}
\mathcal{M}(N,d,q,T) 
:=
\intR w_0\Big(\frac{t}{T}\Big) \sum_{\psi \shortmod{d}} \frac{d}{\varphi(d)}
\Big| \sum_{n} w_N(n) \tau(n) \psi(n) \chi(n) n^{-it} \Big|^2 dt 
\ll 
NdT 
(qT)^{\varepsilon},
\end{equation*}
where $w_N(x)$ is a smooth function supported on $[N,2N]$, satisfying $w_N^{(j)}(x) \ll x^{-j}$, for all $j \geq 0$, and $w_0$ is a fixed smooth nonnegative function.  Moreover, we may assume
\begin{equation}
\label{eq:Nsize}
 N \ll (qT)^{1+\varepsilon}.
\end{equation}

Opening the square and executing the $\psi$ sum and $t$-integral, we have
\begin{equation*}
\mathcal{M}(N,d,q,T) 
= 
dT \sum_{m \equiv n \shortmod{d}} \tau(m) \chi(m) \tau(n) \overline{\chi}(n) \widehat{w_0}\Big(\frac{T}{2\pi} \log\Big(\frac{m}{n}\Big)\Big) w_N(m) w_N(n),
\end{equation*}
where $\widehat{w_0}(y) = \intR w_0(t) e(-ty) dt$ is the standard Fourier transform.
The contribution from the diagonal terms $m=n$ give a main term of size $O(NdT N^{\varepsilon})$, which is acceptable.  

Next consider the off-diagonal terms.  By symmetry, it suffices to consider the terms with $m>n$, in which case we write $m=n+h$, with $h \geq 1$, and $d|h$.  By the rapid decay of $\widehat{w_0}$, the sum over $h$ may be truncated at $h \ll H$ where $H = \frac{N}{T} (Nq)^{\varepsilon}$.  By the positivity in \eqref{eq:fourthmoment}, we may also assume $T \gg (qN)^{\varepsilon}$ so that $H \ll N$.
We also open $\tau(n) = \sum_{n_1 n_2=n} 1$ and employ dyadic partitions of unity to the sums over $n_1$ and $n_2$.
Let $\mathcal{M}_1(N_1, N_2,d,q,T)$ denote the contribution of these terms to $\mathcal{M}(N,d,q,T)$, where $N_1 N_2 \asymp N$ and $n_j \asymp N_j$ for $j=1,2$.  Then
\begin{equation*}
(dT)^{-1} \mathcal{M}_1(N,d,q,T) =S(\chi),
\end{equation*}
where the weight function $w(n_1, n_2, h)$ is given by
\begin{equation*}
 w(x,y,z) = \widehat{w_0}\Big(\frac{T}{2\pi} \log\Big(\frac{xy+z}{xy}\Big)\Big) w_N(xy + z) w_N(xy) \omega_{1}(x) \omega_{2}(y),
\end{equation*}
where $\omega_1$ and $\omega_2$ are part of the dyadic partitions of unity.  It is easy to check that $w(x,y,z)$ satisfies \eqref{eq:wderivativebounds}.

Theorem \ref{thm:ShiftedSumBounds}  will complete the proof of Theorem \ref{thm:fourthmoment}, since $\frac{H}{q} \ll \frac{N}{qT} (qN)^{\varepsilon} \ll (qT)^{\varepsilon}$.

  \section{Automorphic Forms}
  \label{section:automorphicforms}
 \subsection{Fourier expansion}\label{sec:fe}
In this section we recall the Fourier expansions of automorphic forms on $\GL_2$ over $\Q$. Using canonical inner products on Whittaker models as in \cite{MichelVenkateshGL2}, we obtain particularly pleasant normalizations of Fourier expansions and Bruggeman-Kuznetsov formulas. To discuss these, we work in greater generality than is strictly required for the other sections of this paper.

 Let $(\pi,V)$ be a standard generic automorphic representation of $\GL_2/\Q$ of conductor $\cond(\pi)$ and analytic conductor $C(\pi)$ (for a definition, see \cite[3.1.8]{MichelVenkateshGL2}). By ``standard'' here we mean, following \cite[2.2.1]{MichelVenkateshGL2}, that $\pi$ occurs in the spectral decomposition of the space of automorphic forms. In particular, it is abstractly unitarizable. For $q \in \mathbb{N}$, let $$ K_1(q) = \{ \left( \begin{smallmatrix} a & b \\ c & d \end{smallmatrix} \right) \in \GL_2(\widehat{\Z}) : c \in (q), \quad  d \in 1+(q)\}\subset \GL_2(\A).$$  The subspace of $V$ consisting of right $K_1(\cond(\pi))$-invariant vectors of minimal non-negative $\SO_2(\R)$-weight is $1$-dimensional (see \cite{Casselman}, \cite[\S2.2]{DeligneAntwerpII}). Any $\phi$ belonging to this $1$-dimensional subspace is called a \emph{newvector}. In Theorem \ref{thm:fe}, we make an explicit choice of a distinguished newvector in this $1$-dimensional space using canonical inner products on Whittakers models.

   To give the precise statement, we must set up some notation. Fix $\psi: \A/\Q \to \C^\times$ the unique additive character which coincides with $e(x)$ on $\R$. For any place $v$ of $\Q$, let us denote by $\psi_v$ the restriction of $\psi$ to $\Q_v$. 
 Let $\mathbb{X} = \PGL_2(\Q) \backslash \PGL_2(\A)$.
The space $\mathbb{X}$ has finite measure, which we normalize to be probability measure. Warning: Michel-Venkatesh use the push-forward measure on $\mathbb{X}$, under which it has volume $2\xi(2)=\pi/3$, see \cite[4.1.2]{MichelVenkateshGL2}.

Given $\pi = \bigotimes_v \pi_v$ unitary, for $v=p<\infty$ let $\phi_p:\Q_p^\times \to \C$ be a local newvector for $\pi_p$ in the Kirillov model $\mathcal{W}(\pi_p,\psi_p)$, normalized so that $\phi_p(1)=1$. Explicit formulas for $\phi_p$ are well-known, see e.g.~\cite[\S 2.4 Summary]{SchmidtRemarksonLocalNewforms} for a nice presentation. In particular, $\supp(\phi_p) \subseteq \Z_p$, $\phi_p(x)$ only depends on $|x|_p$, and $\lambda_\pi(n) = |n|^{1/2} \prod_{p} \phi_p(p^{v_p(n)})$ coincides with the $n$th Hecke eigenvalue of $\pi$  normalized so that the Ramanujan conjecture predicts that $|
\lambda_\pi(p)|\leq 2$. We also have $|\phi_p(x)|\leq |x|^{1/2}$ for all $x \in \Q_p^\times$ if $\pi_p$ is ramified, and in  particular,
\begin{equation}
\label{eq:RamanujanBoundRamifiedPrimes}
 |\lambda_{\pi}(p)| \leq 1 \qquad \text{if} \qquad p \mid \cond(\pi).
\end{equation}

   Following the notation in \cite[4.1.5]{MichelVenkateshGL2}, for $L$ a meromorphic function, we write $L^*(s_0)$ for the leading coefficient in the Laurent series of $L(s)$ at $s=s_0$. 
 By the analytic continuation of Rankin-Selberg $L$-functions, the series \begin{equation}\label{naiveRS} \mathscr{L}_\pi(s) := \sum_{n\geq 1} \frac{|\lambda_\pi(n)|^2}{n^s}\end{equation} admits a meromorphic continuation to $\real(s)>1/2$, with no poles except at $s=1$.  We have \cite{IwaniecSmallValuesofLaplace, HoffsteinLockhart} 
 \begin{equation}
 \label{eq:HLIw}
 \mathscr{L}_\pi^*(1) = C(\pi)^{o(1)}.\end{equation}

Lastly, for $\phi \in \pi$ recall from \cite[\S 2.2.2]{MichelVenkateshGL2} the canonical norm $\| \phi \|_{\it can}$ on the space of $\pi$. (Note that there is a missing factor of $\xi_F^*(1)/\xi_F(2)$ on the right hand side of \cite[(2.3)]{MichelVenkateshGL2}. That this factor is missing is suggested by the notational conventions for infinite products in section 4.1.5; also compare (2.3) to e.g.\ (4.16) or (4.28) for a precise check.)

 Let $W_{\lambda, \mu}(z)$ be the Whittaker function defined and normalized as in \cite[9.220.4]{GR} and $K_\nu$ be the standard $K$-Bessel function, as in \cite[9.235.2]{GR}.  Let $\alpha$ be the character of $\R^\times$ given by $x \mapsto |x|$ and $\sgn$ be the character of of $\R^\times$ given by $x \mapsto x /|x|$. 

 \begin{mytheo}\label{thm:fe}
 Let $\pi = \bigotimes_v \pi_v$ be a standard generic  automorphic representation of  $\GL_2/\Q$ with finite-order central character. 
 \begin{enumerate}
 \item For $( \begin{smallmatrix} y & x \\ & 1 \end{smallmatrix}) \in \GL^+_2(\R) \hookrightarrow \GL_2(\A)$, a newvector $\phi$ for $\pi$  admits a Fourier expansion of the form 
 $$ \phi\left( ( \begin{smallmatrix} y & x \\ & 1 \end{smallmatrix})\right) = c_\phi(y) +  \sum_{n \neq 0} \frac{\rho_\phi(n)}{|n|^{1/2}} W \left(  ny
 \right) e(nx),$$
 where $c_\phi(y)$ is a (possibly vanishing) constant term, the coefficients $\rho_\phi(n) = \rho_\phi(1) \lambda_\pi(|n|)$, the function $W$ is a minimal non-negative weight vector in the Kirillov model  $\cW = \cW(\pi_\infty, \psi_\infty)$ with  $\|W\|_{L^2(\cW)}^2=1$, and 
 \begin{equation}
 \label{eq:fcheckeevals} 
 \| \phi\|_{\it can}^2 = 2 \xi(2)|\rho_\phi(1)|^2 \mathscr{L}_\pi^*(1).
 \end{equation} 
 \item A Whittaker function $W$ satisfying the hypotheses of the previous point can be given explicitly as follows. 
 \begin{enumerate}
\item\label{wt0} 
If $\pi_\infty \simeq \pi(\alpha^{s}\sgn^\epsilon, \alpha^{-s}\sgn^\epsilon)$ with $\epsilon \in \{0,1\}$ 
and 
$ s \in i\R \cup (-1/2,1/2)$, 
then we have
 \begin{equation*}
 W\left( y 
 \right) = 
 (\sgn y)^\epsilon \left( \frac{ \cos \pi s}{\pi} \right)^{1/2} W_{0,s} (4 \pi |y|)  =  (\sgn y)^\epsilon \left( \frac{ \cos \pi s}{\pi} \right)^{1/2} 2 \sqrt{|y|} K_{s}(2 \pi |y|).
 \end{equation*}
 \item\label{wt1ps} 
 If $\pi_\infty \simeq \pi(\alpha^{s}\sgn, \alpha^{-s})$ with $s=it  \in i \R$, then we have
 $$ W\left( y 
 \right) = \left( \frac{ \sinh \pi t}{\pi t^{\sgn y}} \right)^{1/2} W_{\frac{1}{2}\sgn y,it} (4 \pi |y|).$$
 \item\label{wt1cs} If $\pi_\infty \simeq \pi(\alpha^{s}\sgn, \alpha^{-s})$ with $s  \in (-1/2,0)\cup (0,1/2)$, then we have
 $$ W\left(y 
 \right) = 
 \left(\frac{\cos \frac{\pi s}{2}}{(\frac{s}{2})^{\sgn y -1}}\right)^{1/2} W_{\frac{1}{2}\sgn y, \frac{s}{2}}(4 \pi |y|).$$
 \item\label{ds} If $\pi_\infty \simeq \sigma (\chi_1,\chi_2)$ with $\chi_1\chi_2^{-1} = \alpha^s \sgn^m$, $ m \in \{0,1\}$ and $s-m \in 1+ 2\Z_{\geq 0}$ or $(s,m)=(0,1)$, then writing $k=s+1$ we have 
 $$W\left( y 
 \right) = \Gamma(k)^{-1/2} W_{\frac{k}{2},\frac{k-1}{2}}(4 \pi y) \delta(y>0) = \left(\frac{(4 \pi y)^{k}}{\Gamma(k)}\right)^{1/2} e^{-2 \pi y} \delta(y>0).$$
 \end{enumerate}
There is a unique newvector $\phi$ in $\pi$ with $W$ as above, $\|\phi\|^2_{\it can} = 1$, and $\rho_\phi(1)>0$.
 \end{enumerate}
 \end{mytheo}

Remark 1: The Selberg eigenvalue conjecture predicts that the cases \eqref{wt0} and \eqref{wt1cs} above with $s$ real and non-zero never occur as local components of any automorphic representation, but one cannot at present rule out this possibility.

Remark 2: The explicit choice of $W$ given in part (2) of the above theorem is used later to justify the choice of normalization in Theorem \ref{thm:KuznetsovTraceFormula} (see the remarks following it). In addition, we believe it could be valuable to record Theorem \ref{thm:fe} for the sake of reference.

Remark 3: In addition to $\| \phi \|_{\it can}$, Michel-Venkatesh define an Eisenstein norm $\| \phi \|_{\rm Eis}$ on the space of $\pi$, see \cite[\S 2.2.1]{MichelVenkateshGL2}. Correcting \cite[(2.3)]{MichelVenkateshGL2} as pointed out following \eqref{eq:HLIw}, for a number field $F$ and $\phi \in \pi$, we have $\|\phi\|_{\it can}^2 = 2 \xi^*_F(1) \| \phi\|_{\rm Eis}^2 $ if $\pi$ is Eisenstein and non-singular, and $\|\phi\|_{\it can}^2 = \| \phi \|_{L^2(\mathbb{X}), {\rm push}}^2$ if $\pi$ is cuspidal and $\mathbb{X}$ is given the push-forward measure (cf.\ \cite[Lem.\ 2.2.3]{MichelVenkateshGL2}).

Let $S_*(q,\chi)$ denote either $S_{it_j}(q,\chi)$ or $S_k(q,\chi)$, the vector space of Maass (resp. holomorphic) cusp forms of level $q$, central character $\chi$ and spectral parameter $t_j$ (resp. weight $k$). There is a natural embedding $f \mapsto \phi_f$ of $S_*(q,\chi)$ in the space of automorphic forms. We have in particular for $f \in S_*(q,\chi)$ that  \begin{equation}\label{eq:classaut2}y^{\frac{k}{2}} f(x+iy) = \phi_f\left( ( \begin{smallmatrix} y & x \\ & 1 \end{smallmatrix})\right)\end{equation} and $\langle f,f \rangle= \|\phi_f\|^2_{L^2(\mathbb{X})},$ where the Petersson inner product is defined with respect to probability measure on $\Gamma_0(q)\backslash \mathcal{H}$ (for details, see e.g.\ \cite[Prop.\ 12.5]{KnightlyLi}). We continue to write $S_*(q,\chi)$ for the image of this space under the map $f \mapsto \phi_f$ despite the abuse of notation.

Let $\mathcal{E}_{it}(q,\chi)$ denote the vector space of Eisenstein series of level $q$, central character $\chi$ and spectral parameter $t$, as in \cite[\S 8.1]{YoungEisenstein}. As with cusp forms, the space $\mathcal{E}_{it}(q,\chi)$ embeds in the space of automorphic forms by $E \mapsto \phi_E$. The second author defined a formal inner product $\langle \cdot,\cdot\rangle_{\rm formal}$ on $\mathcal{E}_{it}(q,\chi)$ in loc.\ cit.\ by setting $\frac{1}{4\pi} \langle E_\fa, E_\fb\rangle_{\rm formal} = \delta_{\fa \fb}$ for Eisenstein series attached to singular cusps $\fa, \fb$, and extending linearly.  By chasing definitions in \cite[\S 5]{KnightlyLiKuznetsov} and \cite[(3.3), Lem.\  8.3]{YoungEisenstein}, one finds that $\frac{1}{4\pi} \langle E,E\rangle_{\rm formal}/\nu(q) = \| \phi_E\|^2_{\rm Eis}$, where $\nu(q) = [\SL_2(\Z) : \Gamma_0(q)]$.

If $f$ is an even (resp.~odd) weight $0$ Maass form or Eisenstein series $f$ of spectral parameter $t$, then Theorem \ref{thm:fe}\eqref{wt0} applies to $\phi_f$ with $s=2it$ and $\epsilon=0$ (resp.~$1$). If $f$ is a weight $k$ holomorphic cusp form, then Theorem \ref{thm:fe}\eqref{ds} applies to $\phi_f$.  Theorem \ref{thm:fe}\eqref{wt1ps} and \eqref{wt1cs} pertain to weight 1 Maass forms.

\begin{proof}[Proof sketch]  If $\pi$ is  generic, then $\phi \in \pi$ admits a Whittaker-Fourier expansion \begin{equation*}\phi(g) = \phi_N(g)+ \sum_{\alpha \in \Q^\times} W\left( (\begin{smallmatrix} \alpha & \\ & 1 \end{smallmatrix})g\right),\end{equation*} where $\phi_N(g) = \int_{\A/\Q} \phi(n(x)g)\,dx$ and $W$ is a global Whittaker function.  By expressing $W$ in terms of local Whittaker functions $W_v$ and restricting to $g= ( \begin{smallmatrix} y & x \\ & 1 \end{smallmatrix}) \times 1_{\rm fin}$ we derive the Fourier expansion found in part (1) of Theorem \ref{thm:fe}. From the product over $v<\infty$ of $W_v$ one extracts the Hecke eigenvalue $\lambda_\pi$. Following \cite[\S 2.2.2]{MichelVenkateshGL2}, the canonical norm $\| \phi \|_{ \it can}^2$ is given by a regularized infinite product of local norms on Whittaker models, which leads to the relation \eqref{eq:fcheckeevals}.

If $\pi_\infty$ is a unitary principal series, the formulas for the Whittaker function $W_\infty$  in part (2) of the Theorem can be derived from the explicit isometry between the induced model and the Whittaker model given by \cite[(3.10)]{MichelVenkateshGL2} and the integral formula \cite[3.384.9]{GR} for $W_{\lambda, \mu}(z)$. If $\pi_\infty$ is complementary series or discrete series, then \cite[(3.10)]{MichelVenkateshGL2} still intertwines the induced model and Whittaker model but may no longer be an isometry. In these cases, we may compute $\|W_\infty\|_{L^2(\cW)}^2$ by hand using \cite[7.611.4]{GR} for complementary series and using the definition of $\Gamma(s)$ for discrete series. \end{proof}

\subsection{Twisting}
Let $F$ be a non-archimedean local field, and $\pi$ an irreducible, admissible, generic representation of $\GL_2(F)$ with central character $\omega_\pi$. Writing $c$ for the conductor exponent, we say that $\pi$ is \emph{twist-minimal} if $c (\pi) \leq c(\pi \otimes \chi)$ for all quasi-characters $\chi$ of $F^\times$. 
The following lemma appears in e.g. 
\cite[Lem.\ 1.4]{BLS} or \cite[Lem.\ 2.7]{CorbettSaha}, and relies principally on \cite[Prop.\ 3.4]{TunnellLLC}. 
\begin{mylemma}\label{lem:twistformula}
For all quasi-characters $\chi$ of $F^\times$ we have 
\begin{equation*}
c(\pi \otimes \chi) \leq \max(c(\pi), c(\chi)+c(\omega_\pi\chi)),
\end{equation*}
with equality if $\pi$ is twist-minimal.
\end{mylemma}
 If $\pi$ is a global automorphic representation of $\GL_2$, then we say that $\pi$ is twist-minimal at $p$ if the associated local representation is twist-minimal. We say that $\pi$ is (globally) twist-minimal if it is twist-minimal at all primes dividing its conductor. In that case, we have
\begin{equation}\label{eq:twistformulaglobal}
\cond(\pi \otimes \chi) = [\cond(\pi), \cond(\chi)\cond(\omega_\pi \chi)],
\end{equation}
where $[m,n]$ denotes $\mathrm{lcm}(m,n)$.

 \subsection{Cusps} Our presentation of cusps and scaling matrices in this subsection is inspired by \cite[\S3.4.1]{NPS}. Here we restrict our attention to cusps with respect to Hecke congruence subgroups $\Gamma_0(q)$. For more general co-compact subgroups of $\GL_2^+(\R)$, see loc.~cit. 
 
 The group $\Gamma = \SL_2(\Z)$ acts transitively on $\P^1(\Q)$ by fractional linear transformations. Let $\Gamma_\infty = \{\pm ( \begin{smallmatrix} 1 & n \\ & 1\end{smallmatrix}): n \in \Z\}$ be the stabilizer of $\infty$ in $\Gamma$. Thus we may identify \begin{equation}\label{P1id}\P^1(\Q) \simeq \Gamma / \Gamma_\infty\end{equation} upon picking the base point of $\P^1(\Q)$ to be $\infty$. \begin{mydefi}\label{cuspdef}
 The set of left $\Gamma_0(q)$-orbits $$\mathcal{C}(q) := \Gamma_0(q) \backslash  \Gamma / \Gamma_\infty$$ is called the set of cusps of $\Gamma_0(q)$. A cusp $\fa$ may be identified with a $\Gamma_0(q)$-orbit in $\P^1(\Q)$ via the bijection \eqref{P1id}. The index $w_\fa : = [\Stab_\Gamma(\fa) : \Stab_{\Gamma_0(q)}(\fa)]$ of a cusp $\fa \in \mathcal{C}(q)$ is called the width of $\fa$.  \end{mydefi}
 The notion of width of a cusp in Definition \ref{cuspdef} matches the usual geometric intuition: choosing a fundamental domain $\mathcal{F}_q$ for $\Gamma_0(q) \backslash \cH$ to be a union of translates of the standard fundamental domain $\mathcal{F}$ for $\Gamma \backslash \cH$, the width $w_\fa$ is the number of translates of $\mathcal{F}$ that touch $\fa$ in $\mathcal{F}_q$. Another description of the width $w_\fa$ is that $w_\fa = [\Gamma_\infty: \Gamma_\infty \cap \tau^{-1} \Gamma_0(q) \tau]$, where $\tau \in \Gamma$ is any representative of $\fa$. 
 
 \begin{mydefi}\label{def:scalingmx} If $\tau \in \Gamma$ represents the cusp $\fa \in \mathcal{C}(q)$, then \begin{equation*}\sigma_\fa = \tau \left( \begin{smallmatrix} w_\fa & \\ & 1 \end{smallmatrix}\right)\end{equation*} is called a scaling matrix for $\fa$. \end{mydefi} 
 A scaling matrix for $\fa$ satisfies $\sigma_\fa \infty = \fa$ and $ \sigma_\fa^{-1} \Stab_{\Gamma_0(q)}(\fa) \sigma_\fa = \Gamma_\infty$, but in contrast to the definition given in \cite[(2.15)]{IwaniecClassical} does \emph{not} in general have determinant 1. 
Note also that the Definition \ref{def:scalingmx} of a scaling matrix is more restrictive than the definition in loc.~cit.\textemdash{}Iwaniec's definition would allow us to multiply $\tau$ on the right by any $(\begin{smallmatrix} 1 & x \\ & 1 \end{smallmatrix})$, $x\in \R$. 

For $\fa \in \mathcal{C}(q)$ and $\sigma_\fa$ a scaling matrix, a vector $\phi \in S_{*}(q,\chi)$ or $\mathcal{E}_{it}(q,\chi)$ 
 admits a Fourier expansion at $\fa$ of the shape 
 \begin{equation}
 \label{eq:fecusp} 
 \phi\left( \sigma_\fa( \begin{smallmatrix} y & x \\ & 1 \end{smallmatrix})\right) = (y^{\frac{k}{2}} f)\vert_{\sigma_\fa}(x+iy)=  c_{\phi, \fa}(y) + \sum_{n \neq 0} \frac{ \rho_\fa(n)}{|n|^{1/2}} W \left( ny  
 \right) e(nx),
 \end{equation}
 where $W \in \cW$ is given by the table in Theorem \ref{thm:fe}(2)and $c_{\phi, \fa}(y)$ is a possibly-zero constant term. In particular, a classical cusp form $f$ or non-holomorphic Eisenstein series admits a Fourier expansion of the form \eqref{eq:fecusp}. Sometimes we write $\rho_{\phi, \fa}(n)$ or $\rho_{f,\fa}(n)$ for $\rho_\fa(n)$ if we want to emphasize the dependence of $\rho_\fa$ on $\phi$ or $f$.
  \begin{mydefi}  
  The coefficients $\rho_\fa(n)$ appearing \eqref{eq:fecusp} are called the Fourier coefficients of $f$ at the cusp $\fa$ and depend on the choice of scaling matrix $\sigma_\fa$.
 \end{mydefi}
The Fourier coefficients $\rho_\fa(n)$ are given in terms of the local Whittaker models at primes dividing $nq$ (see e.g. \cite[\S3.4.2]{NPS}). The Fourier coefficients $\rho_\fa(n)$ may also be more explicitly computed in terms of the Hecke eigenvalues $\lambda_\pi(n)$ and other invariants of $\pi$ using the Jacquet-Langlands local functional equations at primes dividing $q$. 
\begin{myexam} Consider the cusps $\infty$ and $0$, and choose $\sigma_\infty = (\begin{smallmatrix} 1 & \\ & 1 \end{smallmatrix})$ and $\sigma_0 = (\begin{smallmatrix} & -1 \\ q & \end{smallmatrix})$. Then for a newvector $\phi$ of conductor $q$ we have \begin{equation}\label{rho0infty}\rho_0(n) = \epsilon(\pi)_{\rm fin} \overline{\rho_\infty(n)},\end{equation} where $\epsilon(\pi)_{\rm fin}$ is the finite root number of the representation $\pi$. It satisfies $|\epsilon(\pi)_{\rm fin}|=1$. The relation \eqref{rho0infty} also follows quickly from  \cite[Thm.\ 3(iii)]{AtkinLehner}, \cite[p.\ 296]{LiNewFormsandFE}.
\end{myexam} Warning: the coefficients $\rho_\fa(n)$ are in general not multiplicative, nor do they even satisfy the weaker condition $\rho_\fa(nm)\rho_\fa(1)= \rho_\fa(n)\rho_\fa(m)$ for pairs of coprime integers $m,n$.  
 
 \subsection{Kloosterman sums at arbitrary cusps}\label{subsec:kloostermansumsatcusps}
 \begin{mydefi}
 \label{def:allowedmoduli}
Let $\fa, \fb \in \mathcal{C}(q)$ and $\sigma_\fa, \sigma_\fb$ be scaling matrices. The set \begin{equation*} 
\mathcal{C}_{\fa \fb} = \{c > 0 : ( \begin{smallmatrix} * & * \\ c & * \end{smallmatrix}) \in  \sigma_\mathfrak{a}^{-1} \Gamma_0(q) \sigma_\mathfrak{b}\}
\end{equation*}
is called the set of \emph{allowed moduli}.
\end{mydefi}
Our change in definition of scaling matrices compared to \cite{IwaniecClassical} also causes an alteration of the definition of the set of allowed moduli, as well as the Kloosterman sum discussed below.  
As a consequence, the new definition has the
advantage that $\mathcal{C}_{\fa \fb} \subseteq \mathbb{N}$ for any cusps $\fa,\fb$. To help the reader translate between Definition \ref{def:allowedmoduli} and \cite{IwaniecClassical}, temporarily define $\mathcal{C}_{\fa \fb}^{\mz}$ to be defined as above, and let $\mathcal{C}_{\fa \fb}^{\text{Iw}}$ be as in \cite{IwaniecClassical}.  Then $\mathcal{C}_{\fa \fb}^{\mz} = (w_{\fb} w_{\fa}^{-1})^{1/2}\mathcal{C}_{\fa \fb}^{\text{Iw}} $.
\begin{myexam}
Take $\fa=\infty$, $\fb = 0$, $\sigma_\infty = (\begin{smallmatrix} 1 & \\ &1 \end{smallmatrix}),$ and $\sigma_0 =  (\begin{smallmatrix}  & -1 \\ q & \end{smallmatrix})$. Then $$\mathcal{C}_{\infty0} = \{ c q : (c,q) = 1, \thinspace c \geq 1 \}.$$
\end{myexam}
Let $\chi$ be an even Dirichlet character modulo $q$.  For $\fa \in \mathcal{C}(q)$ and $\sigma_\fa$ a scaling matrix, let $u_\fa \in \Gamma_0(q)$ be such that $\sigma_\fa^{-1} u_\fa \sigma_\fa = (\begin{smallmatrix} 1 & 1 \\ & 1 \end{smallmatrix})$. \begin{mydefi} If a Dirichlet character $\chi$ modulo $q$ satisfies $\chi(u_\fa)=1$, then we say that $\fa$ is \emph{singular} for $\chi$.\end{mydefi}
 
\begin{mydefi}
If $\fa, \fb$ are singular cusps for $\chi$, then the sum 
\begin{equation}\label{eq:kloostermanDefinition}
		S_{\mathfrak{a} \mathfrak{b}}  (m,n;c;\chi) =
		\sum_{\gamma = (\begin{smallmatrix} a & b\\ c &d \end{smallmatrix}) \in \Gamma_\infty \backslash \sigma_\mathfrak{a}^{-1} \Gamma_0(q) \sigma_\mathfrak{b} / \Gamma_\infty }
		\overline{\chi}(\sigma_\mathfrak{a} \gamma \sigma_\mathfrak{b}^{-1}) e\left(\frac{am + dn }{c} \right)
            \end{equation}     
            is called the Kloosterman sum attached to the cusps $\fa,\fb$. 
\end{mydefi}
If $|c | \not \in \mathcal{C}_{\fa\fb}$, then the sum appearing in \eqref{eq:kloostermanDefinition} is empty, hence $S_{\mathfrak{a} \mathfrak{b}}  (m,n;c;\chi)= 0$. 
Temporarily denote by $S_{\fa \fb}^\Z$ the Kloosterman sum appearing in \eqref{eq:kloostermanDefinition} and by $S_{\fa \fb}^{\text{Iw}}$ the sum appearing in e.g.~\cite[(3.13)]{IwaniecClassical}.  
If $\chi$ is an even Dirichlet character modulo $q$, then $S_{\fa \fb}^\Z$ and $S_{\fa \fb}^{\text{Iw}}$ are related by \begin{equation} 
\label{eq:Sabcomparison}
S_{\fa \fb}^\Z (m,n, c;\chi) = 
S_{\fa \fb}^{\text{Iw}}(m,n,  
\sqrt{w_{\fa} w_{\fb}^{-1}}
c
; \chi).
\end{equation}

\begin{myexam}
Take $\fa=\infty$, $\fb = 0$, $\sigma_\infty = (\begin{smallmatrix} 1 & \\ &1 \end{smallmatrix}),$ and $\sigma_0 =  (\begin{smallmatrix}  & -1 \\ q & \end{smallmatrix})$. Then (see \cite[(2.20)]{KiralYoungKlEis}) \begin{equation} \label{eq:zeroInfinityKloostermanSum}
            S_{\infty0}  (m,n;c q;\chi) = \overline{\chi}(c) S(\overline{q}m,n;c).
        \end{equation} 
\end{myexam}

 \subsection{The Bruggeman-Kuznetsov formula}\label{section:bkformula}
 Let \begin{equation*} 
 V(q) = \Vol \left( \Gamma_0(q) \backslash \cH \right) = \frac{\pi}{3} q \prod_{p \mid q} (1+p^{-1}),
 \end{equation*}
choose $\Phi \in C^\infty_c(\R_{>0})$, 
$\chi$ an even Dirichlet character modulo $q$, singular cusps $\fa,\fb$, and set
\begin{equation}\label{eq:GeneralSumOfKloostermanSum}
		\mathcal{K} = \sum_{c \in \mathcal{C}_{\mathfrak{a} \mathfrak{b} } } S_{\mathfrak{a} \mathfrak{b} }  (m,n;c;\chi) \Phi\Big(\sqrt{\frac{w_\fa}{w_\fb}} c\Big).
	\end{equation}
	
  Define the integral transforms 	\begin{equation}\label{eq:holomorphicIntegralTransformInMellin}
			2 \pi i^{-k} \mathcal{L}^{\rm hol} \Phi (k) 
			=  \frac{1}{2\pi i} \int_{(1)} \frac{2^{s }\Gamma\left(\frac{s + k-1}{2}\right)}{\Gamma\left(\frac{k+1-s}{2}\right)} 
				\widetilde{\Phi}(s + 1) (4\pi \sqrt{|mn|})^{-s} d s,  
		\end{equation}
		and
\begin{equation}
\label{eq:plusminusIntegralTransformInMellin}
 \mathcal{L}^\pm \Phi(t) = \frac{1}{2 \pi i} \int_{(2)} h_{\pm}(s,t) \widetilde{\Phi}(s +1) (4\pi \sqrt{|mn|})^{-s} d s,
\end{equation}
where
\begin{equation}
\label{eq:hplusminusDefinition}
h_{\pm}(s,t) = \frac{2^{s }}{2\pi^2} \Gamma(\tfrac s2  + it )\Gamma(\tfrac s2 -it) \begin{cases}
 \cos(\pi s/2), \qquad &\pm = + \\
  \cosh(\pi t) , \qquad &\pm = -,
\end{cases}
\end{equation}	
and $\widetilde{\Phi}(s) = \int_0^{\infty} \Phi(x) x^s \frac{dx}{x}$ denotes the Mellin transform of $\Phi$. 
\begin{mytheo}[Bruggeman-Kuznetsov Formula] \label{thm:KuznetsovTraceFormula}
		  Let $\Phi \in C^\infty_c(\R_{>0})$ and $\mathcal{K}$ be as in \eqref{eq:GeneralSumOfKloostermanSum}.  We have 
		\[
			\mathcal{K} = \mathcal{K}_{\rm Maass} + \mathcal{K}_{\rm Eis} + \mathcal{K}_{\rm hol}, 
		\]
		where 	
	\begin{equation}\label{eq:discreteKuznetsovSpectrum}
			\mathcal{K}_{\rm Maass}  =  \frac{4\pi }{V(q)} \sum_{t_j} \mathcal{L}^{\pm} \Phi(t_j) \sum_{\epsilon =0,1} (\pm 1)^\epsilon \sum_{\substack { f \in \mathcal{B}_{it_j}(q, \chi) \\ \text{ of parity } \epsilon}} \overline{\rho_{f,\mathfrak{a}}(m)} \rho_{f,\mathfrak{b}}(n), 
\end{equation}
\begin{equation}\label{eq:continuousKuznetsovSpectrum}
			\mathcal{K}_{\rm Eis} = \frac{4 \pi}{V(q)}
			\int_{-\infty}^\infty \mathcal{L}^{\pm}\Phi(t)  \sum_{\epsilon =0,1} (\pm 1)^\epsilon  \sum_{\substack{ E \in \mathcal{B}_{it, {\rm Eis}}(q,\chi) \\ \text{ of parity } \epsilon}}  \overline{\rho_{E,\mathfrak{a} }(m)}
			\rho_{E,\mathfrak{b} } (n )
			d t,
\end{equation}
where one takes  $+$ in $\pm$ (resp.\ $-$) if $mn>0$ (resp.\ $mn<0$), 
and
\begin{equation}
		\label{eq:holomorphicKuznetsovSpectrum}
			\mathcal{K}_{\rm hol} =  \frac{4 \pi}{V(q)} \sum_{k >0, \text{ even}} \mathcal{L}^{\rm hol}\Phi(k) 
			 \sum_{f \in \mathcal{B}_k(q,\chi) } \overline{\rho_{f, \mathfrak{a} }(m)}  \rho_{f, \mathfrak{b}}(n)
\end{equation}
		if  $mn>0$, and $ \mathcal{K}_{\rm hol}= 0$ if  $mn<0$.
	\end{mytheo}
	Above, $\mathcal{B}_*(q,\chi)$ denotes any orthonormal basis of $S_*(q,\chi)$ with respect to the probability measure on $\Gamma_0(q) \backslash \cH$, and $ \mathcal{B}_{it, {\rm Eis}}(q,\chi)$ denotes an orthonormal basis of $\mathcal{E}_{it}(q,\chi)$  with respect to the formal inner product divided by $V(q)$ \cite[\S2.2]{PetrowYoung}. 

Remark: The above formula is taken from \cite[Thm.\ 3.5]{KiralYoung5th}  and \cite[(10.2)]{YoungEisenstein} (which contains a typo: the factor of $4 \pi$ on the right hand side should be deleted), but has been normalized differently in two ways. First, we have defined the Fourier coefficients $\rho_\fa(n)$ using the canonical normalization of archimedean Whittaker models chosen in Theorem \ref{thm:fe}. One of the consequences of this choice is that $\rho_\fa(n)=\rho_\fa(-n)$ by definition (the factor $(\sgn n)^\epsilon$ is naturally part of the archimedean Whittaker function). Explicitly, we have $\nu_\fa(n) = 2 (\sgn n)^\epsilon \rho_\fa(n)$, where $\nu_\fa(n)$ is defined by \cite[(8.5)]{IwaniecSpectralBook}. Secondly, we have chosen probability measure on $\Gamma_0(q)\backslash \cH$ to define inner products, whereas most authors choose the push-forward measure from $\cH$. These two choices result in the appearance of the factor of $\frac{V(q)}{4\pi}$, which is natural, it being the leading constant in Weyl's law for $\Gamma_0(q) \backslash \cH$.

 \subsection{Explicit choice of basis} \label{section:basischoice}
Let $\cH_{it_j}(m,\chi)$ be the (finite) set of cuspidal automorphic representations $\pi$ with conductor $\cond(\pi)=m$, finite order central character $\omega_\pi = \chi$, and $ \pi_\infty \simeq \pi(\alpha^{it_j} \sgn^\epsilon, \alpha^{-it_j} \sgn^\epsilon)$.  One may alternatively (and equivalently) take $\cH_{it_j}(m,\chi)$ as in \cite[\S2.1]{PetrowYoung} to be the set of cuspidal Hecke-Maass newforms of level $m$, spectral parameter $t_j$, and central character $\chi$. 

Similarly, let $\cH_k(m,\chi)$ be the (finite) set of cuspidal automorphic representations $\pi$ with $\cond(\pi)= m$, finite order $\omega_\pi = \chi$ and $ \pi_\infty \simeq \sigma(\chi_1,\chi_2)$ with $\chi_1\chi_2^{-1} = \alpha^s \sgn^m$ for some $s-m \in 1+2 \Z_{\geq 0} $ satisfying $ s+1 = k$. One may also just as well take $\cH_k(m,\chi)$ as in \cite[\S2.1]{PetrowYoung} to be the set of cuspidal holomorphic newforms of level $m$, weight $k$, and central character $\chi$. All statements that follow involving $\cH_{*}(m,\chi)$ will hold equally well with either definition.

Finally, let $\cH_{it,{\rm Eis}}(m,\chi)$ be the (finite) set of pairs $(\mu_1,\mu_2)$ of unitary Hecke characters of $\Q$ such that the global principal series representation $\pi = \pi(\mu_1,\mu_2)$ (see e.g. \cite[\S 3.7]{Bump})  has $\cond(\pi)=m$, $\mu_1\mu_2=\chi$ of finite order, and $ \pi_\infty \simeq \pi(\alpha^{it} \sgn^\epsilon, \alpha^{-it} \sgn^\epsilon)$. One may also take $\cH_{it,{\rm Eis}}(m,\chi)$ to be the set of newform Eisenstein series of level $m$ and character $\chi$ as defined in \cite[\S2.2]{PetrowYoung}. Using the notation of loc.\ cit.\ \S 2.2, the bijection between these two definitions for $\cH_{it,{\rm Eis}}(m,\chi)$ is given by $$(\mu_1,\mu_2) \mapsto E_{\chi_1,\chi_2}(z, 1/2+it),$$ where $\chi_1$ and $\chi_2$ are the primitive Dirichlet characters corresponding to $\chi_1= \mu_1 \vert_{\widehat{\Z}^\times}$ and $\chi_2= \overline{\mu_1} \vert_{\widehat{\Z}^\times}$.  
 Either definition will make sense in what follows. 

If $\chi=1$ is trivial, we may use the shorthand $\cH_*(m) := \cH_*(m,1)$, as well as the shorthand $\cH_* := \bigcup_m \cH_*(m)$, where $* = it_j$, $k$, or $it, {\rm Eis}$.

  For $(\pi, V)$ a cuspidal representation of conductor $m$ and central character $\omega_\pi$, write 
  $$ \pi^{K_0(m \ell)} =  \{ \phi \in V : \pi(g) \phi = \omega_\pi(g)\phi \text{ for all } g \in K_0(m \ell )\}.$$ 
  The set of fixed vectors $\pi^{K_0(m\ell)}$ is also called an oldclass in the classical terminology, i.e.  $$S_*(\ell, f,\chi) = \pi^{K_0(m \ell)}, \quad \text{ via } f \mapsto \phi_f,$$ where $S_*(\ell, f,\chi)$ was the notation used in \cite[(2.5)]{PetrowYoung}.   
  
  As a first step in the construction of an orthonormal basis for $S_{*}(q,\chi)$, observe that forms $\phi \in S_{*}(q,\chi)$ that generate distinct irreducible cuspidal automorphic representations are necessarily orthogonal to each other. Thus, we have the orthogonal direct sum
 \begin{equation}\label{ALdecompautomorphic}
 S_{*}(q,\chi) =  \bigoplus_{m \ell= q} \bigoplus_{\pi \in \cH_*(m,\chi)} \pi^{K_0(m \ell)} .
 \end{equation}
By \eqref{ALdecompautomorphic}, the problem of choosing an orthonormal basis for $S_*(q,\chi)$ reduces to choosing orthonormal bases for the oldclasses $\pi^{K_0(m \ell)}$. 

 Write $\phi_d$ for the function $g \mapsto \phi( (\begin{smallmatrix} d & \\ & 1 \end{smallmatrix}) g )$ with $ (\begin{smallmatrix} d & \\ & 1 \end{smallmatrix}) \in \GL_2(\R) \hookrightarrow \GL_2(\A)$ in the first position. If $\phi$ is a newvector for $\pi$ of conductor $m$, we have by Atkin-Lehner-Li theory that 
 \begin{equation}
 \label{ALLcoords}
 \pi^{K_0(m \ell)} = {\rm span}\{\phi_d : d \mid \ell\}.
 \end{equation}  
 We write an orthonormal basis $\mathcal{B}(\ell,\pi)$ for $\pi^{K_0(m \ell)}$ in the coordinates \eqref{ALLcoords} as 
 \begin{equation*}
 \mathcal{B}(\ell,\pi)  = \{ \phi^{(\delta)} = \sum_{d | \ell } x_\delta(d)\phi_d : \delta \mid \ell\}
 \end{equation*}
  for some choice of coefficients $x_\delta(d)$.
  Thus an orthonormal basis for $S_*(q,\chi)$ is given by \begin{equation*}
  \mathcal{B}_*(q,\chi) = \bigcup_{m \ell= q} \bigcup_{\pi \in \cH_*(m,\chi)} \{ \phi^{(\delta)}  : \phi \text{ newvector for } \pi, \delta \mid \ell\}. 
  \end{equation*}
  Taking the Fourier expansion of the newvector $\phi$ at $\infty$ as in Theorem \ref{thm:fe}, we have that the Fourier coefficients at infinity of the forms $\phi_d$ and $\phi^{(\delta)}$ are related by
  \begin{equation}
\label{eq:nufvertdFormula}
 \rho_{\phi_d}(n) = d^{1/2} \rho_\phi(n/d),
 \qquad \text{and} \qquad
 \rho_{\phi^{(\delta)}}(n) = \sum_{d|\ell} d^{1/2} x_\delta(d) \rho_\phi(n/d),
\end{equation}
  where if $n/d$ is not an integer, we interpret $\rho_\phi(n/d)=0$. Since $\rho_\phi(n)$ are directly related to Hecke eigenvalues via \eqref{eq:fcheckeevals}, we also define
  \begin{equation}\label{lambda(delta)def}
  \lambda_\pi^{(\delta)} (n) = \sum_{d|\ell} d^{1/2} x_\delta(d) \lambda_\pi(n/d),
  \end{equation} where likewise $\lambda_\pi(n/d)=0$ if $n/d$ is not an integer. Note that we have $\rho_{\phi^{(\delta)}}(n) = \rho_\phi(1)  \lambda_\pi^{(\delta)} (n)$. 
  
  We denote by $\epsilon(\pi)_{-1}$ the parity of $\pi$ (in line with Iwaniec's notation $T_{-1}$ for the involution $f(z) \mapsto f(-\overline{z})$ on Hecke-Maass forms). Lastly, we set $\epsilon_\pi^{(\pm)} = (\pm 1)^{\epsilon(\pi)_{-1}}\epsilon(\pi)_{\rm fin}$, where $\epsilon(\pi)_{\rm fin}$ was defined in \eqref{rho0infty}. With these notations we have the following.

 \begin{mytheo}[Explicit Bruggeman-Kuznetsov Formula for cusps $\infty, 0$] \label{thm:KuznetsovTraceFormulainf0}
		  Let $\Phi \in C^\infty_c(\R_{>0})$.  We have 
		\begin{equation}
		\label{eq:Kdefinfinity0}
\mathcal{K} = \sum_{(c,q)=1 } \overline{\chi}(c)S(\overline{q} m,n;c)\Phi(q^{1/2} c) = \mathcal{K}_{\rm Maass} + \mathcal{K}_{\rm Eis} + \mathcal{K}_{\rm hol}, 
		\end{equation}
with notation as follows.  We have
\begin{equation}\label{eq:discreteKuznetsovSpectrumexplicit}
			\mathcal{K}_{\rm Maass} =
\frac{4\pi  }{V(q) } \sum_{t_j}  \mathcal{L}^\pm\Phi(t_j)  \sum_{\ell r = q} 
\sum_{{\pi} \in \mathcal{H}_{it_j}(r, \chi)}
\frac{ \epsilon_\pi^{(\pm)} }{ \mathscr{L}_\pi^*(1)} \sum_{\delta \mid \ell} \overline{\lambda}_{\pi}^{(\delta)}(|m|) \overline{\lambda}_{\pi}^{(\delta)}(|n|),
		\end{equation}
	\begin{equation}
	\label{eq:continuousKuznetsovSpectrumexplicit}
\mathcal{K}_{\rm Eis} 
=  \frac{4\pi  }{V(q) }
\int_{-\infty}^\infty  \mathcal{L}^{\pm}\Phi(t) 
\sum_{\ell r = q} 
\sum_{{\pi} \in \mathcal{H}_{it, {\rm Eis}}(r, \chi)}
\frac{ \epsilon_\pi^{(\pm)}}{2\pi \mathscr{L}_\pi^*(1)}
 \sum_{\delta \mid \ell} \overline{\lambda}^{(\delta)}_{{\pi}} (|m|)
\overline{\lambda}^{(\delta)}_{{\pi}} (|n| )
d t,
	\end{equation}
 where one takes  $+$ in $\pm$ (resp.\ $-$) if $mn>0$ (resp.\ $mn<0$), and
		\begin{equation}
	\label{eq:holomorphicKuznetsovSpectrumexplicitinf0}
			\mathcal{K}_{\rm hol} =  \frac{4\pi  }{V(q) } \sum_{k >0, \text{ even}} 
			\mathcal{L}^{\rm hol}\Phi(k) 
			\sum_{\ell r = q} \sum_{{\pi} \in \mathcal{H}_{k}(r, \chi)} 
			\frac{ \epsilon(\pi)_{\rm fin}}{ \mathscr{L}_\pi^*(1) }
			\sum_{\delta \mid \ell} \overline{\lambda}_{\pi}^{(\delta)}(|m|) \overline{\lambda}_{\pi}^{(\delta)}(|n|)
		\end{equation}
		if  $mn>0$, and $ \mathcal{K}_{\rm hol}= 0$ if  $mn<0$. 
 \end{mytheo}

   There are various choices of basis for $\pi^{K_0(m\ell)}$ in the literature (see e.g. \cite{ILS} \cite{PetrowYoungPeterssonFormula} \cite{BlomerMili}), and it is not clear that there is any canonical choice for general level. Let $\xi_\delta(d)$ be the coefficients defined in \cite[Prop.\ 7.1]{PetrowPetersson}. The choice $x_\delta(d)= \xi_\delta(d)$ defines an orthonormal basis $\{\phi^{(\delta)}: \delta \mid \ell\}$ for $\pi^{K_0(m\ell)}$ (see \cite[Thm.\ 3.2]{SPYbasis} for a nice proof that avoids the Rankin-Selberg method). The coefficients $\xi_\delta(d)$ are given in terms of the divisors of $d$ and $\delta$ and the Hecke eigenvalues of $\pi$. Inspecting the definition of $\xi_\delta(d)$, one deduces the following lemma.
  \begin{mylemma}\label{lem:xigdBound} The coefficients $\xi_\delta(d)$ enjoy the following properties:
\begin{enumerate}
\item The coefficients $\xi_\delta(d)$ are supported on $d|\delta$.
\item The function $\xi_\delta(d)$ is jointly multiplicative in $\delta,d$.
\item We have
$
 \xi_\delta(d) \ll (\delta d)^{\varepsilon}.
$
\end{enumerate}
\end{mylemma}
As a consequence of Lemma \ref{lem:xigdBound}(1), the $\lambda_\pi^{(\delta)} (n)$ associated to $\xi_\delta(d)$ (see \eqref{lambda(delta)def}) is jointly multiplicative in $\delta,n$ since it is the Dirichlet convolution of jointly multiplicative functions.
 The coefficients $\xi_{\delta}(d)$ also give an orthonormal basis in the case of the Eisenstein series (see \cite[\S8]{YoungEisenstein} for details).

\section{Tools from analytic number theory}
\label{section:analyticnumbertheory}
\subsection{Gauss sums}
\label{section:GaussSums}
We will need estimates for Gauss sums of non-primitive Dirichlet characters. 
\begin{mylemma}
\label{lemma:GaussSums}
 Let $\chi$ be a Dirichlet character modulo $q$, induced by the primitive character $\chi'$ modulo $q'$.  For $n \in \mz$, let
 \begin{equation*}
  \tau(\chi,n) = \sum_{x \shortmod{q}} \chi(x) e_q(nx).
 \end{equation*}
Then
\begin{equation}
\label{eq:tauchinExactFormula}
 \tau(\chi,n) = \tau(\chi') \sum_{d|(n, q/q')} d \thinspace \overline{\chi'} \Big(\frac{n}{d}\Big) \chi'\Big(\frac{q}{dq'}\Big) \mu\Big(\frac{q}{dq'}\Big).
\end{equation}
In particular, $\tau(\chi) = \tau(\chi,1) = \mu(q/q') \chi'(q/q') \tau(\chi')$.
Moreover, if $\chi$ is any Dirichlet character modulo $q$, induced by $\chi'$ modulo $q'$ (including the trivial character with $q'=1$), we have
\begin{equation}
\label{eq:tauchinUpperBound}
 |\tau(\chi,n)| \leq (q')^{1/2} \Big(n, \frac{q}{q'}\Big).
\end{equation}
\end{mylemma}
Remark.  \cite[Lem.\ 3.2]{IK} is relevant but has misprints, so we have included a proof.
\begin{proof}

By M\"obius inversion, 
\begin{equation*}
\tau(\chi,n)= \sum_{d|\frac{q}{q'}} \mu\Big(\frac{q}{dq'}\Big) \chi'\Big(\frac{q}{dq'}\Big)
\sum_{y \shortmod{q'd}} \chi'(y) e_{q' d}(ny)
.
\end{equation*}
Changing variables $y \rightarrow y + q'$ shows that the inner sum over $y$ vanishes unless $d|n$, in which case the sum over $y$ is a Gauss sum for $\chi'$ repeated $d$ times.  It is well-known that 
$$\sum_{y \shortmod{q'}} \chi'(y) e_{q'}(my) = \overline{\chi'}(m) \tau(\chi'),$$
valid for all $m \in \mz$.  This gives \eqref{eq:tauchinExactFormula}. 
 Finally, \eqref{eq:tauchinUpperBound} follows easily from \eqref{eq:tauchinExactFormula}.
\end{proof}

\begin{mycoro}\label{cor:GaussSums}
 Suppose $\chi$ is a character of prime power modulus $q=p^{\beta}$, $\beta \geq 1$ and conductor $q'$.  Let $n$ be an integer.  Then $\tau(\chi) \tau(\chi,n) = 0$ except when the following conditions hold:
\begin{enumerate}
 \item If $q'=q$ and $(n,q) = 1$. 
 \item If $q'=1$ and $q=p$.
\end{enumerate}
\end{mycoro}
\begin{proof}
 If $1 < q' < q$, then $\tau(\chi)=0$.  If $q'=q$, then $\tau(\chi,n) = 0$ unless $(n,q) = 1$.  If $q'=1$ then $\tau(\chi) = S(1,0;q)$, which vanishes unless $q=p$.
 \end{proof}

\subsection{Approximate functional equation for a divisor function times a character}
There are various ways to solve a shifted convolution/divisor problem, including the circle method, the delta symbol method, and via inner products with Poincare series.  
Here we prove a generalized form of \cite[Lem.\ 5.4]{Young4th}, which will be convenient for our purposes.
\begin{mylemma}
\label{lemma:divisorfunctionAFE}
Let $\chi$ be a primitive Dirichlet character modulo $q$.  
Let $G(s)$ be an even entire holomorphic function with rapid decay in vertical strips, satisfying $G(0) = 1$ (e.g. $G(s) = \exp(s^2)$).
Then
 \begin{equation}
 \label{eq:divisorfunctionAFE}
 \tau(n) \chi(n) = 
 \frac{2}{\tau(\overline{\chi})} \sum_{c=1}^{\infty} \frac{\chi(c)}{c}
 f\Big(\frac{c}{\sqrt{n}}\Big)
 \sumstar_{r \shortmod{c q}} \overline{\chi}(r) 
 e_{cq}(nr)
,
\end{equation}
where
\begin{equation*}
 f(x) = \frac{1}{2 \pi i} \int_{(1)} x^{-2s} L(1+2s, \chi_{0,q}) \frac{G(s)}{s} ds,
\end{equation*}
and where $\chi_{0,q}$ denotes the trivial character modulo $q$.
\end{mylemma}
Remarks.    
\begin{enumerate}
\item The proof  of Lemma \ref{lemma:divisorfunctionAFE} gives an even more general formula than \eqref{eq:divisorfunctionAFE}. 
\item  It turns out to be highly convenient that $c$ runs over integers coprime to $q$. \item  It is not hard to check that $f(x)$ is smooth for $x > 0$, and satisfies the bound
\begin{equation}
\label{eq:fbound}
x^j f^{(j)}(x) \ll_{j,\varepsilon, A} 
\frac{x^{-\varepsilon} q^{\varepsilon}}{(1+x)^{A}}. 
\end{equation}
\end{enumerate}

\begin{proof}
 For Dirichlet characters $\chi_1, \chi_2$ to moduli $q_1, q_2$ respectively, define
 \begin{equation*}
 \lambda_{\chi_1, \chi_2}(n,s) :=\sum_{ab=n} \chi_1(a) \overline{\chi_2}(b) \Big(\frac{b}{a}\Big)^{s-\frac12},
\end{equation*}
 where the notation matches that in \cite{YoungEisenstein}.  Observe $\chi(n) \tau(n) = \lambda_{\chi, \overline{\chi}}(n, 1/2)$, and
note the functional equation
\begin{equation}
\label{eq:lambdaFE}
 \lambda_{\chi_1, \chi_2}(n,1-s) = \lambda_{\overline{\chi_2}, \overline{\chi_1}}(n, s).
\end{equation}
Now suppose $\chi_1,\chi_2$ are primitive, and observe that 
 \begin{equation*}
  \lambda_{\chi_1, \chi_2}(n,s)
  = \frac{n^{s-\frac12}}{\tau(\chi_2)} \sum_{c=1}^{\infty} \frac{\chi_1(c)}{c^{2s}} \sum_{r \shortmod{c q_2}} \chi_2(r) e_{c q_2}(nr),
 \end{equation*}
by splitting the sum over $r$ into residue classes modulo $q_2$. Next we
  factor out $a= \gcd(c,r)$ and change variables $c \rightarrow ac$ and $r \rightarrow ar$, giving 
\begin{equation}
\label{eq:lambdaAdditiveCharacterFormula}
 \lambda_{\chi_1, \chi_2}(n,s)
 = \frac{n^{s-\frac12}}{\tau(\chi_2)} L(2s, \chi_1 \chi_2) \sum_{c=1}^{\infty} \frac{\chi_1(c)}{c^{2s}} \sumstar_{r \shortmod{c q_2}} \chi_2(r) 
 e_{cq_2}(nr).
\end{equation}

Consider
\begin{equation*}
 \frac{1}{2 \pi i} \int_{(1)} \lambda_{\chi_1, \chi_2}(n, s+\tfrac12) \frac{G(s)}{s} ds.
\end{equation*}
Shifting the contour to $\real(s)=-1$, applying \eqref{eq:lambdaFE}, and changing variables $s \rightarrow -s$  gives
\begin{equation}
\label{eq:lambdachi1chi2formula}
 \lambda_{\chi_1,\chi_2}(n,1/2)  = 
 \frac{1}{2 \pi i} \int_{(1)} \lambda_{\chi_1, \chi_2}(n, s+\tfrac12) \frac{G(s)}{s} ds
+
\frac{1}{2 \pi i} \int_{(1)} \lambda_{\overline{\chi_2}, \overline{\chi_1}}(n, s+\tfrac12) \frac{G(s)}{s} ds.
 \end{equation}
Since both $\chi_1, \chi_2$ are primitive, we may insert \eqref{eq:lambdaAdditiveCharacterFormula} into the two integrals.  The first term in \eqref{eq:lambdachi1chi2formula} then equals
\begin{equation*}
 \frac{1}{\tau(\chi_2)} \sum_{c=1}^{\infty} \frac{\chi_1(c)}{c}
 \sumstar_{r \shortmod{c q_2}} \chi_2(r) e_{cq_2}(nr) 
 f\Big(\frac{c}{\sqrt{n}}\Big).
\end{equation*}
where
\begin{equation*}
 f(x) = \frac{1}{2 \pi i} \int_{(1)} x^{-2s} L(1+2s, \chi_1 \chi_2) \frac{G(s)}{s} ds,
\end{equation*} 
and the second term is similar.
The lemma then follows, taking $\chi_1 = \chi$, $\chi_2 = \overline{\chi}$.  
\end{proof}

\subsection{The large sieve inequality}
Let us denote by \begin{equation}\label{not:intleqT} \int_{* \leq T} \quad \text{ any of }  \quad \sum_{|t_j| \leq T},\, \sum_{k \leq T}, \,\text{ or }  \int_{|t| \leq T}\,dt\end{equation} according to whether $* = it_j,k,$ or $it, {\rm Eis}$.  
\begin{mylemma}[Spectral large sieve]
\label{lemma:largesieve}
For any sequence of complex numbers $a_n$, we have
\begin{equation*}
\int_{* \leq T} \sum_{\pi \in \mathcal{H}_{*}(q)} \Big| \sum_{n \leq N} a_n \lambda_{\pi}(n) \Big|^2\ll_\eps (T^2 q + N)(qTN)^{\varepsilon} \sum_{n \leq N} |a_n|^2.
  \end{equation*}
 \end{mylemma}
\begin{proof}
 The spectral sum on the left hand side is estimated by
 \cite[Thm.\ 2]{DeshouillersIwaniec}, but with weights $\mathscr{L}_\pi^*(1)^{-1}$. These weights may be absorbed into the factor $(qT)^{\varepsilon}$ on the right hand side.  
 \end{proof}

\subsection{Additional spectral bounds} 

\begin{mylemma}
\label{lemma:spectralboundOneFourierCoefficient}
Suppose $(q_1, q_2) = 1$ and $(n,q_1 q_2) = 1$.  Then
\begin{equation}
\label{eq:SpectralBoundOneFourierCoefficient}
\sum_{|t_j| \leq T}
\psum_{\psi \shortmod{q_2}}
\sum_{\pi \in \mathcal{H}_{it_j}(q_1 q_2, \psi)}
|\lambda_{\pi}(n)|^2
\ll (T^2 q_1 q_2^2 + n^{1/2} q_2^{1/2}) (n T q_1 q_2)^{\varepsilon},
\end{equation}
where the $+$ indicates that the sum runs over even Dirichlet characters $\psi$.
\end{mylemma}

Remarks.  The case $q_1 = q_2 = 1$ can be found in \cite[Lem.\ 2.4]{Motohashi}, and the case $q_2 = 1$ is a special case of \cite[Lem.\ 12]{BlomerMili}.
The idea is to use the Bruggeman-Kuznetsov formula together with the Weil bound.  We have not stated the analogous bounds for holomorphic forms or Eisenstein series, since these cases follow immediately from $|\lambda_f(n)| \leq \tau(n)$ which is Deligne's bound in the holomorphic case, and directly established for Eisenstein series.

\begin{proof}
Weighting by $\mathscr{L}^*_\pi(1)$ and extending the newforms to an orthogonal basis of $S_{it_j}(q_1q_2,\psi)$ in an arbitrary way, we have by \eqref{eq:HLIw}, \eqref{eq:fcheckeevals} and positivity that the left hand side of \eqref{eq:SpectralBoundOneFourierCoefficient} is 
\begin{equation}\label{eq:onepf1} \ll (q_1q_2T)^{o(1)} \sum_{|t_j| \leq T}
\psum_{\psi \shortmod{q_2}}
\sum_{\phi \in \mathcal{B}_{it_j}(q_1q_2,\psi)} |\rho_\phi(n)|^2.
\end{equation}
Next we extend the sum over $t_j$ in \eqref{eq:onepf1} to the whole spectrum and insert the following smooth weights. To capture $t_j \ll (T q_1 q_2)^{\varepsilon}$, we attach the weight function $h_V(t) = (t^2 + 1/4) \exp(-t^2/V^2)$, with $V= (T q_1 q_2)^{\varepsilon}$ to the spectrum. For $t_j \gg (T q_1q_2)^\eps$ we attach a sum of weights of the form $h_{U,V}(t) = \sum_{\pm} \exp(-(\pm t-U)^2/V^2)$, with $(T q_1 q_2)^{\varepsilon} \ll U \ll T$ and $V = U^{1-\varepsilon}$.
We then apply the Bruggeman-Kuznetsov formula, showing that \eqref{eq:onepf1} is bounded by a sum of expressions of the form
\begin{equation*}
\psum_{\psi \shortmod{q_2}} (UV q_1 q_2 + q_1 q_2 K_{\psi}) (q_1 q_2 T)^{\varepsilon},
\end{equation*}
where
\begin{equation*}
K_{\psi} = \sum_{c \equiv 0 \shortmod{q_1 q_2}} c^{-1} S_{\psi}(n,n;c) B\Big(\frac{4\pi n}{c}\Big),
\qquad
S_{\psi}(m,n;c) = \sum_{y \shortmod{c}} \overline{\psi}(y) e_c(ym + \overline{y} n),
\end{equation*}
and $B(x)$ is the Bessel transform of either $h_V$ or $h_{U,V}$ that appears in the Bruggeman-Kuznetsov formula (see \cite[(9.10)]{IwaniecSpectralBook}).

For a spectral weight function of the type $h_{U,V}$, Jutila and Motohashi \cite{JutilaMotohashi}  showed the bound $B(x) \ll x^{-1/2} U^2$ for $x \gg U^{2+\varepsilon}$, and that $B(x)$ is very small otherwise.   For the case of $h_V$, one may also easily show two crude bounds as follows.  One simple bound is $B(x) \ll V^4$, using the easy bound $\frac{|J_{2it}(x)|}{\cosh(\pi t)} \ll 1$ and  which follows from the integral representation \cite[8.411.4]{GR}.  Hence $B(x) \ll V^4 \ll V (q_1 q_2 T)^{\varepsilon}$.  We also claim
 $B(x) \ll x V^{C}$ for some fixed $C>0$, which can be derived by shifting contours to the line $\real(2it) = 1$, in the integral representation \cite[(9.10)]{IwaniecSpectralBook}, and bounding the integral trivially (one can find more details in \cite[Pf.\ of Lem.\ 10.2]{PetrowYoung}).  Altogether, we derive the bound
\begin{equation}
\label{eq:Bbound}
 B(x) \ll (q_1 q_2 T)^{\varepsilon} \min\Big(V, \frac{x}{V}\Big),
\end{equation}
valid for both classes of test functions $h_{U,V}$ or $h_{V}$.

It suffices to bound the contribution from $K_{\psi}$.  We have
\begin{equation*}
K:=\sumprime_{\psi \shortmod{q_2}} K_{\psi} = 
\sum_{c \equiv 0 \shortmod{q_1 q_2}} c^{-1} B\Big(\frac{4\pi n}{c}\Big)
 \sum_{\psi \shortmod{q_2}} \tfrac12 (1+ \psi(-1)) 
\sum_{y \shortmod{c}} \overline{\psi}(y) e_c(y n + \overline{y} n).
\end{equation*}
The sum over $\psi$ detects the condition $y \equiv \pm 1 \pmod{q_2}$, giving
\begin{equation*}
K=\tfrac12 \varphi(q_2)
\sum_{\pm}
\sum_{c \equiv 0 \shortmod{q_1 q_2}} c^{-1} B\Big(\frac{4\pi n}{c}\Big)
\sum_{\substack{y \shortmod{c} \\ y \equiv \pm 1 \shortmod{q_2}}}  e_c(y n + \overline{y} n).
\end{equation*}
Write $c=c_1 c_2$ where $c_2 |q_2^{\infty}$ and $(c_1, q_2) = 1$.  
We claim
\begin{equation}
\label{eq:KloostermanSumSubsumWithCongruence}
\sum_{\substack{y \shortmod{c} \\ y \equiv \pm 1 \shortmod{q_2}}}  e_c(y n + \overline{y} n) \ll \tau(c_1) c_1^{1/2} (n, c_1)^{1/2} \frac{c_2}{q_2},
\end{equation}
as we now show.
The sum \eqref{eq:KloostermanSumSubsumWithCongruence} factors as $S_1 S_2$ where
\begin{equation*}
S_1 = \sumstar_{y \shortmod{c_1}} e_{c_1}(yn\overline{c_2} + \overline{y} n \overline{c_2}), \qquad 
S_2 = \sum_{\substack{y \shortmod{c_2} \\ y \equiv \pm 1 \shortmod{q_2}}}  e_{c_2}(y n\overline{c_1} + \overline{y} n \overline{c_1}).
\end{equation*}
By a trivial bound, we have $S_2 \ll \frac{c_2}{q_2}$, while $S_1 = S(n\overline{c_2}, n\overline{c_2};c_1)$ is the usual Kloosterman sum.  The Weil bound completes the proof of the claim.  

Therefore, we have
\begin{equation*}
K
\ll 
\sum_{\substack{c_2 \equiv 0 \shortmod{q_2} \\ c_2|q_2^{\infty}}}
\sum_{\substack{c_1 \equiv 0 \shortmod{q_1} \\ (c_1, q_2) = 1}}
 c_1^{-1/2+\varepsilon} (n, c_1)^{1/2}  \Big|B\Big(\frac{4\pi n}{c_1 c_2}\Big)\Big| 
\ll \frac{n^{1/2+\varepsilon} q_2^{1/2} (q_1 q_2 T)^{\varepsilon}}{q_1 q_2},
\end{equation*}
using \eqref{eq:Bbound},
which completes the proof.
\end{proof}

\begin{mytheo}
\label{thm:fourthmomentSpectralBound}
Suppose $q=q_1 q_2$ with $(q_1, q_2) = 1$ and $T \geq 1$.  Then
\begin{equation}
\label{eq:fourthmomenttwistedLHS}
 \sum_{\eta \shortmod{q_2}} 
 \sum_{|t_j| \leq T} 
 \sum_{\ell m = q} (\ell, q_1) 
 \sum_{\pi \in \mathcal{H}_{it_j}(m, \eta^2)} 
|L(1/2+it, \pi \otimes \overline{\eta} )|^4 \ll_t q_1 q_2^2 T^2 (q_1 q_2 T)^{\varepsilon},
\end{equation}
with polynomial dependence on $t$.  A similar bound holds true for holomorphic forms, as well as the Eisenstein series. 
\end{mytheo}
Remarks.  
If $q_2 = 1$, then this is a ``standard" fourth moment bound for automorphic $L$-functions, which follows from the spectral large sieve inequality (Lemma \ref{lemma:largesieve}).  It is thus the $q_2$-aspect that has novelty.  
Our proof of Theorem \ref{thm:fourthmomentSpectralBound} eventually reduces the problem to the case $q_2 = 1$.

To gauge the content of Theorem \ref{thm:fourthmomentSpectralBound}, it is helpful to discuss two special cases.  First, suppose that $q_1 =1$ and $q_2 = p$, prime.  The main contribution to \eqref{eq:fourthmomenttwistedLHS} comes from $m=p$ and $\eta$ non-trivial, in which case $\pi \otimes \overline{\eta} \in \mathcal{H}_{it_j}(p^2,1)$.   On such forms the map $(\pi,\eta) \rightarrow \pi \otimes \overline{\eta}$ is at most two-to-one (see Lemma \ref{lem:twistclasses} below), the multiplicity arising from a quadratic twist.  Hence Theorem \ref{thm:fourthmomentSpectralBound} follows from the standard fourth moment bound of level $p^2$.

Next consider the case $q_1=1$, $q_2 = p^2$, with $p$ prime.  If $m=p^2$ and $\eta$ is primitive modulo $p^2$, then $\pi \otimes \overline{\eta} \in \mathcal{H}_{it_j}(p^4, 1)$.  Again, the multiplicity of the map $(\pi, \eta) \rightarrow \pi \otimes \overline{\eta}$ is bounded, and the standard level $p^4$ fourth moment bound suffices to estimate the contribution of these forms to the left hand side of \eqref{eq:fourthmomenttwistedLHS}.  
Now consider the contribution from $\eta$ of conductor $p$.  Consider the typical case that $\pi$ is twist-minimal with $m=p^2$.  Then $\pi \otimes \overline{\eta} \in \mathcal{H}_{it_j}(p^2, 1)$, which is of lower-level than the previous case.  On the other hand, the map $(\pi, \eta) \rightarrow \pi \otimes \overline{\eta}$ has multiplicity $\gg p$, seen as follows.  Suppose $\pi \in \mathcal{H}_{it_j}(p^2, \eta^2)$, and suppose $\chi$ has conductor $p$.  Then $\pi_{\chi}:=\pi \otimes \chi \in \mathcal{H}_{it_j}(p^2, (\eta \chi)^2)$ and $\pi \otimes \overline{\eta} = \pi_{\chi} \otimes \overline{\eta \chi}$, so that there are $p-1$ distinct pairs $(\pi_\chi, \chi \eta)$ all having the same twisted form $\pi \otimes \overline{\eta}$. 
  Luckily, the extra multiplicity is compensated by the saving in the number of forms of level $p^2$ compared to those of level $p^4$.

\begin{proof} For simplicity of exposition, we only give a proof in the case that $q_1=1$ and $t=0$. The generalization to Theorem \ref{thm:fourthmomentSpectralBound} consists of only notational difficulties.  

Abusing notation, for the duration of this proof we denote by $\widehat{(\Z/q\Z)^\times}$ the group of finite-order Hecke characters of $\Q$ with conductor dividing $q$. (For intuition, note also that $\widehat{(\Z/q\Z)^\times}$ is naturally isomorphic to the group of Dirichlet characters modulo $q$.)
Define 
\begin{equation}\label{eq:cHtwdef}
\cH^{\rm tw}_{*\leq T}(q) := \bigcup_{\eta \in \widehat{(\Z/q\Z)^\times}} \bigcup_{* \leq T} \bigcup_{m |q } \{(\pi, \eta): \pi \in \cH_*(m, \eta^2)\},
\end{equation}
where $*$ is any of $it_j$, $k$, or $it, {\rm Eis}$ as in Section \ref{section:basischoice}, and $* \leq T$ denotes either $|t_j| \leq T$, $k \leq T$, or $|t|\leq T$ in each of the three cases of $*$, respectively.  

If there exists $\chi \in \widehat{(\Z/q\Z)^\times}$ such that $\eta_1\chi =  \eta_2$ and $\pi_1 \otimes \chi \simeq \pi_2$, then we say that $(\pi_1, \eta_1), (\pi_2, \eta_2) \in \cH^{\rm tw}_{*\leq T}(q)$ are \emph{twist-equivalent} and write $(\pi_1, \eta_1) \sim (\pi_2, \eta_2)$. The relation $\sim$ is an equivalence relation, and thus we may partition $\cH^{\rm tw}_{*\leq T}(q)$ into twist classes $\TC \in  \cH^{\rm tw}_{*\leq T}(q) / \sim$. The twist classes also arise naturally as the fibers of the map 
\begin{equation*}
\Phi: \cH^{\rm tw}_{*\leq T}(q) \to \bigcup_{* \leq T} \cH_*, \quad (\pi, \eta) \mapsto \pi \otimes \overline{\eta}.
\end{equation*}
With the notation defined in \eqref{not:intleqT}, we therefore have
\begin{equation}\label{eq:twistclassdec}
\sum_{\eta \in \widehat{(\Z/q\Z)^\times}} \int_{* \leq T}  \sum_{m | q} \sum_{\pi \in \cH_*(m, \eta^2)} | L(1/2, \pi \otimes \overline{\eta})|^4   = \sum_{\TC \in  \cH^{\rm tw}_{*\leq T}(q) / \sim}   |L(1/2,\Phi(\TC))|^4 |\TC|,
\end{equation}
where $L(1/2,\Phi(\TC))= L(1/2, \pi \otimes \overline{\eta})$ only depends on the twist class $\TC$.

Note that the automorphic representation $\Phi(\TC)$ has trivial central character and conductor dividing $q^2$. We now estimate the size of $|\TC|$ as well as its conductor in order to show that whenever the conductor of $\Phi(\TC)$ is large, then $|\TC|$ is small to compensate, and vice-versa. 

First of all, each twist class $\TC$ contains a pair $(\pi, \eta)$ for which $\pi$ is twist-minimal at all primes dividing $q$, and so we choose such a twist-minimal pair in each $\TC$, say $(\pi_\TC, \eta_\TC)$. By Lemma \ref{lem:twistformula} (i.e. \eqref{eq:twistformulaglobal}) we have for any $(\pi, \eta) \in \TC$
\begin{equation}\label{eq:condofcc}
 \cond(\pi \otimes \overline{\eta}) = \cond(\pi_\TC \otimes \overline{\eta}_{\TC}) = [ \cond(\pi_\TC),\cond(\eta_{\TC})^2].
\end{equation}

Secondly, to estimate the sizes of the twist classes $|\TC|$ we have the following estimate.
\begin{mylemma}
\label{lem:twistclasses}
For an integer $n \geq 1$, define $\flrt(n)$ to be the largest integer $d$ so that $d^2|n$. 
For $\TC \in  \cH^{\rm tw}_{*\leq T}(q) / \sim$ we have
 \begin{equation}
 \label{eq:cCbound}
  |\TC| \ll \Big(\flrt(q), \frac{q}{\cond(\eta_{\TC}^2)}\Big) q^{\varepsilon}.
 \end{equation}
\end{mylemma}
\begin{proof}
Since every $(\eta, \pi)$ in $\TC$ is a twist of $(\eta_\TC,\pi_{\TC})$, we have
\begin{equation}
 \label{pftwistclasses}
 |\TC| = \#\{\chi \in \widehat{(\Z/q\Z)^\times} : \cond( \pi_\TC \otimes \chi) \mid q\}.
\end{equation}
By \eqref{eq:twistformulaglobal}, the condition $\cond(\pi_\TC \otimes \chi)|q$ is equivalent to $\cond(\chi) \cond(\eta_\TC^2 \chi) | q$.  Since we now see that $|\TC|$ is multiplicative in $\eta$ and $q$, and since the bound \eqref{eq:cCbound} is also, it suffices to prove the lemma under the assumption that $q$ is a prime power.

We now assume that $q=p^{v}$, and switch to the conductor exponent, $c$. For notational simplicity, replace $\eta_{\TC}$ by $\eta$.  Note $\flrt(p^v) = p^{\lfloor v/2 \rfloor}$.
By the previous discussion, we have $|\TC| = \# \{\chi: c(\chi)+c(\eta^2\chi) \leq v\}.$ There are three classes of $\chi$ to consider:
\begin{enumerate}
\item The case $c(\chi) \leq c(\eta^2)$ and $c(\chi \eta^2) = c( \eta^2),$ the latter condition being automatic if $c(\chi)< c(\eta^2)$.   Under this assumption, the condition $c(\chi)+c(\eta^2\chi) \leq v$ is equivalent to $c(\chi)+c(\eta^2) \leq v$, which in turn implies $c(\chi) \leq \min(c(\eta^2), v -c(\eta^2))$. This last quantity is always $\leq \lfloor v /2\rfloor$, so the number of characters $\chi$ satisfying the above hypotheses is bounded as claimed in the lemma.
\item The case $c(\chi) = c(\eta^2)$ and $c(\chi \eta^2) < c( \eta^2).$ Such $\chi$ are of the form $\chi = \overline{\eta}^2\chi'$ with $c(\chi')<c(\eta^2)$.
 Then $c(\chi)+ c(\eta^2 \chi) = c(\chi')+c(\eta^2)$, and so the number of  $\chi \in \TC$ satisfying the hypotheses of this case is bounded as in the previous case.
\item The case $c(\chi) > c(\eta^2)$. This hypothesis implies that $c(\chi \eta^2) = c(\chi)$, and so $c(\chi)+c(\eta^2\chi) = 2c(\chi)$. Thus any $\chi$ in this case which satisfies $c(\chi)+c(\eta^2\chi) \leq v$ also has $c(\chi) \leq \lfloor v/2 \rfloor$.  On the other hand, since  $c(\chi) > c(\eta^2)$ we also have 
$v \geq c(\chi) + c(\eta^2 \chi) > c(\chi) + c(\eta^2)$, so $c(\chi) < v - c(\eta^2)$,
finishing the proof.  \qedhere
\end{enumerate}
\end{proof}
Lastly, to control the set $ \cH^{\rm tw}_{*\leq T}(q)/ \sim$ we will use that it is in bijection with the image $\Phi( \cH^{\rm tw}_{*\leq T}(q)) $
 of $\Phi$. The conductors of the forms $\sigma \in \Phi( \cH^{\rm tw}_{*\leq T}(q))$ are then given by \eqref{eq:condofcc}. 

The above three facts, along with the spectral large sieve inequality (Lemma \ref{lemma:largesieve}) will suffice to finish the proof of Theorem \ref{thm:fourthmomentSpectralBound}. To implement them, we must parametrize the possible values of $\cond(\eta_{\TC})^2, \cond(\eta_{\TC}^2)$, and $\cond(\pi_\TC)$ that may occur as $\TC$ runs over $\cH^{\rm tw}_{*\leq T}(q) / \sim$. 
We thus write the right hand side of \eqref{eq:twistclassdec} as
\begin{equation}\label{eq:parametrize1}
\sum_{r \mid q} \sumprime_{d \mid r} \sum_{\substack{m \mid q \\ d \mid m}} \sum_{\substack{\TC : \cond(\eta_\TC) = r \\ \cond(\eta_\TC^2) = d \\ \cond(\pi_\TC) = m}}  |L(1/2,\Phi(\TC))|^4|\TC|,
\end{equation}
where the $'$ on the sum over $d$ indicates that there are some extra constraints on the parameters $r$ and $d$, which we now explicate for later use. 

A character $\eta \in \widehat{(\Z/q\Z)^\times}$ of conductor $r$ factors over places $\eta= \prod_p \eta_p$, where each $\eta_p:\Z_p^\times \to \C^\times$ and $\cond(\eta) = \prod_{p\mid r} p^{c_p(\eta_p)}$. For $p$ odd, $c_p(\eta_p) =c_p(\eta_p^2)$ unless $c_p(\eta_p)=1$ and $\eta_p$ is the Legendre symbol. When $p=2$, one may similarly check that if $c_2(\eta_2)=\beta$ with $\beta\geq 4$ then $c_p(\eta_2^2)=\beta-1$. Considering $\beta=2,3$ separately, we can conclude that $c_2(\eta_2) - c_2(\eta_2^2) \in \{1,2,3\}$. Therefore, the $'$ on the sum in \eqref{eq:parametrize1} indicates that the sum runs over those $d \mid r$ such that there exists $k \in \{1,2,3\}$ and an odd square-free integer $r'$ such that $r/d = 2^k r'$ and $(d, r')=1$. 

By positivity, Lemma \ref{lem:twistclasses}, and \eqref{eq:condofcc} we have that the sum in \eqref{eq:twistclassdec} is 
\begin{equation*}
 \ll q^\eps \sum_{r \mid q} \sumprime_{d \mid r} \sum_{\substack{m \mid q \\ d \mid m}} ( \flrt(q), \frac{q}{d})   \int_{* \leq T} \sum_{\sigma \in \cH_*([m,r^2])}  |L(1/2, \sigma)|^4.
\end{equation*}
Now we are in a position to apply the spectral large sieve Lemma \ref{lemma:largesieve}
(more precisely, the special case $q_2 =1$ of Theorem \ref{thm:fourthmomentSpectralBound}), giving that \eqref{eq:twistclassdec} is
\begin{equation}\label{eq:twistthm1}
\ll T^2 (qT)^\eps \sum_{r \mid q} \sumprime_{d \mid r} \sum_{\substack{m \mid q \\ d \mid m}}  ( \flrt(q), \frac{q}{d})  [m,r^2]  \ll qT^2 (qT)^\eps \sum_{r \mid q} \sumprime_{d \mid r} \frac{( \flrt(q) d, q)}{d}  \frac{r^2}{(q,r^2)} .
\end{equation}
Our goal is to show that this is $\ll (qT)^{2+\eps}$, so it suffices to show that the innermost double sum in \eqref{eq:twistthm1} is $\ll q^{1+\eps}$, and since both sides are multiplicative, it suffices to show it when $q$ is a prime power, which we now assume.  If $q$ is odd, then the conditions  indicated by the $'$ imply that either $d=r$, or $r$ is a prime and $d=1$. In the case that $d=r$, observe that $$r \frac{( \flrt(q) r,q)}{(r^2,q)} = [\flrt(q),r],$$ which implies the desired bound. The desired bound $\ll q$ is even easier to check in the case that $r$ is a prime and $d=1$. If $p=2$ then one uses that $\frac{r^2}{d} \mid 8 r$ along with the previous reasoning to obtain the desired bound.
\end{proof}

\section{Harmonic analysis steps}
\label{section:harmonicanalysis}
We now begin the proof of Theorem \ref{thm:ShiftedSumBounds}. 
The sequence of steps used in the proof is motivated in Section \ref{section:sketch}.
Let 
\begin{equation*}
S(\chi,h) =  \sum_n \tau(n+h)\chi(n+h) \overline{\chi}(n) \sum_{n_1n_2=n} w(n_1,n_2,h),
\end{equation*}
with $w(x,y,z)$ as in Section \ref{subsec:shifteddivisorwchar}, which in particular satisfies \eqref{eq:wderivativebounds}, and has support on $r \ll N$.
We also recall that $\chi$ is primitive modulo $q$.

\subsection{Approximate functional equation}
Applying Lemma \ref{lemma:divisorfunctionAFE} to $S(\chi,h)$, we obtain
\begin{equation}
\label{eq:SchihPostAFE}
S(\chi, h) = \frac{2}{\tau(\overline{\chi})} \sum_{c} \frac{\chi(c)}{c}  \sumstar_{r \shortmod{cq}} \overline{\chi}(r)
e_{cq}(hr)
\sum_{n_1, n_2} 
e_{cq}(n_1 n_2 r)
 \overline{\chi}(n_1 n_2) w(n_1, n_2,h) f\Big(\frac{c}{\sqrt{n_1 n_2 +h}} \Big).
\end{equation}

\subsection{Poisson summation}

\begin{mylemma}\label{lem:firstSchihlem}
 We have
\begin{equation}
\label{eq:SchihPostPoisson}
S(\chi, h) =   \sum_{(c,q)=1} 
\frac{2}{c^2 q^2} \sum_{n_1, n_2 \in \mz}  I(c,n_1, n_2,h)
S(h \overline{q}, -n_1 n_2 \overline{q};c) T_{\chi}(h,\overline{c} n_1, \overline{c} n_2),
\end{equation}
where
\begin{equation}
\label{eq:Idef}
 I(y,t_1, t_2,h) = \int_0^{\infty} \int_0^{\infty} w(x_1, x_2, h) 
 f\Big(\frac{y}{\sqrt{x_1 x_2 +h}} \Big)
 e_{yq}(-x_1 t_1 - x_2 t_2) dx_1 dx_2,
\end{equation}
and
\begin{equation}
\label{eq:Tchidef}
 T_{\chi}(h,m,n) = \sumstar_{x,y \shortmod{q}} 
 \chi(x+h) \overline{\chi}(x) e_q(m x \overline{y} + n y).
\end{equation}
\end{mylemma}
Remark. By integration by parts, and using \eqref{eq:fbound} it is easy to see that
for $r \ll N$,
\begin{equation}
\label{eq:Ibound}
 \frac{\partial^{k+j_1+j_2+\ell}}{\partial y^k \partial t_1^{j_1} \partial t_2^{j_2} \partial r^{\ell}} I(y, t_1, t_2,r)  \ll_{k,j_1, j_2, \ell, A_1, A_2, \varepsilon }  \frac{N^{1+\varepsilon} y^{-k} |t_1|^{-j_1} |t_2|^{-j_2} r^{-\ell} y^{-\varepsilon} q^{\varepsilon}}{(1+\frac{y^2}{N})^{A_1} (1+\frac{|t_1| N_1}{yq})^{A_2} (1+\frac{|t_2| N_2}{yq})^{A_2} } ,
\end{equation}
where $A_1, A_2 > 0$ may be taken to be arbitrarily large.
\begin{proof}
Consider a sum of the form
\begin{equation*}
V= \sum_{n_1, n_2} e_{cq}(n_1 n_2 r)  \overline{\chi}(n_1 n_2)  W(n_1, n_2),
\end{equation*}
where $W$ is smooth of compact support, and $(r,cq) = (c,q)= 1$.  We split the sum into arithmetic progressions modulo $cq$ and apply Poisson summation, giving
\begin{equation}
\label{eq:Vformula}
 V =  \frac{1}{c^2 q^2} \sum_{n_1, n_2} 
 A(n_1, n_2, c, q, \chi)
 I(n_1, n_2),
\end{equation}
where
\begin{equation*}
 A(n_1, n_2, c, q, \chi) = \sum_{x_1, x_2 \shortmod{cq}} \overline{\chi}(x_1 x_2) e_{cq}(x_1 x_2 r + x_1 n_1 + x_2 n_2),
\end{equation*}
and
\begin{equation*}
 I(n_1, n_2) = \int_0^{\infty} \int_0^{\infty} W(x_1, x_2) e_{cq}(-x_1 n_1 - x_2 n_2) dx_1 dx_2.
\end{equation*}
By the Chinese remainder theorem, we have
\begin{equation*}
 A(n_1, n_2, c, q, \chi) = A_q(n_1, n_2, \chi) A_c(n_1, n_2),
\end{equation*}
where
\begin{equation*}
 A_q(n_1, n_2, \chi) = \sum_{x_1, x_2 \shortmod{q}} 
\overline{\chi}(x_1 x_2)
e_{q}(x_1 x_2 \overline{c} r + n_1 \overline{c} x_1 + n_2 \overline{c} x_2),
\end{equation*}
and
\begin{equation*}
 A_c(n_1, n_2) = 
\sum_{x_1, x_2 \shortmod{c}} 
e_{c}(x_1 x_2 \overline{q} r + n_1 \overline{q} x_1+ n_2 \overline{q} x_2)
= c e_c(-n_1 n_2 \overline{rq}).
\end{equation*}

Now we insert into \eqref{eq:SchihPostAFE} the formula \eqref{eq:Vformula} and subsequent evaluations, obtaining
\begin{equation*}
S(\chi, h) =  \frac{2}{\tau(\overline{\chi})} \sum_{c} \frac{\chi(c)}{c^2 q^2}  \sumstar_{r \shortmod{cq}} \overline{\chi}(r)
e_{cq}(hr) 
 \sum_{n_1, n_2}  A_q(n_1, n_2, \chi)
 e_c(-n_1 n_2 \overline{rq})
 I(c,n_1, n_2,h) .
\end{equation*}
Next we evaluate the $r$-sum, that is
\begin{equation*}
B(n_1, n_2, c, q, \chi) := \sumstar_{r \shortmod{cq}} \overline{\chi}(r)
e_{cq}(hr) A_q(n_1,n_2, \chi) e_c(-n_1 n_2 \overline{rq}).
\end{equation*}
By the Chinese Remainder Theorem, we have $$B(n_1, n_2, c, q, \chi) = B_c(n_1, n_2) B_q(n_1, n_2, \chi),$$ where
\begin{equation*}
B_c(n_1, n_2) = \sumstar_{r \shortmod{c}} 
e_c(hr \overline{q}-n_1 n_2 \overline{rq})
= S(h \overline{q}, -n_1 n_2 \overline{q};c), 
\end{equation*}
and 
\begin{equation*}
B_q(n_1, n_2, \chi) 
= \sumstar_{r \shortmod{q}} \overline{\chi}(r) e_q(hr \overline{c}) \sum_{x_1, x_2 \shortmod{q}} 
\overline{\chi}(x_1 x_2)
e_{q}(x_1 x_2 \overline{c} r + n_1 \overline{c} x_1 + n_2 \overline{c} x_2).
\end{equation*}
Evaluating the $r$-sum with Lemma \ref{lemma:GaussSums}, we obtain
\begin{equation*}
B_q(n_1, n_2, \chi) = \tau(\overline{\chi}) \overline{\chi}(c)  \sum_{x_1, x_2 \shortmod{q}} 
 \chi(x_1 x_2 + h)
\overline{\chi}(x_1 x_2)
e_{q}(n_1 \overline{c} x_1 + n_2 \overline{c} x_2).
\end{equation*}
By a simple change of variables, we have
$ B_q(n_1, n_2, \chi) = \tau(\overline{\chi}) \overline{\chi}(c)  
 T_{\chi}(h, \overline{c} n_1 , \overline{c} n_2 ), $
which completes the proof.
\end{proof}

\subsection{Estimation of zero frequency terms}
We begin with an elementary bound.
\begin{mylemma}
\label{lemma:sumofRamanujanswithsomeCancellation}
 Suppose that $M \geq 1$ and $w= w_M$ 
 satisfies
 \begin{equation*}
 x^j w^{(j)}(x) \ll_j \Big(1+\frac{x}{M}\Big)^{-100}.
 \end{equation*}
Then for $q \neq 1$ we have 
 \begin{equation*}
  \sum_{m} w(m) S(m,0;q) \ll (Mq)^{\varepsilon} \min(M,q).
 \end{equation*}
\end{mylemma}
\begin{proof}
By integration by parts, the derivative bounds on $w$ imply that for ${\rm Re}(s) = \sigma \in (0,2]$ the Mellin transform  of $w$ satisfies
\begin{equation*}
 \widetilde{w}(s) = \int_0^{\infty} w(x) x^s \frac{dx}{x} \ll_{\sigma} M^{\sigma} |s(s+1) \dots (s+100)|^{-1}.
\end{equation*}
 By Mellin inversion, we have
 \begin{equation*}
  \sum_{m}w(m) S(m,0;q)
  =\frac{1}{2 \pi i} \int_{(2)} \widetilde{w}(s) \Big(\sum_{m=1}^{\infty} \frac{S(m,0;q)}{m^s} \Big) ds.
 \end{equation*}
Note
\begin{equation*}
 \sum_{m=1}^{\infty} \frac{S(m,0;q)}{m^s}
 = \zeta(s) \sum_{a|q} a^{1-s} \mu(q/a).
\end{equation*}
 The crucial feature is that this Dirichlet series does not have a pole at $s=1$, by the M\"obius inversion formula, since $q > 1$.  We may then freely move the contour of integration to either the $\varepsilon$-line or the $(1+\varepsilon)$-line, leading to the claimed bound.
\end{proof}

\begin{mylemma}
\label{lemma:SchihAfterPoisson}
We have
\begin{equation}
\label{eq:SchihAfterPoisson}
 S(\chi,h) =  \sum_{(c,q)=1} 
\frac{2}{c^2 q^2} \sum_{n_1, n_2 \neq 0}  I(c,n_1, n_2,h)
S(h \overline{q}, -n_1 n_2 \overline{q};c) T_{\chi}(h,\overline{c} n_1, \overline{c} n_2)
+ O\Big( N \frac{(h,q)}{q} (qN)^{\varepsilon}\Big). 
\end{equation}
\end{mylemma}
Remark.  This error term is consistent with \eqref{eq:SchiFinalBound}, since
\begin{equation}\label{eq:errortermsum}
 \sum_{\substack{1 \leq h \ll H \\ h \equiv 0 \shortmod{d}}} \frac{(h,q)}{q} \ll \frac{H}{q} (Nq)^{\varepsilon}.
\end{equation}
For simplicity, and since it suffices for our application of Theorem \ref{thm:ShiftedSumBounds},  we have bounded the main term trivially, but it could be extracted in explicit form with more work.

\begin{proof}
 We need to bound the terms in \eqref{eq:SchihPostPoisson} with some $n_i = 0$.  
Recall the bound on $I$ from \eqref{eq:Ibound}, and also observe that
 \begin{equation*}
  T_{\chi}(h, 0,n) = S(n,0;q) \sumstar_{x \shortmod{q}} \chi(x+h) \overline{\chi}(x) = S(n,0;q) S(h,0;q),
 \end{equation*}
where the $x$-sum may be evaluated by changing basis to additive characters.
Therefore the contribution to $S(\chi, h)$ from $n_1=n_2 = 0$ is
\begin{equation*}
  2 \sum_{(c,q)=1} 
\frac{\varphi(q)}{c^2 q^2}   I(c,0, 0, h)
S(h, 0;c) S(h,0;q) 
\ll N(Nq)^{\varepsilon} \frac{(h,q)}{q}.
\end{equation*}
Similarly, the contribution from $n_1=0$ and $n_2 \neq 0$ is
\begin{equation*}
 2 \sum_{(c,q)=1} 
\frac{S(h ,0 ;c) S(h,0;q)}{c^2 q^2} \sum_{n_2 \neq 0}  I(c,0, n_2, h)
 S(n_2,0;q).
\end{equation*}
Using Lemma \ref{lemma:sumofRamanujanswithsomeCancellation} and \eqref{eq:Ibound}, we deduce
\begin{equation*}
 \sum_{n_2 \neq 0}  I(c,0, n_2, h)
 S(n_2,0;q) \ll (Nq)^{\varepsilon} N \min\Big(\frac{cq}{N_2}, q \Big).
\end{equation*}
Therefore the contribution to $S(\chi,h)$ from $n_1=0$ and $n_2 \neq 0$ is
\begin{equation*}
 \ll N (Nq)^{\varepsilon} \frac{(h,q)}{q}
\sum_{c} \frac{(h,c)}{c^2} \Big(1+\frac{c^2}{N}\Big)^{-A} \min\Big(1, \frac{c}{N_2}\Big)
 \ll N_1 \frac{(h,q)}{q} (Nq)^{\varepsilon}.
\end{equation*}
Since $N_1 \ll N$, this is even better than the claimed bound that arose from $n_1 = n_2 = 0$.
By a symmetry argument, a similar bound holds for the terms with $n_2 = 0$ and $n_1 \neq 0$.
\end{proof}

At this point, we pause to record a crude bound for $S(\chi,h)$.
Using the trivial bound $|T_{\chi}(h, \overline{c} n_1, \overline{c} n_2)| \leq q^2$, the Weil bound for Kloosterman sums, and the bound \eqref{eq:Ibound} on $I$, we deduce from Lemma \ref{lemma:SchihAfterPoisson}:
\begin{mycoro}
\label{coro:SchihBoundTchiTriviallyBounded}
 We have
 \begin{equation*}
  S(\chi,h) \ll q^{2} N^{3/4} (qN)^{\varepsilon} +  N \frac{(h,q)}{q} (qN)^{\varepsilon}
 \end{equation*}
\end{mycoro}
Remarks. When $q=1$, this corresponds to the classical $N^{3/4+\varepsilon}$ error term for the smoothed shifted divisor problem, which follows from the Weil bound (e.g. see \cite[Thm.\ 1]{DFIQuadraticDivisor}).   

We emphasize  that the arithmetical conditions on $d,h,q$ appearing in the statement of Theorem \ref{thm:ShiftedSumBounds} have not been used yet
(though the archimedean condition $h \ll N$ was used in \eqref{eq:Ibound})
.

This bound is trivial for $q \gg N^{1/8}$, while the reader may recall that for Theorem \ref{thm:fourthmoment}, we need $N \ll q$.  Naturally, our next aim will be to improve this bound by proving a non-trivial bound on $T_{\chi}$.  An optimistic guess is that $T_{\chi}$ is $O(q^{1+\varepsilon})$ in a suitable average sense.  If true, this would only improve the error term $q^2 N^{3/4}$ to $q N^{3/4}$, which is still trivial for $q \gg N^{1/4}$, so more advanced techniques will be necessary.

\subsection{Properties of $T_{\chi}$}
Recall that $T_{\chi}$ was defined by \eqref{eq:Tchidef}.
First observe the symmetry
\begin{equation}
\label{eq:TchiSymmetry}
 T_{\chi}(h,m,n)  = T_{\chi}(h,n,m).
\end{equation}
Suppose that $\chi = \chi_1 \chi_2$ with $\chi_j$ modulo $q_j$ with $(q_1, q_2) = 1$.  The Chinese remainder theorem gives
\begin{equation}
\label{eq:TchiCRT}
 T_{\chi}(h,m,n) = 
 T_{\chi_1}(h,m \overline{q_2},n \overline{q_2})
 T_{\chi_2}(h,m \overline{q_1},n \overline{q_1}).
\end{equation}

So far the conditions $d|q$, $q^3 | d^2$, and $h \equiv 0 \pmod{d}$ appearing in Theorem \ref{thm:ShiftedSumBounds}
have not been used.  In the next lemma, we give a simplification of $T_{\chi}$ under the assumption $q|d^2$ (whence $q|h^2$), which is weaker than the condition $q^2|d^3$ (whence $q^2 | h^3$) that will be used later.
We write
\begin{equation*}
h_q = (h,q).
\end{equation*}
\begin{mylemma}
\label{lemma:Tchi}
 Suppose $q|h^2$.
If $q|h$ then
 \begin{equation}
 \label{eq:TchiqdivideshCase}
  T_{\chi}(h,m,n) = S(m,0;q) S(n,0;q).
 \end{equation}
If $v_p(h) < v_p(q)$ for each $p|q$, then $T_{\chi}(h,m,n) = 0$ unless $h_q | (m,n)$.  In that case, 
we 
have
\begin{equation}
\label{eq:Tchievaluation}
T_{\chi}( h,  m, n) = h_q^2 {\rm Kl}_3\Big(\ell_{\chi} \frac{h}{h_q}, \frac{m}{h_q}, \frac{n}{n_q};\frac{q}{h_q}\Big), 
\end{equation} 
where ${\rm Kl}_3(a,b,c;q)$ is the hyper-Kloosterman sum defined by
\begin{equation*}
 {\rm Kl}_3(a,b,c;q) = \sum_{\substack{x,y,z \shortmod{q} \\ xyz \equiv 1 \shortmod{q}}} e_q(ax + by + cz),
\end{equation*}
 and for $\ell_\chi$ a certain integer with $(\ell_\chi,q)=1$. 
\end{mylemma}
Remarks.
\begin{itemize}
\item  The two cases considered in Lemma \ref{lemma:Tchi}, along with \eqref{eq:TchiCRT}, are enough to completely evaluate $T_{\chi}$ (see Corollary \ref{coro:TchiEvaluation} below). 
\item The integer $\ell_\chi$  appearing in Lemma \ref{lemma:Tchi} is such that when $q \mid h_q^2$ we have $\chi(1+h_q t) = e_{q/h_q}(\ell_\chi t)$ for all $t \in \mz$. Warning: there is some subtlety involved in passing from the locally defined $\ell_\chi$ appearing in Lemma \ref{postnikov} to the global $\ell_\chi$ defined here.
\end{itemize}
\begin{proof}
If $q|h$, then \eqref{eq:TchiqdivideshCase} is immediate from \eqref{eq:Tchidef}.

 Next suppose that $1 \leq v_p(h) < v_p(q)$ for all $p|q$.   Since 
  $q|h^2$, we have
$\chi(1+hu) \chi(1+hv) = \chi(1+h(u+v))$, which means $\chi(1+hu)$, as a function of $u$, is an additive character modulo $q/h_q$.  
Therefore, there exists an integer $\ell_{\chi}$ so that $\chi(1+hu) = e_{q/h_q}(\ell_{\chi} h' u)$, where $h = h' h_{q}$.
Since $\chi$ is primitive, and $q/h_q$ shares the same set of prime factors as $q$, then $(\ell_{\chi}, q) = 1$.  In particular, 
$\chi(x+h) \overline{\chi}(x) = \chi(1+h \overline{x}) = e_{q/h_q}(\ell_{\chi} h' \overline{x})$, and so 
\begin{equation*}
 T_{\chi}(h_q h', m,n) = \sumstar_{x \shortmod{q}} \thinspace \sumstar_{y \shortmod{q}} e_{q/h_q}(\ell_{\chi} h' \overline{x}) e_q(mx \overline{y} + n y).
\end{equation*}
Changing variables $x \rightarrow x + \frac{q}{h_q}$ shows the sum vanishes unless $h_q | m$, in which case it is the same sum repeated $h_q$ times.  The same argument with $y$ shows it vanishes unless $h_q | n$, and then gives
\eqref{eq:Tchievaluation}.
\end{proof}

\begin{mycoro}
\label{coro:TchiEvaluation}
Let notation and assumptions be as in Lemma \ref{lemma:Tchi}.  
 Write $q = q_1 q_2$ where
\begin{equation}
\label{eq:q1q2def}
 q_1 = \prod_{\substack{p \mid q \\ v_p(q) \leq v_p(h)}} p^{v_p(q)}, \qquad q_2 = \prod_{\substack{ p \mid q \\ v_p(q)> v_p(h)}} p^{v_p(q)},
\end{equation}
so that $(q_1,q_2)=1$.   Write $\chi = \chi_1 \chi_2$ where $\chi_j$ has conductor $q_j$.  Let
\begin{equation*}
 h_{q_2} = (h, q_2), \qquad \text{so} \qquad (h,q) = q_1 h_{q_2}.
\end{equation*}
Then $T_{\chi}(h, \overline{c} n_1, \overline{c} n_2) = 0$ unless $h_{q_2}|(n_1, n_2)$, in which case  
\begin{equation*}
 T_{\chi}(h_{q_2}h', \overline{c} h_{q_2} n_1, \overline{c} h_{q_2} n_2) 
 =  
 h_{q_2}^2
 S(0, n_1;q_1) S(0,n_2;q_1) 
 {\rm Kl}_3\Big(\ell_{\chi_2} h', \overline{c q_1} n_1, \overline{c q_1} n_2; \frac{q_2}{h_{q_2}}\Big).
\end{equation*}
\end{mycoro}
Remark.  The factorization of $q$ as $q_1 q_2$ is very natural,
since the $q_1$-part of $\chi$ is almost irrelevant, since $\chi_1(n+h) \overline{\chi_1}(n) = 1$ for $(n,q_1) = 1$.   

Applying Corollary \ref{coro:TchiEvaluation} to \eqref{eq:SchihAfterPoisson}, with $h = h' h_{q_2}$ and with the replacements $n_i \mapsto h_{q_2} n_i$, we deduce
\begin{multline}
\label{eq:SchihWithKl3}
 S(\chi,h) =  \Big[  \sum_{(c,q)=1} 
\frac{2 h_{q_2}^2}{c^2 q^2} \sum_{n_1, n_2 \neq 0}  I(c,h_{q_2}n_1, h_{q_2}n_2, h' h_{q_2})
S(h_{q_2} h' \overline{q}, -n_1 n_2 h_{q_2}^2 \overline{q};c) 
\\
S(0, n_1;q_1) S(0,n_2;q_1)  
{\rm Kl}_3\Big(\ell_{\chi_2} h', \overline{c q_1} n_1, \overline{c q_1} n_2; \frac{q_2}{h_{q_2}}\Big)
\Big] 
+ O\Big( N \frac{(h,q)}{q} (qN)^{\varepsilon}\Big). 
\end{multline}

Again, we pause our analysis to record a simple bound on $S(\chi,h)$.
\begin{mycoro}
\label{cor:Schihbound2}
 Suppose 
 $d|(h,q)$ with $q|d^2$ and write $q = q_1 q_2$ as in \eqref{eq:q1q2def}.  Then
 \begin{equation*}
  S(\chi, h) \ll N^{3/4} \frac{q_2}{h_{q_2}} (qN)^{\varepsilon} + N \frac{(h,q)}{q} (qN)^{\varepsilon},
 \end{equation*}
and
\begin{equation}
\label{eq:SchiAfterTchiSimplifications}
 S(\chi) \ll (N^{3/4} \frac{Hq}{d^2} + N \frac{H}{q}) (qN)^{\varepsilon}.
\end{equation}
\end{mycoro}
\begin{proof}
Smith \cite[Thm.\ 6]{Smith} showed ${\rm Kl}_3(a,b,c;q) \ll q^{1+\varepsilon}$ assuming $(a,q) = 1$ (of course, the most difficult case where $q$ is prime and $(abc,q) = 1$ is due to Deligne \cite[Sommes trig. \S7]{SGA45}).  Applying this bound to \eqref{eq:SchihWithKl3}, noting that $(\ell_{\chi_2} h', \frac{q_2}{h_{q_2}}) = 1$, easily finishes the proof.
\end{proof}

Remarks.
\begin{itemize}
\item Observe the significant savings of the factor $\frac{q_2}{h_{q_2}}$ 
compared to the optimistic  factor of $q$ in the discussion following
Corollary \ref{coro:SchihBoundTchiTriviallyBounded}.  This arose in large part from the important feature that $h_{q_2}|(n_1, n_2)$ (which essentially cancelled against the $h_{q_2}^2$ factor in \eqref{eq:SchihWithKl3}) while the modulus of $T_{\chi}$ is greatly reduced to $q_2/h_{q_2}$.  We also save a $q_1$-factor from the boundedness (on average) of Ramanujan sums.  
\item The bound \eqref{eq:SchiAfterTchiSimplifications} is consistent with our ultimate goal \eqref{eq:SchiFinalBound} when $N \asymp q$ provided that $d \gg q^{7/8}$, which is still a bit too restrictive, since the condition $q^2|d^3$ means $d$ can be as small as $q^{2/3}$.
\item The bound \eqref{eq:SchiAfterTchiSimplifications} is compatible with the discussion in the sketch following \eqref{eq:SchiBoundIntroSketch}.
\end{itemize}

\subsection{Some properties of ${\rm Kl}_3$, and application to $S(\chi,h)$}
\begin{mylemma}
\label{lemma:Kl3FourierTransform}
 Suppose that $q \geq 1$ and $a,b,c \in \mathbb{Z}$.  Write $a=a_0 a'$, $b=b_0 b'$, $c=c_0 c'$ where $a_0b_0c_0|q^{\infty}$ and $(a'b'c', q) = 1$.  Then 
 \begin{equation*}
  {\rm Kl}_3(a,b,c;q) = \frac{1}{\varphi(q)} \sum_{\eta \shortmod{q}}
 \tau(\overline{\eta}, a_0) \tau(\overline{\eta}, b_0)\tau(\overline{\eta}, c_0) 
 \eta(a'b'c'),
 \end{equation*}
 where 
 for $n \in \mz$  the Gauss sum $\tau(\eta,n)$ was defined in Lemma \ref{lemma:GaussSums}. 
\end{mylemma}
\begin{proof}
The proof is standard, a similar calculation appearing in e.g. \cite[\S2]{Smith}.
\end{proof}

We apply Lemma \ref{lemma:Kl3FourierTransform} to \eqref{eq:SchihWithKl3}. We already saw in the proof of Corollary \ref{cor:Schihbound2} that $(\ell_{\chi_2} h', \frac{q_2}{h_{q_2}}) = 1$. To handle $n_1,n_2$, we write  $n_1 = n_{10} n_1'$ and $n_2= n_{20} n_2'$ where $n_{10} n_{20} | q_2^{\infty}$ and $(n_1' n_2', q_2) = 1$. These two conditions are equivalent to $(n_1'n_2' , q_2/h_{q_2})=1$ and $n_{10}n_{20} \mid (q_2/h_{q_2})^\infty$ since by definition \eqref{eq:q1q2def} of $q_2$, the numbers $q_2$ and $q_2/h_{q_2}$ have the same set of prime factors. Thus
\begin{multline*}
 {\rm Kl}_3\Big(\ell_{\chi_2} h', \overline{c q_1} n_1, \overline{c q_1} n_2; \frac{q_2}{h_{q_2}}\Big) \\
 =
 \frac{1}{\varphi(q_2/h_{q_2})}
 \sum_{\eta \shortmod{q_2/h_{q_2}}} \overline{\eta}^2(q_1) \tau(\overline{\eta})
 \tau(\overline{\eta}, n_{10}) 
\tau(\overline{\eta}, n_{20})  
\overline{\eta}^2(c) \eta(\ell_{\chi_2}h' n_1' n_2').
 \end{multline*}
Therefore, 
\begin{multline}
\label{eq:SchihSumOfKloostermans}
 S(\chi,h) = \Big(
 \frac{2 h_{q_2}^2}{ q^2\varphi(q_2/h_{q_2})} 
 \sum_{\eta \shortmod{q_2/h_{q_2}}} \overline{\eta}^2(q_1)  \eta(\ell_{\chi_2}h') 
 \sum_{n_{10}, n_{20} | q_2^{\infty}}
 \tau(\overline{\eta})
 \tau(\overline{\eta}, n_{10}) 
\tau(\overline{\eta}, n_{20}) \\
 \sum_{\substack{n_1', n_2' \neq 0}}
 S(0, n_1';q_1) S(0,n_2';q_1)  
 \eta( n_1' n_2')
 K
\Big) 
+ O\Big( N \frac{(h,q)}{q} (qN)^{\varepsilon}\Big),
\end{multline}
where $h = h' h_{q_2}$ and $K$ is shorthand for the following sum of Kloosterman sums:
\begin{equation}
\label{eq:Kdef}
 K = \sum_{(c,q)=1} 
\frac{I(c,h_{q_2}n_1, h_{q_2}n_2, h' h_{q_2})}{c^2 }   
S(h_{q_2} h' \overline{q} , -  n_1 n_2 h_{q_2}^2 \overline{q} ;c) 
\overline{\eta}^2(c),
\end{equation}
with $n_i = n_{i0} n_i'.$

\section{Spectral analysis of the shifted divisor sum}
\label{section:spectralanalysisofshiftedsum}
\subsection{Set-up for Bruggeman-Kuznetsov}
\label{section:BKsetup}
Our next major goal is to apply the Bruggeman-Kuznetsov formula to $K$ defined by \eqref{eq:Kdef}.  We begin with some simplifications.  The evaluation of $T_{\chi}$ appearing in Lemma \ref{lemma:Tchi} used that $q|d^2$, which is a weaker condition than what is assumed in Theorem \ref{thm:ShiftedSumBounds}, namely that $q^2|d^3$.  We will use this stronger assumption now.

Recall that we factor $q=q_1q_2$ according to the value of $h$ as in Corollary \ref{coro:TchiEvaluation}, and correspondingly set $h_{q_2}=(h,q_2)$ and $h=h_{q_2}h'$. By the definition of the factorization $q=q_1q_2$, this implies that $q_1 | h'$. 
Recalling that $d|h$, for $p|q_2$ we have $3v_p(h_{q_2}) \geq 2 v_p(q_2)$, and therefore $h_{q_2}^2/q_2$ is an integer. 
Note the hypothesis $(c,q)=1$ from \eqref{eq:Kdef} and observe
\begin{equation*}
 S(h_{q_2} h' \overline{q}, -n_1 n_2h_{q_2}^2 \overline{q};c) = 
 S\Big(\frac{h'}{q_1} , - n_1 n_2\frac{h_{q_2}^2}{q_2}  \overline{q_1}
 \thinspace \overline{(q_2/h_{q_2})};c\Big),
\end{equation*}
using that $S(ax,y;c) = S(x,ay;c)$ for $(a,c) = 1$.

Another small remark is that the condition $(c,q) = 1$ in the definition of $K$ is equivalent to the pair of conditions $(c,q_1) = 1$ and $(c,q_2) = 1$.  The latter of these is enforced by the presence of $\overline{\eta}^2(c)$, since $q_2/h_{q_2}$ has the same prime factors as $q_2$ (by definition of $q_2$). The condition $(c,q_1) = 1$ can be detected by multiplying \eqref{eq:Kdef} by the trivial character modulo $q_1$, which we denote $\chi_{0,q_1}$.

Taking these observations together, we derive that $K$ equals
\begin{equation}\label{eq:endS4}
 \sum_{(c,q) = 1} 
\frac{I(c,h_{q_2}n_1, h_{q_2}n_2, h' h_{q_2})}{c^2 }   
S\Big(\frac{h'}{q_1}, - n_1 n_2\frac{h_{q_2}^2}{q_2}  \overline{q_1} \thinspace \overline{(q_2/h_{q_2})};c\Big)
\overline{\eta}^2(c) \chi_{0,q_1}(c).
\end{equation}

Our next goal is to see that $K$ can be viewed as an instance of $\mathcal{K}$ (as in \eqref{eq:Kdefinfinity0}), with the following choices of parameters.  The level, say  $r$, is given by
\begin{equation*}
 r = q_1 \frac{q_2}{h_{q_2}} = \frac{q}{h_{q_2}} = \frac{q}{(h,q_2)},
\end{equation*}
the central character is $\eta^2 \chi_{0,q_1}$,  the pair of cusps is $\infty,0$ and $m=h'/q_1$, $n=- n_1 n_2 h_{q_2}^2/q_2$.   Define the weight function $\Phi(y) = \Phi(y,\cdot)$ (we temporarily suppress the dependence of $\Phi$ on the other variables) by
\begin{equation}
\label{eq:phidef}
\Phi(y) = \frac{r}{y^2} I( r^{-1/2} y, t_1, t_2, h), \quad  \text{ with } \quad t_i = h_{q_2} n_i.
\end{equation}
Therefore, $K = \mathcal{K}$ as in \eqref{eq:Kdefinfinity0}, and hence by Theorem \ref{thm:KuznetsovTraceFormulainf0}, we have
 $K = K_{\rm Maass} + K_{\rm Eis} + K_{{\rm hol}}$,
where
\begin{equation}
\label{eq:KMaassSpectralFirstFormula}
 K_{\rm Maass} = \sum_{t_j}  \mathcal{L}^\pm\Phi(t_j) 
 \sum_{\ell m = \frac{q}{h_{q_2}}} 
 \sum_{\pi \in \mathcal{H}_{it_j}(m, \eta^2)} \frac{4 \pi \epsilon_\pi^{(\pm)}}{V(q/h_{q_2}) \mathscr{L}_\pi^*(1)} 
 \sum_{\delta|\ell} 
 \overline{\lambda}_{\pi}^{(\delta)}\Big(\frac{h'}{q_1} \Big) \overline{\lambda}_{\pi}^{(\delta)}\Big(n_{10} n_{20} | n_1'  n_2'| \frac{h_{q_2}^2}{q_2}\Big).
\end{equation}
and similar formulas hold for $K_{\rm Eis}$ and $K_{{\rm hol}}$.

Inserting this into \eqref{eq:SchihSumOfKloostermans}, we correspondingly write
\begin{equation}
\label{eq:SchihSpectralDecomposition}
S(\chi,h) = S_{\rm Maass}(\chi,h) + S_{\rm Eis}(\chi, h) + S_{{\rm hol}}(\chi,h) + O( N  \frac{(h,q)}{q} (qN)^{\varepsilon}).
\end{equation}

\subsection{The behavior of the integral transforms}
\label{section:BKintegraltransforms}
Here we study the analytic behavior of the integral transforms $\mathcal{L}^{\pm}\Phi$ that occur in the Bruggeman-Kuznetsov formula.

\begin{mylemma}
\label{lemma:BKweightfunctionProperties}
For each $q \in \N$, let $I= I(y,t_1,t_2,h)$ be a smooth function on $\R_{>0} \times \R_{\neq 0}^2 \times \R_{>0}$ supported on $h \ll H \ll N$ in the last variable, and satisfying the bounds \eqref{eq:Ibound} with $N_1N_2=N$. Let $\Phi$ be defined in terms of $I$ by \eqref{eq:phidef}, which is a function of the variables $y,q,h,n_1,n_2$ via the substitutions $r = \frac{q}{h_{q_2}}$ and $t_i = h_{q_2} n_i$, $i=1,2$. Let us write \begin{equation*}
|mn| := \frac{ h' }{q_1} |n_1 n_2| \frac{h_{q_2}^2}{q_2} = \frac{h}{q} |n_1 n_2| h_{q_2} .
\end{equation*}
Then, there exists a function $H = H_{\pm}(s, t, t_1, t_2, h )$
so that
\begin{equation}
\label{eq:LplusminusMellinWithHplusminus}
  \mathcal{L}^{\pm}\Phi(t) = \int_{(\sigma)} H_{\pm}(s, t, t_1, t_2, h) (mn)^{-s/2} ds, 
\end{equation}
where $H_{\pm}$ is holomorphic in $s$ for $\real(s/2) > \theta,$ with $\theta$ any bound towards the Ramanujan conjecture. 
If $|t| \gg (Nq|n_1 n_2|)^{\varepsilon}$ then 
\begin{equation}
\label{eq:phiplusminusboundFortlarge}
 \mathcal{L}^{\pm}\Phi(t) \ll_{A, \varepsilon} (1+|t|)^{-A} (Nq
 |n_1 n_2|)^{-100}.
\end{equation}
If $|t|\ll (qN|n_1 n_2|)^\eps$, then for $\real(s) > 1$ it satisfies the bound 
\begin{equation}
\label{eq:Hbound}
 H_{\pm}(s, t, t_1, t_2, h)
\ll_{\sigma} (qN)^{\varepsilon}
(Nr)^{\frac12+\frac{\sigma}{2}} 
(1+|s|)^{-A} 
\Big(1+\frac{|t_1| N_1}{\sqrt{N} q}\Big)^{-A} \Big(1+\frac{|t_2|  N_2}{\sqrt{N} q}\Big)^{-A}.
\end{equation}

Similarly, there exists a function $H=H_{{\rm hol}}(s,k,t_1, t_2, h)$ that is holomorphic in $s$ for $\real(s) > 0$ so that
\begin{equation*}
 \mathcal{L}^{{\rm hol}}\Phi(k) = \int_{(\sigma)} H_{{\rm hol}}(s, k, t_1, t_2, h) (mn)^{-s/2} ds,
\end{equation*}
and satisfies the same bound as \eqref{eq:Hbound} for $\real(s) > 1$ when $k \ll (qN|n_1 n_2|)^\eps$. If $k \gg (Nq|n_1 n_2|)^{\varepsilon}$ then
\begin{equation*}
 \mathcal{L}^{{\rm hol}}\Phi(k) \ll_{A,\varepsilon} k^{-A} (qN|n_1 n_2|)^{-100}.
\end{equation*}
\end{mylemma}
Remarks.  Since the claimed bounds are the same for the three choices of $\pm$, $\mathrm{hol}$, we may easily treat these cases in unison.
In addition, the assumption that $t_1, t_2 \neq 0$ is unobjectionable by the reduction in Lemma \ref{lemma:SchihAfterPoisson}.
\begin{proof}
We focus on the case of $\mathcal{L}^{\pm}$ first.
Recall that $\mathcal{L}^{\pm}\Phi(t)$ is given by \eqref{eq:plusminusIntegralTransformInMellin}, where
\begin{equation*}
 \widetilde{\Phi}(s+1)  = r \int_0^{\infty} I(r^{-1/2} y, t_1, t_2, h) y^{s-2} dy = r^{\frac{s+1}{2}} \int_0^{\infty} I( y, t_1, t_2, h) y^{s-1} \frac{dy}{y}.
\end{equation*}
By \eqref{eq:Ibound}, and using that $t_1 t_2 \neq 0$, the function $I$ has rapid decay at $0$ and $\infty$, and so $\widetilde{\Phi}$ is entire. 
One sees that $\mathcal{L}^{\pm}\Phi$ has an integral representation of the form \eqref{eq:LplusminusMellinWithHplusminus}, with
\begin{equation}
\label{eq:HplusminusInTermsofPhitilde}
 H_{\pm}(s, t, t_1, t_2, h) = \frac{(4\pi)^{-s}}{2 \pi i} \widetilde{\Phi}(s+1) h_{\pm}(s,t),
\end{equation}
where recall $h_{\pm}(s,t)$ is defined by \eqref{eq:hplusminusDefinition}.  It is easy to check  that $h_{\pm}(s,t)$ is analytic for $\real(s/2) > \theta$ since $t \in \mr$ or $-\theta \leq it \leq \theta$.

Next we work out bounds on $\widetilde{\Phi}$.
The decay of $I(y)$ from \eqref{eq:Ibound} means in practice that $y \ll \sqrt{N}$ and
$y \gg \max(\frac{|t_1| N_1}{q},\frac{|t_2| N_2}{q})$.
For any $a,b>0$, and $s \neq 1$ with $\mathrm{Re}(s) = \sigma$ we have 
$|\int_a^{b} y^{s-2} dy| \leq |s-1|^{-1}  (b^{\sigma-1} + a^{\sigma-1})$.
Integrating by parts, using the triangle inequality, and then using this bound, 
 we deduce for any $j\geq 0$ that
\begin{equation}
\label{eq:phitildebound}
 \widetilde{\Phi}(s+1) \ll_{j,\sigma, A} \frac{N r^{\frac12+\frac{\sigma}{2}} (Nq)^{\varepsilon} }{|(s-1) \dots (s+j-2)|}  
\frac{N^{\frac{\sigma-1}{2}} + \max\Big(\frac{|t_1| N_1}{q}, \frac{|t_2| N_2}{q}\Big)^{\sigma-1} }{\Big(1+\frac{|t_1| N_1}{\sqrt{N} q}\Big)^{A} \Big(1+\frac{|t_2| N_2}{\sqrt{N} q}\Big)^{A}}
 .
\end{equation}
For future use, we also record derivative bounds with respect to the other variables, which leads to the following minor generalization of \eqref{eq:phitildebound}, valid for $\real(s) = \sigma > 1$:
\begin{multline}
\label{eq:phitildebound2}
|t_1|^{j_1} |t_2|^{j_2} h^{j_3} \frac{\partial^{j_1+j_2+j_3} }{\partial t_1^{j_1} \partial t_2^{j_2} \partial h^{j_3}} 
\widetilde{\Phi}(s+1, t_1, t_2, h) \\
\ll_{j,\sigma, A} \frac{N r^{\frac12+\frac{\sigma}{2}} (Nq)^{\varepsilon} }{(1+|s|)^{A}}  
\frac{N^{\frac{\sigma-1}{2}} + \max\Big(\frac{|t_1| N_1}{q}, \frac{|t_2| N_2}{q}\Big)^{\sigma-1} }{\Big(1+\frac{|t_1| N_1}{\sqrt{N} q}\Big)^{A} \Big(1+\frac{|t_2| N_2}{\sqrt{N} q}\Big)^{A}}.
\end{multline}
In addition, it is supported on $h \ll H$.

We also need a bound on $h_{\pm}(s,t)$. Let $d(\cdot,\cdot)$ be the distance function on $\C$.
Stirling's approximation shows that for $\sigma$ fixed and with $d(\frac{\sigma}{2}, \mathbb{Z}_{\leq 0}) \geq \frac{1}{100}$  we have
\begin{equation}
\label{eq:hpmbound}
 h_{\pm}(\sigma + iv, t) \ll 
 (1 + |t+\tfrac{v}{2}|)^{\frac{\sigma-1}{2}} 
(1 + |t-\tfrac{v}{2}|)^{\frac{\sigma-1}{2}}.
\end{equation}
Now it is easy to derive the bound \eqref{eq:Hbound} by putting together \eqref{eq:HplusminusInTermsofPhitilde}, \eqref{eq:phitildebound}, and \eqref{eq:hpmbound}.

It remains to show \eqref{eq:phiplusminusboundFortlarge}.
To see this, suppose $|t| \gg (Nq|n_1 n_2|)^{\varepsilon}$ and shift the contour of integration in \eqref{eq:plusminusIntegralTransformInMellin} far to the left.  There are poles of $h_{\pm}(s,t)$ at $s/2 \pm it =0,-1,-2,\dots$, which have $|\text{Im}(s)| \asymp |t| \gg (Nq|n_1 n_2|)^{\varepsilon}$. Since $\widetilde{\Phi}(s+1)$ is small at this height, these residues give a contribution to $\mathcal{L}^{\pm}\Phi(t)$ that are consistent with \eqref{eq:phiplusminusboundFortlarge}.  From a trivial bound on the new line $\sigma$ with $d(\frac{\sigma}{2},\mz_{\leq 0}) \geq \frac{1}{100}$, we obtain
\begin{equation*}
\mathcal{L}^{\pm}\Phi(t) \ll \frac{N r}{\sqrt{mn}} \Big(\frac{|t| \sqrt{r} a}{\sqrt{mn}}\Big)^{\sigma-1} + O(t^{-A} (Nq|n_1 n_2|)^{-100}),\end{equation*}  
with
\begin{equation*} \quad a = \max\Big(\frac{|t_1| N_1}{q}, \frac{|t_2| N_2}{q} \Big).
\end{equation*}
 Here $a$ is temporary shorthand notation not used past \eqref{eq:somerandomformula} below.

Note that
\begin{equation}
\label{eq:somerandomformula}
\frac{\sqrt{mn}}{a \sqrt{r}}  = \frac{\sqrt{|n_1 n_2|} \sqrt{h}}{ \max(|n_1|N_1, |n_2|N_2)} 
\ll \frac{\sqrt{n_1 n_2} \sqrt{N_1 N_2}}{ \max(|n_1|N_1, |n_2|N_2)},
\end{equation}
where in the last bound we used $h  \ll N = N_1 N_2$.  
If $|t| \gg (qN |n_1 n_2|)^{\varepsilon}$, we may take
$\sigma$ far to the left to see that $\mathcal{L}^{\pm}\Phi(t)$ is very small, as desired.

Now we quickly treat the holomorphic case.
If $s=\sigma + iv$ with $\sigma$ fixed and $d(\frac{k+\sigma-1}{2}, \mathbb{Z}_{\leq 0}) \geq \frac{1}{100}$, then analogously to \eqref{eq:hpmbound}, we have
 \begin{equation}
 \label{eq:gammaratio}
\Big|2^{s-1} \frac{\Gamma(\tfrac{k}{2} + \tfrac{s-1}{2})}{\Gamma(\tfrac{k}{2} - \tfrac{s-1}{2})} \Big| \ll |k+iv|^{\sigma-1}.
\end{equation}
An essentially identical method now shows that $\mathcal{L}^{{\rm hol}}\Phi(k)$ is very small if $k \gg (qN|n_1 n_2|)^{\varepsilon}$.

\end{proof}

\subsection{The spectral expansion of   $S(\chi, h)$}
We will mainly focus on $S_{\rm Maass}(\chi,h)$. 

We apply Lemma \ref{lemma:BKweightfunctionProperties} to the function $I$ defined in \eqref{eq:Idef}, and insert the result in \eqref{eq:KMaassSpectralFirstFormula} to derive (with 
$t_i = h_{q_2} n_i$, 
$n_i = n_{i0} n_i'$)
\begin{multline}
\label{eq:KMaassSpectralRoughShape}
 K_{\rm Maass} = \sum_{t_j}  \int_{(2+\varepsilon)}   
 H_\pm(s,t_j,t_1, t_2, h_{q_2}h')  
 \Big(\frac{q}{h_{q_2}^2 h' n_{10} n_{20} |n_1' n_2'|}\Big)^{s/2}  ds 
 \\
\times \sum_{\ell m = \frac{q}{h_{q_2}}} 
 \sum_{\pi \in \mathcal{H}_{it_j}(m, \eta^2)}
 \frac{4 \pi \epsilon_\pi^{(\pm)}}{V(q/h_{q_2}) \mathscr{L}_\pi^*(1)}
 \sum_{\delta|\ell} 
 \overline{\lambda}_{\pi}^{(\delta)}\Big(\frac{h'}{q_1} \Big) \overline{\lambda}_{\pi}^{(\delta)}\Big(n_{10} n_{20} | n_1'  n_2'| \frac{h_{q_2}^2}{q_2}\Big) .
\end{multline}

Looking back to \eqref{eq:SchihSumOfKloostermans}, our next goal is to convert the sums over $n_1',n_2'$ to integrals of Dirichlet series, and to do so we must reduce to positive values of $n_i'$. We collect the contributions from the four quadrants of the $n_1',n_2'$ summation by setting for $t_1,t_2>0$ \begin{multline}\label{h1def}
 H_1(s,t,t_1,t_2,h) = H_+(s,t,t_1,t_2,h)  + \eta(-1) H_-(s,t,-t_1,t_2,h)\\ + \eta(-1) H_-(s,t,t_1,-t_2,h)+ H_+(s,t,-t_1,-t_2,h). 
\end{multline}
By Mellin inversion there exists a function $H_2$ such that 
\begin{equation}\label{h1invert}
H_1(s,t,t_1,t_2,h) = \int_{\bf{u}} t_1^{-u_1} t_2^{-u_2} h^{-u_3} H_2(s,t,u_1, u_2, u_3) d \bf{u},
 \end{equation}
where we write $\bf{u} = (u_1, u_2, u_3)$, and which is absolutely convergent for $\real(u_j) > 0$ for all $j$. Using \eqref{eq:phitildebound2}, we have for $\real(s) > 1$,  $\real(u_j)>0$, and $t_j \ll (Nq)^{\varepsilon}$  
\begin{equation}
\label{eq:H2bound}
 H_2(s,t_j,u_1, u_2, u_3)
 \ll (Nr)^{\frac{1+\sigma}{2}} (Nq)^{\varepsilon}
 M_1^{\alpha_1} M_2^{\alpha_2} H^{\alpha_3}
 (1+|s|)^{-A} \prod_{j=1}^{3}(1+|u_j|)^{-A},
\end{equation}
where $\alpha_j = \real(u_j)$, and
\begin{equation}
\label{eq:Mdef}
 M_j = \frac{\sqrt{N} q}{N_j}, \quad  j=1,2.
\end{equation}
Gathering together \eqref{eq:SchihSumOfKloostermans}, \eqref{eq:SchihSpectralDecomposition}, \eqref{eq:KMaassSpectralRoughShape}, \eqref{h1def}, and \eqref{h1invert} we get
\begin{multline}
\label{eq:SMaassFormula}
 S_{\rm Maass}(\chi,h_{q_2}h') = 
 \frac{2 h_{q_2}^2}{ q^2\varphi(q_2/h_{q_2})} 
 \sum_{\eta \shortmod{q_2/h_{q_2}}} \overline{\eta}^2(q_1)  \eta(\ell_{\chi_2}h') 
 \sum_{t_j} 
 \sum_{\ell m = \frac{q}{h_{q_2}}} \sum_{\pi \in \mathcal{H}_{it_j}(m, \eta^2)} \\
  \frac{4 \pi \epsilon_\pi^{(\pm)}}{V(q/h_{q_2}) \mathscr{L}_\pi^*(1)} 
   \int_{\bf{u}} \int_{(2+\varepsilon)}
  \Big(\frac{q}{h_{q_2}^2}\Big)^{s/2} 
  \frac{H_2(s,u_1, u_2, u_3)}{h_{q_2}^{u_1 + u_2 +u_3} ( h')^{\frac{s}{2}+u_3}} 
 Z(s,u_1, u_2 ;h_{q_2} h') 
  ds\,  d \bf{u},
\end{multline} 
where 
\begin{multline}
\label{eq:Zdef}
 Z(s,u_1, u_2;h_{q_2} h')= \sum_{\delta|\ell} 
 \sum_{n_{10} n_{20} | q_2^{\infty}}
 \frac{\tau(\overline{\eta}) \tau(\overline{\eta}, n_{10}) 
\tau(\overline{\eta}, n_{20}) }{n_{10}^{\frac{s}{2}+u_1} n_{20}^{\frac{s}{2}+u_2}}
\\
\sum_{\substack{n_1', n_2' \geq 1}}
\frac{ S(0, n_1';q_1) S(0,n_2';q_1)}
    {(n_1')^{\frac{s}{2}+u_1} (n_2')^{\frac{s}{2}+u_2}} 
 \eta( n_1' n_2')
  \overline{\lambda}_{\pi}^{(\delta)}\Big(\frac{h'}{q_1} \Big)
  \overline{\lambda}_{\pi}^{(\delta)}\Big(n_{10} n_{20} n_1' n_2' \frac{h_{q_2}^2}{q_2}\Big)
  .
\end{multline}

For clarity, we recollect the origin of the relevant variable names.  Firstly, $q$ and $d$ are given integers with $d|q$, $q^2|d^3$.  For Theorem \ref{thm:ShiftedSumBounds}, we want to sum over $h \equiv 0 \pmod{d}$.  In our analysis so far, without summation over $h$, we factored $q$ according to $h$, via $q = q_1 q_2$ where $q_1 | h$ and $1 \leq v_p(h) < v_p(q_2)$ for all $p|q_2$. 

 Our next task is to sum \eqref{eq:SMaassFormula} over $h = h_{q_2} h' \equiv 0 \pmod d$ and move the sum over $h'$ to the inside. To implement this swap of the order of summation, we parametrize
over all factorizations $q= q_1 q_2$ with $(q_1, q_2)=1$ and all possible values of $h_{q_2}$ satisfying  $h_{q_2}|q_2$ with $v_p(h_{q_2}) < v_p(q_2)$ for all $p|h_{q_2}$.  Note that these constraints enforce $(h, q_2) = h_{q_2}$, and $h \equiv 0 \pmod{q_1}$.  We may also factor $d=d_1 d_2$ where $d_1 | q_1$ and $d_2|q_2$, and then only sum over $h_{q_2}$ with $d_2|h_{q_2}$.  In this way, we obtain 
\begin{equation}
\label{eq:SMaassSummedOverhUsefulRef}
\sum_{\substack{h \equiv 0 \shortmod{d}}} S_{\rm Maass}(\chi,h)
= \sum_{\substack{q_1 q_2 = q \\ (q_1, q_2) = 1}} 
\sum_{\substack{h_{q_2} |q_2 \\ v_p(h_{q_2}) < v_p(q_2) \\ d_2 | h_{q_2}}} 
\sum_{\substack{h' \equiv 0 \shortmod{q_1} \\ (h', q_2) = 1}}
S_{\rm Maass}(\chi, h_{q_2} h').
\end{equation}
Caution: the status of $h_{q_2}$ has changed. Prior to \eqref{eq:SMaassSummedOverhUsefulRef}, $h_{q_2}$ was a function of $q$ and $h$, whereas now it is only a summation variable constrained by the conditions indicated above.
Thus,
\begin{multline}
\label{eq:SMaassFormula2}
\sum_{\substack{h' \equiv 0 \shortmod{q_1} \\ (h', q_2) = 1}} S_{\rm Maass}(\chi,h_{q_2}h') = 
 \frac{2 h_{q_2}^2}{ q^2\varphi(q_2/h_{q_2})} 
 \sum_{\eta \shortmod{q_2/h_{q_2}}} \overline{\eta}^2(q_1)  \eta(\ell_{\chi_2})
 \sum_{t_j} 
 \sum_{\ell m = \frac{q}{h_{q_2}}} 
 \sum_{\pi \in \mathcal{H}_{it_j}(m, \eta^2)} \\
 \frac{4 \pi \epsilon_\pi^{(\pm)}}{V(q/h_{q_2}) \mathscr{L}_\pi^*(1)}
   \int_{\bf{u}} \int_{(2+\varepsilon)}
  \Big(\frac{q}{h_{q_2}^2}\Big)^{s/2}
   \frac{H_2(s,u_1, u_2, u_3)}{h_{q_2}^{u_1 + u_2 +u_3}}
 \mathcal{Z}(s,u_1, u_2, u_3;h_{q_2}) 
  ds \, d \bf{u},
\end{multline} 
where
\begin{equation}
\label{eq:ZfancyDef}
 \mathcal{Z}(s,u_1, u_2, u_3;h_{q_2}) = \sum_{\substack{h' \equiv 0 \shortmod{q_1} \\ (h', q_2)=1}} \frac{\eta(h') Z(s, u_1, u_2, u_3;h' h_{q_2})}{(h')^{s/2+u_3}}.
\end{equation}

\subsection{Properties of $Z$ and $\mathcal{Z}$}
Define 
\begin{equation*}
 \lambda_\pi^*(n) = \sum_{d|n} |\lambda_\pi(d)|.
\end{equation*}

Recalling $\ell m = q_1 \frac{q_2}{h_{q_2}}$, where $(q_1, q_2) = 1$, write
 \begin{equation*}
  \ell = \ell_1 \ell_2, \qquad m = m_1 m_2, 
  \quad
  \text{where}
  \quad
  \ell_1 m_1 = q_1,  \qquad \ell_2 m_2 = \frac{q_2}{h_{q_2}},
 \end{equation*}
 and note $(\ell_1 m_1, \ell_2 m_2) = 1$.  Recall also that $(h', q_2) = 1$, 
 $q_1|h'$, 
 that $\eta$ has modulus $q_2/h_{q_2}$, and that $q_2$ shares the same prime factors as $q_2/h_{q_2}$.  In Lemma \ref{lemma:Zproperties} below, we will also make use of the assumptions $q_2^2 \mid d_2^3$ and $d_2 \mid h_{q_2}$. 
 
 Finally, we mention that $\pi$ is an automorphic representation/newform of conductor $m = m_1 m_2$ and central character $\eta^2$.
 
\begin{mylemma}
\label{lemma:Zproperties}
 Let $Z(s,u_1,u_2 ; h_{q_2} h')$ be defined by \eqref{eq:Zdef}, initially with $\real(s/2 +u_i)$ large, $i=1,2$.  Then $Z$ has a factorization
 $Z = Z_{\text{good}} Z_{\text{bad}}$, where 
 \begin{equation*}
 Z_{\text{good}}(s,u_1,u_2 ;h_{q_2}h') 
 =
 L(s/2+u_1, \overline{\pi} \otimes \eta) 
  L(s/2+u_2, \overline{\pi} \otimes \eta), 
\end{equation*}
 and for $\real(s/2+u_i) \geq \sigma/2 >  1/2$, $i=1,2$ the series $Z_{\text{bad}}$ is holomorphic, and we have 
 \begin{equation}
 \label{eq:ZbadIndividualh'}
  Z_{\text{bad}}(s,u_1,u_2;h_{q_2}h')  \ll_{\sigma} (qN)^{\varepsilon}
  (\frac{h'}{q_1}, \ell_1)^{1/2-\theta} q_1 \ell_1 \ell_2^{1/2} 
  \Big(\frac{q_2}{h_{q_2}}\Big)^{3/2} \lambda_\pi^*\Big(\frac{h'}{q_1} \frac{h_{q_2}^3}{q_2^2}\Big).
 \end{equation}
 \end{mylemma}

\begin{proof}
To begin, we recall the definition of $\lambda_\pi^{(\delta)}$ from \eqref{lambda(delta)def} and the estimate of Lemma \ref{lem:xigdBound}(2), which together with Hecke multiplicativity imply
\begin{equation}
\label{eq:lambdafgBound}
 \lambda_{\pi}^{(\delta)}(n) \ll n^{\varepsilon} \sum_{d|(n,\delta)} d^{1/2} |\lambda_\pi(n/d)| 
\leq n^{\varepsilon} \sum_{d|(n,\delta)} d^{1/2} \lambda_\pi^*(n/d).
\end{equation}
The function $e \rightarrow \lambda_{\pi}^*(e)$ is approximately sub-multiplicative in the sense that $\lambda_{\pi}^*(de) \ll_{\varepsilon} (de)^{\varepsilon} \lambda_{\pi}^*(d) \lambda_{\pi}^*(e)$.  Moreover, $\lambda_{\pi}^*(d) \ll d^{\theta + \varepsilon} \ll d^{1/2}$.
These facts imply that the function $d \rightarrow d^{1/2} \lambda_{\pi}^*(n/d)$ is essentially (i.e.\ up to a factor $d^{\varepsilon}$) monotonically increasing in $d$.  For example, we may deduce $\lambda_{\pi}^{(\delta)}(n) \ll n^{\varepsilon} (n,\delta)^{1/2} \lambda_{\pi}^*(\frac{n}{(n,\delta)})$.
This monotonicity property will be used repeatedly in the proof below.

The Dirichlet series $Z = Z(s,u_1,u_2;h_{q_2}h')$ (see \eqref{eq:Zdef}) factors as follows. 
 Write $\delta$ uniquely as $\delta = \delta_1 \delta_2$ with $\delta_1 | \ell_1$ and $\delta_2 |\ell_2$, and let
 \begin{equation*}
h' = q_1 h_1 h'', \quad \text{where} \quad h_1 | q_1^{\infty} 
\quad
\text{and} 
\quad
(h'', q) = 1.
\end{equation*}
 Then
$Z = Z_0  Z_1 Z_2 $ where 
\begin{equation}
\label{eq:Z'def}
 Z_0(s, u_1, u_2; h'') = 
 \sum_{\substack{(n_1 n_2, q) = 1}}
\frac{\eta( n_1 n_2)}
    {n_1^{\frac{s}{2}+u_1} n_2^{\frac{s}{2}+u_2}} 
 \overline{\lambda}_{\pi}(h'' ) \overline{\lambda}_{\pi}(n_1 n_2),
\end{equation}
\begin{equation*}
 Z_1(s, u_1, u_2;h_1)  =  \sum_{\substack{\delta_1 | \ell_1 }} 
 \sum_{\substack{n_1, n_2|q_1^{\infty}}}
\frac{ S(0, n_1;q_1) S(0,n_2;q_1)}
    {n_1^{\frac{s}{2}+u_1} n_2^{\frac{s}{2}+u_2}} 
 \eta( n_1 n_2)
  \overline{\lambda}_{\pi}^{(\delta_1)}(h_1)
  \overline{\lambda}_{\pi}^{(\delta_1)}(n_1 n_2)
  ,
\end{equation*}
and
 \begin{equation}
 \label{eq:Z2def}
Z_2(s, u_1, u_2; h_{q_2})
=
\sum_{\substack{\delta_2 | \ell_2 }} 
\sum_{n_1, n_2 | q_2^{\infty}}
 \frac{\tau(\overline{\eta}) \tau(\overline{\eta}, n_1) 
\tau(\overline{\eta}, n_2) }{n_1^{\frac{s}{2}+u_1} n_2^{\frac{s}{2}+u_2}} 
\overline{\lambda}_{\pi}^{(\delta_2)}\Big(n_1 n_2 \frac{h_{q_2}^2}{q_2} \Big).
\end{equation}

Let us begin with $Z_0$, for which it is not hard to see that
\begin{equation}
\label{eq:Z'formula}
 Z_0 = \overline{\lambda_\pi}(h'') \frac{L^{(q)}(s/2+u_1, \overline{\pi} \otimes \eta)  L^{(q)}(s/2+u_2, \overline{\pi} \otimes \eta)}{\zeta^{(q)}(s+u_1+u_2)},
\end{equation}
where $L^{(q)}(s,\pi)$ denotes the $L$-function $L( s,\pi)$  with the Euler factors at the primes dividing $q$ removed.

For $Z_1$, we claim
\begin{equation}
\label{eq:Z1bound}
 Z_1 \ll (h_1 q)^{\varepsilon} (h_1, \ell_1)^{1/2-\theta} q_1 \ell_1 \lambda_\pi^*(h_1),
\end{equation}
as we now proceed to show.  Let $\sigma$ be such that $\real(s/2+u_i) \geq \sigma/2 > 1/2$.
By the first estimate of \eqref{eq:lambdafgBound}, 
\begin{equation}
\label{eq:Z1BoundPreSimplified}
 Z_1 \ll h_1^\eps  \sum_{\delta_1 | \ell_1} 
 \sum_{\substack{n_1,n_2|q_1^{\infty}}}
\frac{ (n_1, q_1) (n_2, q_1)}
    {(n_1 n_2)^{\sigma/2-\eps}} 
  \sum_{\substack{d|(h_1 , \delta_1) \\ e|(n_1 n_2, \delta_1)}} (de)^{1/2} 
  \Big|\lambda_\pi\Big(\frac{h_1}{d}\Big)\Big|
  \cdot
  \Big|\lambda_\pi\Big(\frac{n_1 n_2}{e}\Big)\Big|.
\end{equation}
Since $Z_1$ as well as its claimed upper bound \eqref{eq:Z1bound} are multiplicative, it suffices to show \eqref{eq:Z1bound} for $q_1$ a prime power.  

In the case that $m_1 \neq 1$ 
we may use $|\lambda_{\pi}(p)| \leq 1$ for $p|m_1$ (see \eqref{eq:RamanujanBoundRamifiedPrimes}) leading quickly to \eqref{eq:Z1bound}. 
In the case $\ell_1=q_1$, we have by positivity
the bound $d^{1/2} |\lambda_{\pi}(h_1/d)| \leq (h_1,\ell_1)^{1/2} \lambda_{\pi}^*(h_1)$, and  the bound  $|\lambda_\pi(n_1 n_2/e)| \ll  (n_1 n_2/e)^{\theta+\varepsilon}$  towards the Ramanujan conjecture. Currently any $\theta > 7/64$ is admissible.
Since $q_1=\ell_1$, observe that the summand in \eqref{eq:Z1BoundPreSimplified} is maximized when $n_1=n_2=e=\delta_1 = \ell_1$. 
Bounding all terms by the $e=\delta_1 = \ell_1$ term, the two preceding bounds show  $Z_1 \ll (h_1 q)^{\varepsilon} (h_1, \ell_1)^{1/2} \ell_1^{3/2+\theta} \lambda_\pi^*(h_1)$; we then derive \eqref{eq:Z1bound} using $(h_1, \ell_1)^{1/2} \ell_1^{3/2+\theta} \leq (h_1, \ell_1)^{1/2-\theta} \ell_1^{2}$, since $\frac32 + 2 \theta \leq 2$.

Next we estimate $Z_2$. We claim
\begin{equation}
\label{eq:Z2bound}
 Z_2 \ll (q h_{q_2})^{\varepsilon} \Big(\frac{q_2}{h_{q_2}} \Big)^{3/2}  \ell_2^{1/2} \lambda_\pi^*(h_{q_2}^3/q_2^2)
 .
\end{equation}
The Dirichlet series $Z_2$ can be factored into prime powers, provided that we correspondingly factor $h$ and $\eta$; the Gauss sums are multiplicative, up to a root of unity.  Since the claimed bound on $Z_2$ is multiplicative, and so is $|Z_2|$, it suffices to check it when $q_2$ is a prime power.

We will use estimates for the Gauss sums from Section \ref{section:GaussSums}. Recall in particular from Corollary \ref{cor:GaussSums} that the product of Gauss sums $\tau(\overline{\eta}) \tau(\overline{\eta}, n_1) \tau(\overline{\eta}, n_2)$ vanishes except in the two cases

\begin{enumerate}
\item
 $\eta$ is primitive modulo $q_2/h_{q_2}$ and $(n_1 n_2,q_2) = 1$ 
 \item $\eta$ is trivial, and $q_2/h_{q_2} = p$, prime.
\end{enumerate}

Case 1) Suppose $\eta$ is primitive modulo $q_2/h_{q_2}$.   Then only the terms $n_1=n_2=1$ contribute to $Z_2$, and we have by the bound \eqref{eq:lambdafgBound} for $\lambda_\pi^{(\delta)}(n)$
\begin{equation}
\label{eq:Z2boundPreSimplified1}
 Z_2 \ll \cond(\eta)^{3/2} q_2^\eps \sum_{\delta_2|\ell_2} \sum_{d|(\delta_2, \frac{h_{q_2}^2}{q_2})} d^{1/2} \lambda_\pi^* \Big(\frac{h_{q_2}^2/q_2}{d}\Big).
\end{equation}
Note that $\ell_2 | \frac{q_2}{h_{q_2}}$ and $\frac{q_2}{h_{q_2}} \mid \frac{h_{q_2}^2}{q_2}$, since we recall $q_2^2 \mid d_2^3$ and $d_2\mid h_{q_2}$. Therefore, the largest value of $d$ appearing in the bound \eqref{eq:Z2boundPreSimplified1} above is $d=\ell_2$.   Since the summand in \eqref{eq:Z2boundPreSimplified1} is essentially monotonic, we have
\begin{equation*}
 Z_2 \ll q_2^{\varepsilon} \cond(\eta)^{3/2} \ell_2^{1/2} \lambda_\pi^*\Big(\frac{h_{q_2}^2/q_2}{\ell_2}\Big).
\end{equation*}
If $m_2 \neq 1$, then $|\lambda_\pi(p)| \leq 1$ at primes dividing $q_2$ (see \eqref{eq:RamanujanBoundRamifiedPrimes})
and we easily obtain \eqref{eq:Z2bound}.  If $m_2 = 1$, then $\ell_2  = q_2/h_{q_2}$, and the bound simplifies as
\begin{equation*}
 Z_2 \ll q_2^{\varepsilon} \cond(\eta)^{3/2} \ell_2^{1/2} \lambda_\pi^*(h_{q_2}^3/q_2^2).
\end{equation*}
Substituting $\cond(\eta) = q_2/h_{q_2}$ gives the desired bound. 

Case 2)  Suppose $\eta$ is the trivial character and $q_2/h_{q_2} = p$ is prime.  We obtain by  \eqref{eq:lambdafgBound}
\begin{equation}
\label{eq:Z2boundPreSimplified2}
 Z_2 \ll q_2^\eps \sum_{\delta_2|\ell_2} 
 \sum_{\substack{n_1,n_2|p^{\infty} }} \frac{(n_1,p)(n_2,p)}{(n_1 n_2)^{\sigma/2-\eps}}
 \sum_{d|(\delta_2, n_1 n_2 \frac{h_{q_2}^2}{q_2})} d^{1/2}\lambda_\pi^*\Big(\frac{n_1 n_2 h_{q_2}^2/q_2}{d}\Big).
\end{equation}
The largest value of $d$ appearing in the above sum is at most $\ell_2$ regardless of $n_1, n_2$, and the summand is monotonic increasing in $d$. Similarly, the summand in \eqref{eq:Z2boundPreSimplified2} is monotonically decreasing as a function of $n_1,n_2$ as soon as $n_1,n_2\geq p$ by current progress towards Ramanujan.  Hence
\begin{equation*}
 Z_2 \ll q_2^{\varepsilon} p \ell_2^{1/2} \lambda_\pi^*(\frac{p^2 h_{q_2}^2/q_2}{\ell_2}) = q_2^{\varepsilon} p \ell_2^{1/2} \lambda_\pi^*\Big(\frac{q_2}{\ell_2}\Big).
\end{equation*} 
We can further simplify this bound as follows. 
If $m_2 \neq 1$, then $|\lambda_{\pi}(p)| \leq 1$ for $p|m_2$.  If $m_2 =1$, then $\ell_2 = q_2/h_{q_2} = p$. In either case, we obtain
\begin{equation*}
 Z_2 \ll 
 q_2^{\varepsilon} p \ell_2^{1/2} \lambda_\pi^*(h_{q_2} ).
\end{equation*}
This bound can in turn be absorbed by  \eqref{eq:Z2bound} using current progress towards Ramanujan, since $h_{q_2} = \frac{h_{q_2}^3}{q_2^2} \frac{q_2^2}{h_{q_2}^2}$, so that
$ \lambda_\pi^*(h_{q_2}) \leq \lambda_\pi^*(h_{q_2}^3/q_2^2) (q_2/h_{q_2})^{2\theta}$, and $1+2\theta \leq 3/2$.

Putting together the previous estimates finishes the proof.
\end{proof}

\begin{mylemma}
\label{lemma:ZpropertiesSummedoverh'}
 Suppose that $q = q_1 q_2$ with $(q_1, q_2) = 1$, and suppose that $h_{q_2}$ is an integer as in \eqref{eq:SMaassSummedOverhUsefulRef}.   
Then
\begin{equation}
\label{eq:Zsummedoverh'Formula}
  \mathcal{Z}(s,u_1, u_2, u_3; h_{q_2}) =
 L(s/2+u_1, \overline{\pi} \otimes \eta) 
  L(s/2+u_2, \overline{\pi} \otimes \eta)
  L(s/2+u_3, \overline{\pi} \otimes \eta) 
  \mathcal{Z}_{\text{bad}},
\end{equation}
where for $\real(s/2+u_i) \geq \sigma/2 > 1/2$, $i=1,2,3$ the series $\mathcal{Z}_{\text{bad}}$ is holomorphic, and we have
\begin{equation}
\label{eq:ZbadSummedOverh'}
\mathcal{Z}_{\text{bad}} \ll_{\sigma} 
(q N)^{\varepsilon} q_1^{-\sigma/2} q_1 \ell_1 \ell_2^{1/2} \Big(\frac{q_2}{h_{q_2}}\Big)^{3/2} 
\lambda_\pi^*\Big(\frac{h_{q_2}^3}{q_2^2}\Big).
\end{equation}
\end{mylemma}
Remark.  Lemma \ref{lemma:ZpropertiesSummedoverh'} implies that $\mathcal{Z}$ has analytic continuation to the region of $\mc^4$ with $\mathrm{Re}(s/2+u_i) > 1/2$, $i=1,2,3$.

\begin{proof}
As in the proof of Lemma \ref{lemma:Zproperties}, we have a factorization of the left hand side of \eqref{eq:Zsummedoverh'Formula} as 
$\mathcal{Z}_0 \mathcal{Z}_1 \mathcal{Z}_2$, analogously to \eqref{eq:Z'def}--\eqref{eq:Z2def}. For instance,
\begin{equation*}
 \mathcal{Z}_1 = \sum_{h_1 |q_1^{\infty}} \frac{\eta(q_1 h_1)}{(q_1 h_1)^{s/2 + u_3}} Z_1(s,u_1, u_2;h_1).
\end{equation*}
Then using \eqref{eq:Z1bound}, we have
\begin{equation*}
\mathcal{Z}_1 \ll q_1^{\varepsilon-\sigma/2} q_1 \ell_1 \sum_{h_1 | q_1^{\infty}} \frac{\lambda_\pi^*(h_1)}{h_1^{\sigma/2}}  (h_1, \ell_1)^{1/2-\theta} \ll q_1^{\varepsilon-\sigma/2} q_1 \ell_1.
\end{equation*}
The case of $\mathcal{Z}_2$ is easy now, because $\mathcal{Z}_2$ is identical to \eqref{eq:Z2def}, since $(h', q_2) = 1$.  Therefore, the bound \eqref{eq:Z2bound} holds for $\mathcal{Z}_2$.  Finally, it is easy to see from 
\eqref{eq:Z'formula} that
\begin{equation*}
\mathcal{Z}_0  = \sum_{(h'', q_2) = 1} \frac{\eta(h'')}{(h'')^{s/2+u_3}} Z_0(s,u_1,u_2,u_3;h'') = 
\frac{\prod_{j=1}^{3}L^{(q)}(s/2+u_j, \overline{\pi} \otimes \eta)  
}{\zeta^{(q)}(s+u_1+u_2)}. \qedhere
\end{equation*}
\end{proof}

\subsection{Properties of $Z$, Eisenstein case}
 Suppose that $\eta$ is a Dirichlet character modulo $ q_2/h_{q_2}$, which we may also identify with a character of $\widehat{\Z}^\times$. For a Hecke character $\mu$, we write $\mu \otimes \eta$ for the twist of $\mu$ by the finite order Hecke character corresponding to $\eta$. Let $Z$ be defined by \eqref{eq:Zdef}, but where $\pi$ is a global principal series/newform Eisenstein series of conductor $m$, central character $\eta^2$, and spectral parameter $it$.  Similarly define $\mathcal{Z}$ for $\pi$ a global principal series/newform Eisenstein series.
\begin{mylemma}
\label{lemma:ZpropertiesEisensteinCase}
 The following properties hold for $Z$ and $\mathcal{Z}$ in the Eisenstein case:
\begin{enumerate}
\item Lemmas \ref{lemma:Zproperties} and \ref{lemma:ZpropertiesSummedoverh'} carry over verbatim.
\item The series $\mathcal{Z}$ has meromorphic continuation to the region of $\mc^4$ with $\text{Re}(s/2+u_i) >1/2$, $i=1,2,3$.   Its only possible poles are at $s/2+u_i \pm it = 1$, which occur
if and only if $\pi \simeq \pi(\mu_1,\mu_2)$  
 with $\mu_i$ unitary Hecke characters of conductors $r_i$, $i=1,2$ such that $\eta = \mu_1 \vert_{\widehat{\Z}^\times} = \mu_2  \vert_{\widehat{\Z}^\times}$. 
 
 In particular, only if $r_1 = r_2$ and $r_1 r_2 = m$, so $m_1 =1$, and $m$ is a square. 
 \item Factoring $\mathcal{Z}= \mathcal{Z}_{\text{good}}\mathcal{Z}_{\text{bad}}$ as in \eqref{eq:Zsummedoverh'Formula},  for $\real(s/2 + u_i) \geq 1 + \varepsilon$, $i=1,2, 3$, we have
\begin{equation}
\label{eq:ZbadBoundEisensteinIndividualh}
  \mathcal{Z}_{\text{bad}}  \ll (q N)^{\varepsilon} 
  q_1^{-1} \ell_1
   \ell_2^{1/2}
  \cond( \eta)^{3/2}.
 \end{equation}
\end{enumerate}

\end{mylemma}

\begin{proof}
Part (1) follows from inspection of the proofs of Lemmas \ref{lemma:Zproperties} and \ref{lemma:ZpropertiesSummedoverh'}.  This also gives the meromorphic continuation of $\mathcal{Z}$ with possible poles only at the poles of $L(s/2+u_i, \overline{\pi} \otimes \eta)$ for $i=1,2,3$.

 Write $\pi  = \pi(\mu_1,\mu_2)$ where $\mu_1, \mu_2$ have conductors $r_1, r_2$, respectively.  
 Therefore, 
 $\mu_1\mu_2= \eta^2$, $r_1 r_2 = m$, and 
 \begin{equation*}
  L(s, \overline{\pi} \otimes \eta) = L(s, \overline{\mu_1} \otimes \eta) L(s, \overline{\mu_2} \otimes \eta).
 \end{equation*}
This $L$-function has a pole if and only if the conductor $\cond(\overline{\mu_1} \otimes \eta) = 1$, equivalently if and only if $\cond(\overline{\mu_2} \otimes \eta)=1$.  In this case, since $\overline{\mu_1}  \otimes \eta$ is unramified at all finite places and $\eta$ is finite order, we have  $\eta = \mu_1\vert_{\widehat{\Z}^\times} = \mu_2\vert_{\widehat{\Z}^\times}$.  In particular, the possible locations of the poles are as stated in the lemma. 
Moreover, since $\eta$ has conductor dividing $q_2/h_{q_2}$ (which is coprime to $q_1$), this implies $m_1 = 1$. Additionally, $r_1 = r_2$, and so $m = r_1^2$.

Finally, to show \eqref{eq:ZbadBoundEisensteinIndividualh}, we simply revisit the estimates of $\mathcal{Z}_1$ and $\mathcal{Z}_2$ from the proofs of Lemmas \ref{lemma:Zproperties} and \ref{lemma:ZpropertiesSummedoverh'}, but now the estimates occur slightly to the right of the $1$-line instead of slightly to the right of the $1/2$-line, and we may additionally use the Ramanujan bound for simplicity.  Indeed, inspection of \eqref{eq:Z1BoundPreSimplified} shows $Z_1 \ll (h_1 q)^{\varepsilon} \ell_1$, and hence $\mathcal{Z}_1 \ll (h_1 q)^{\varepsilon} q_1^{-1} \ell_1$.
Similarly, \eqref{eq:Z2boundPreSimplified1} and \eqref{eq:Z2boundPreSimplified2} lead to $\mathcal{Z}_2 = Z_2 \ll q_2^{\varepsilon} \cond(\eta)^{3/2} \ell_2^{1/2}$. 
\end{proof}

\subsection{Using the spectral bounds}
\label{sec:usingspectral}
\begin{myprop}
\label{prop:SMaassBound}
 With notation as in this section, we have
\begin{equation}
\label{eq:SMaassBoundSummedOverh}
 \sum_{\substack{h \equiv 0 \shortmod{d}}} S_{\rm Maass}(\chi, h)
 \ll
 N (qN)^{\varepsilon}
.
\end{equation}
\end{myprop}
Observe that Proposition \ref{prop:SMaassBound} is consistent with Theorem \ref{thm:ShiftedSumBounds}.

\begin{proof}
We take the expression \eqref{eq:SMaassFormula2}, move the contours of integration so $\real(s) = 1 + \varepsilon$, $\real(u_j) = \varepsilon$, $j=1,2,3$, and apply the triangle inequality.   We use \eqref{eq:H2bound}  and Lemma \ref{lemma:ZpropertiesSummedoverh'} to bound $H_2$ and $\mathcal{Z}$, respectively.  Altogether, we obtain 
\begin{multline}
\label{eq:SMaassSummedOverhUsefulRef2}
 \sum_{\substack{h' \equiv 0 \shortmod{q_1} \\ (h', q_2) = 1}} S_{\rm Maass}(\chi,h_{q_2}h')
  \ll  \frac{h_{q_2}^2 q_1^{-1/2} (qN)^{\varepsilon}}{ q^2\varphi(q_2/h_{q_2})} 
 \sum_{\eta \shortmod{q_2/h_{q_2}}}  \frac{h_{q_2}}{q}
\Big(\frac{q}{h_{q_2}^2}\Big)^{1/2} \frac{Nq}{h_{q_2}}  
 \sum_{t_j \ll (qN)^{\varepsilon}}  \sum_{\ell m = \frac{q}{h_{q_2}}}  \\
 q_1 \ell_1 \ell_2^{1/2}
  \Big(\frac{q_2}{h_{q_2}}\Big)^{3/2}
 \sum_{\pi \in \mathcal{H}_{it_j}(m, \eta^2)}  \lambda_\pi^*\Big(\frac{h_{q_2}^3}{q_2^2}\Big)
\int 
\prod_{i=1}^{3} |L(s/2+u_i, \overline{\pi} \otimes \eta)|ds\,  d \bf{u}
,
\end{multline} 
plus a small error term.
The limits on the integral sign are not displayed; to be definite, the integrals (over $s, u_1, u_2, u_3$) have real parts as fixed above.  Since the function $H_2$ has rapid decay along vertical lines, the integrals may be truncated at $(qN)^{\varepsilon}$ with a very small error term.

Before applying any more advanced tools to bound this expression, we first clean it up with trivial simplifications, giving
\begin{multline}
\label{eq:SMaassLfunctions}
 \sum_{\substack{h' \equiv 0 \shortmod{q_1} \\ (h', q_2) = 1}} S_{\rm Maass}(\chi,h_{q_2}h') \ll 
 \frac{h_{q_2}^{1/2}  N (qN)^{\varepsilon}}{ q^{} } 
 \sum_{\eta \shortmod{q_2/h_{q_2}}} 
 \sum_{t_j \ll (qN)^{\varepsilon}}
 \\ 
 \sum_{\ell m = \frac{q}{h_{q_2}}}   \ell_1 \ell_2^{1/2}
 \sum_{\pi \in \mathcal{H}_{it_j}(m, \eta^2)}  \lambda_\pi^*\Big(\frac{h_{q_2}^3}{q_2^2}\Big)
\int 
\prod_{i=1}^{3} |L(s/2+u_i, \overline{\pi} \otimes \eta)| ds\,  d\bf{u}
,
\end{multline} 
plus a small error term.  Inspired by the method of \cite{BlomerMili}, we apply H\"{o}lder's inequality with exponents $(4,4,4,4)$.
Note that $(\lambda_\pi^*(m))^2 \ll m^{\varepsilon} \lambda_\pi^*(m^2)$, using the Hecke relations.  This gives
\begin{equation*}
 \sum_{\substack{h' \equiv 0 \shortmod{q_1} \\ (h', q_2) = 1}} S_{\rm Maass}(\chi,h_{q_2}h') 
 \ll 
 \frac{h_{q_2}^{1/2}  N (qN)^{\varepsilon}}{ q^{} } 
 A_1^{1/4} A_2^{1/4} A_3^{1/4} A_4^{1/4},
\end{equation*} 
where
\begin{equation*}
 A_1 = \sum_{t_j \ll (qN)^{\varepsilon}} 
 \sum_{\ell m = \frac{q}{h_{q_2}}}   \ell_1  \ell_2^{2}
 \sum_{\eta \shortmod{q_2/h_{q_2}}} 
 \sum_{\pi \in \mathcal{H}_{it_j}(m, \eta^2)}  
 \Big|\lambda_\pi^*\Big(\frac{h_{q_2}^6}{q_2^4}\Big)\Big|^2,
\end{equation*}
and for $i=2,3,4$,
\begin{equation*}
 A_i = \sum_{\eta \shortmod{q_2/h_{q_2}}} 
 \sum_{t_j \ll (qN)^{\varepsilon}} 
 \\
 \sum_{\ell m = \frac{q}{h_{q_2}}}   \ell_1 
 \sum_{\pi \in \mathcal{H}_{it_j}(m, \eta^2)}  
\int |L(s/2+u_i, \overline{\pi} \otimes \eta)|^4 ds\,  du_i.
\end{equation*}
Note that we arranged $\ell_1$ evenly in each $A_i$, but took $(\ell_2^{1/2})^4 =\ell_2^2$ in $A_1$.  

Now we turn to the estimates for each $A_i$, starting with $A_1$.  Our aim is to apply Lemma \ref{lemma:spectralboundOneFourierCoefficient}, but $A_1$ is not quite in the correct form. This can be easily remedied by arranging $$\{\eta \shortmod{ q_2/h_{q_2}}\} = \bigcup_{\substack{\psi \shortmod {q_2/h_{q_2}} \\ \text{even}}} \{\eta: \eta^2 = \psi\}.$$ Note that the interior set on the right hand side is of cardinality $\ll q_2^\eps$, since it is a coset of the subgroup of characters modulo $q_2/h_{q_2}$ of order dividing 2.   With this observation,
Lemma \ref{lemma:spectralboundOneFourierCoefficient} gives 
\begin{equation*}
 A_1 \ll 
(qN)^{\varepsilon} \sum_{\ell_1 m_1 = q_1} \ell_1
 \sum_{\ell_2 m_2 = \frac{q_2}{h_{q_2}}} \ell_2^2
 \Big(m_1 m_2^2 + 
 \Big(\frac{h_{q_2}^3}{q_2^2}\Big) 
 m_2^{1/2}
 \Big),
\end{equation*}
Recalling that $(q_2/h_{q_2})^2 \leq h_{q_2}$, we deduce
\begin{equation*}
 A_1 \ll (qN)^{\varepsilon} q_1 h_{q_2}.
\end{equation*}
For $A_i$, $i=2,3,4$, by Theorem \ref{thm:fourthmomentSpectralBound} we have
\begin{equation*}
 A_i \ll (qN)^{\varepsilon} \frac{q_1 q_2^2}{h_{q_2}^2}.
\end{equation*}
Therefore,
\begin{multline}
\label{eq:SMaassSumOverh'}
\sum_{\substack{h' \equiv 0 \shortmod{q_1} \\ (h', q_2) = 1}} S_{\rm Maass}(\chi,h_{q_2}h') \ll 
 \frac{h_{q_2}^{1/2}  N (qN)^{\varepsilon}}{ q^{} } 
 q_1^{1/4} h_{q_2}^{1/4} \Big(\frac{q_1 q_2^2}{h_{q_2}^2}\Big)^{3/4}   = (qN)^{\varepsilon} N \frac{q_2^{1/2}}{h_{q_2}^{3/4}}.
\end{multline} 

Finally, we insert \eqref{eq:SMaassSumOverh'} into \eqref{eq:SMaassSummedOverhUsefulRef}, giving
\begin{equation*}
 \sum_{\substack{ h \ll H \\ h \equiv 0 \shortmod{d}}} S_{\rm Maass}(\chi,h) \ll (qN)^{\varepsilon} N
 \sum_{\substack{q_1 q_2 = q \\ (q_1, q_2) = 1}} 
\sum_{\substack{h_{q_2} |q_2 \\ v_p(h_{q_2}) < v_p(q_2) \\ d_2 | h_{q_2}}} \frac{q_2^{1/2}}{h_{q_2}^{3/4}}.
\end{equation*}
 The summand is largest when $h_{q_2} = d_2$, and since $q_2^2 \mid d_2^3$, we have $d_2^{3/4} \geq q_2^{1/2}$,  which completes the proof of Proposition \ref{prop:SMaassBound}.
\end{proof}

\begin{myprop}\label{prop:SholBound}
 The bound stated in Proposition \ref{prop:SMaassBound} holds for $S_{{\rm hol}}$.
 \end{myprop}
\begin{proof}
 The proof is nearly identical, changing the obvious things that need  to be changed.  
\end{proof}

\subsection{The Eisenstein series contribution}
The contribution of the Eisenstein series $S_{\rm Eis}$ is largely similar to the Maass form case, with one important difference.  In the Maass case, we related $K_{\rm Maass}$ to a Mellin integral involving twisted $L$-functions, which we shifted slightly to the right of the critical line.  In the Eisenstein case, we will perform the same steps, but there will be a polar contribution arising since the twisted $L$-functions are now products of Dirichlet $L$-functions, which may have a pole when the character is trivial.  We denote by $S_{\rm Pole}$  the contribution from this pole.
 \begin{myprop}\label{prop:SEisBound}
The same bound stated in Proposition \ref{prop:SMaassBound} holds for $S_{\rm Eis} - S_{\rm Pole}$.
 \end{myprop}
 \begin{proof} 
  This is clear from Lemma \ref{lemma:ZpropertiesEisensteinCase}.
 \end{proof}

\begin{myprop}\label{prop:SPoleBound}
 We have
\begin{equation*}
 S_{\rm Pole}(\chi) \ll d^{-1/8} \frac{NH}{q} (Nq)^{\varepsilon} .
\end{equation*}
\end{myprop}
This bound is more than satisfactory for Theorem \ref{thm:ShiftedSumBounds}.
\begin{proof}
The term $S_{\rm Pole}(\chi)$ only arises when $\mathcal{Z}$ has a pole.  By Lemma \ref{lemma:ZpropertiesEisensteinCase}, such  poles arise at $\frac{s}{2}+u_i \pm it =1$ precisely when $\pi \simeq \pi(\mu_1,\mu_2)$ 
 with $\eta = \mu_1 \vert_{\widehat{\Z}^\times} = \mu_2 \vert_{\widehat{\Z}^\times}.$.  In particular, $\cond(\pi) = \cond(\eta)^2$, and $\ell_2 \cond(\eta)^2  = \frac{q_2}{h_{q_2}}$.  The contribution of the polar terms may be estimated by setting the integrals in the Eisenstein analogue of \eqref{eq:SMaassFormula2} to $\real(s)=1+\varepsilon$, and $\real(u_i) = 1/2 + \varepsilon$ for $i=1,2, 3$, and using \eqref{eq:ZbadBoundEisensteinIndividualh} to estimate $\mathcal{Z}_{\text{bad}}$ in lieu of \eqref{eq:ZbadSummedOverh'}.  Writing $\cond(\eta)= v$, and using \eqref{eq:H2bound} we obtain 
\begin{multline*}
\sum_{\substack{h' \equiv 0 \shortmod{q_1} \\ (h', q_2) = 1 }} S_{\rm Pole}(\chi, h_{q_2} h') \ll \frac{h_{q_2}^2}{q^2 \frac{q_2}{h_{q_2}}}
 \sum_{\ell_2 v^2 = \frac{q_2}{h_{q_2}}} 
  \frac{\ell_1}{q_1} \ell_2^{1/2} v^{3/2}
\\ \times \sumstar_{\eta \shortmod{v}} \frac{(q/h_{q_2}^2)^{1/2}}{q/h_{q_2}}   \Big(\frac{N q_2}{h_{q_2}}\Big) 
  \frac{(M_1 M_2 H)^{1/2}}{h_{q_2}^{3/2} } (Nq)^{\varepsilon}.
\end{multline*}
Recalling that $M_1 M_2  = q^2$ (see \eqref{eq:Mdef}) and simplifying gives
\begin{equation*}
 \sum_{\substack{h' \equiv 0 \shortmod{q_1} \\ (h', q_2) = 1 }} S_{\rm Pole}(\chi, h_{q_2} h') \ll \frac{N H^{1/2} h_{q_2}^{1/2} (Nq)^{\varepsilon}}{   q^{3/2} }  
 \sum_{\ell_2 v^2 = \frac{q_2}{h_{q_2}}} \ell_2^{1/2} v^{5/2}.
\end{equation*}
Note $\ell_2^{1/2} v^{5/2} 
\leq (q_2/h_{q_2})^{5/4}$.  Therefore,
\begin{equation*}
 \sum_{\substack{h' \equiv 0 \shortmod{q_1} \\ (h', q_2) = 1 }} S_{\rm Pole}(\chi, h_{q_2} h') \ll \frac{N H^{1/2}   (Nq)^{\varepsilon}}{   q_1^{3/2} q_2^{1/4} h_{q_2}^{3/4} } .
\end{equation*}
Including the outer summations appearing in \eqref{eq:SMaassSummedOverhUsefulRef}, we then deduce
 \begin{equation*}
 S_{\rm Pole}(\chi) \ll 
 \sum_{\substack{q_1 q_2 = q \\ (q_1, q_2) = 1}} 
 \frac{N H^{1/2}   (Nq)^{\varepsilon}}{   q_1^{3/2} q_2^{1/4} d_{2}^{3/4} } = \frac{NH (Nq)^{\varepsilon}}{q}  \frac{q_2^{3/4}}{q_1^{1/2} H^{1/2} d_2^{3/4}}.
\end{equation*}
Using $H \geq q_1 d_{2}$ (since $h \ll H$ and $h \equiv 0 \pmod{q_1 d_2}$) and $q_2 \leq d_{2}^{3/2}$ shows the claimed bound.
 \end{proof}
 
 Applying Propositions \ref{prop:SMaassBound}, \ref{prop:SholBound}, \ref{prop:SEisBound}, \ref{prop:SPoleBound} and \eqref{eq:errortermsum} to the sum over $h \equiv 0 \pmod d$ of \eqref{eq:SchihSpectralDecomposition}, we conclude the proof of Theorem \ref{thm:ShiftedSumBounds}.

\end{document}